\documentclass{amsart}

\usepackage{times}
\usepackage{color}
\usepackage{appendix}
\usepackage{psfrag,graphicx, 
 }
\usepackage{amsfonts,amssymb}
\usepackage{amsmath,amscd,amsxtra,latexsym,mathrsfs}
\usepackage{dsfont}
\usepackage{mathabx}
\usepackage{stmaryrd}

\usepackage[latin1]{inputenc}
\usepackage{mathtools}

\theoremstyle{plain}
\newtheorem{thm}{Theorem}[section]
\newtheorem{lemma}[thm]{Lemma}

\newtheorem{prop}[thm]{Proposition}

\newtheorem{conj}[thm]{Conjecture}

\theoremstyle{definition}
\newtheorem{definition}[thm]{Definition}
\newtheorem{remark}[thm]{Remark}

\newtheorem{assumption}[thm]{Assumption}

\catcode`\@=11
\def\mequal{\mathrel{\mathpalette\@mvereq{\hbox{\sevenrm m}}}} 
\def\@mvereq#1#2{\lower.5\p@\vbox{\baselineskip\z@skip\lineskip1.5\p@
    \ialign{$\m@th#1\hfil##\hfil$\crcr#2\crcr=\crcr}}}
\def\partr#1#2{/\kern-.08333em/_{#1,#2}^{\phantom{.}}}
\def\invpartr#1#2{/\kern-.08333em/_{#1,#2}^{-1}} 
\def\hpartr#1#2{/\kern-.08333em/_{#1,#2}^{h}}
\def\Epartr#1#2{/\kern-.08333em/_{#1,#2}^{E}}

\def\newdot{{\kern.8pt\cdot\kern.8pt}}

\def\,{\relax\ifmmode\mskip\thinmuskip\else\thinspace\fi}
\def\{{\relax\ifmmode\lbrace\else $\lbrace$\fi}
\def\}{\relax\ifmmode\rbrace\else $\rbrace$\fi}

\DeclareSymbolFont{rsfs}{U}{rsfs}{m}{n}
\DeclareSymbolFontAlphabet{\mathscr}{rsfs}
\DeclareFontFamily{U}{rsfs}{}
\DeclareFontShape{U}{rsfs}{m}{n}{%
   <5> <6> rsfs5
   <7> rsfs7
   <8> <9> <10> <10.95> <12> <14.4> <17.28> <20.74> <24.88> rsfs10
}{}

\DeclareFontFamily{U}{msb}{}
\DeclareFontShape{U}{msb}{m}{n}{
  <5> <6> <7> <8> <9> gen * msbm
  <10> <10.95> <12> <14.4> <17.28> <20.74> <24.88> msbm10
  }{}
\DeclareSymbolFont{AMSmsb}{U}{msb}{m}{n} 

\SetSymbolFont{AMSmsb}{bold}{U}{msb}{m}{n}
\DeclareSymbolFontAlphabet{\Bbb}{AMSmsb}

\font\sevenrm=cmr7

\newcommand{\EE}{{\Bbb E}}
\newcommand{\E}{{\Bbb E}}
\newcommand{\NN}{{\Bbb N}}
\newcommand{\RR}{{\Bbb R}}
\newcommand{\R}{{\Bbb R}}
\newcommand{\PP}{{\Bbb P}}

\newcommand{\SC}{{\mathscr C}}
\newcommand{\SD}{{\mathscr D}}

\newcommand{\SF}{{\mathscr F}}
\newcommand{\SL}{{\mathscr L}}

\newcommand{\SP}{{\mathscr P}}

\newcommand{\SU}{{\mathscr U}}

\newcommand{\df}{\coloneqq}

\newcommand{\sign}{\mathrm{sign}}
\newcommand{\bq}{\begin{eqnarray*}}
\newcommand{\bqn}[1]{\begin{eqnarray}\label{#1}}
\newcommand{\eq}{\end{eqnarray*}}
\newcommand{\eqn}{\end{eqnarray}}
\newcommand{\fo}{\forall\ }
\newcommand{\lan}{\lt\langle}

\newcommand{\lVe}{\lt\Vert}
\newcommand{\ran}{\rt\rangle}

\newcommand{\rVe}{\rt\Vert}
\newcommand{\lt}{\left}
\newcommand{\me}{\medskip}
\newcommand{\ri}{\rightarrow}
\newcommand{\rt}{\right}
\newcommand{\st}{\,:\,}
\newcommand{\un}{\mathds{1}}
\def\mathpal#1{\mathop{\mathchoice{\text{\rm #1}}%
   {\text{\rm #1}}{\text{\rm #1}}%
   {\text{\rm #1}}}\nolimits}

\def\id{\mathpal{id}}

\newcommand{\inte}{\mathrm{int}}

\def\di{\displaystyle}
\def\f{\frac}
\def\a{\alpha }
\def\b{\beta }
\def\D{\Delta }
\def\d{\delta }
\def\e{\varepsilon }

\def\g{\gamma }
\def\l{\lambda }
\def\n{\nabla }
\def\Om{\Omega }
\def\om{\omega }

\def\s{\sigma }

\newcommand{\na}{\nabla}
\newcommand{\pa}{\partial}

\newcommand{\usm}{\underline{\smash{\mu}}}
\newcommand{\wi}{\widetilde}
\newcommand{\sm}{\smallskip}

\newcommand{\cB}{\mathcal{B}}
\newcommand{\cC}{\mathcal{C}}
\newcommand{\cD}{\mathcal{D}}

\newcommand{\cF}{\mathcal{F}}

\newcommand{\cK}{\mathcal{K}}
\newcommand{\cL}{\mathcal{L}}

\newcommand{\cM}{\mathcal{M}}

\newcommand{\cR}{\mathcal{R}}

\newcommand{\cV}{\mathcal{V}}

\newcommand{\iy}{\infty}
\newcommand{\fA}{\mathfrak{A}}

\newcommand{\fG}{\mathfrak{G}}

\newcommand{\fF}{\mathfrak{F}}

\newcommand{\ff}{\mathfrak{f}}
\newcommand{\fg}{\mathfrak{g}}

\newcommand{\fs}{\mathfrak{s}}

\newcommand{\ZZ}{\mathbb{Z}}
\newcommand{\lin}{\llbracket}
\newcommand{\rin}{\rrbracket}

\DeclareMathOperator*{\esssup}{ess\,sup}

\renewcommand{\theequation}{\arabic{section}.\arabic{equation}}

\newcommand{\barD}{D}

\begin{document}
\title[Couplings of Brownian motions with set-valued dual processes]{Couplings of Brownian motions with set-valued dual processes on Riemannian manifolds}
\author[M. Arnaudon]{Marc Arnaudon} \address{Univ. Bordeaux, CNRS, Bordeaux INP,\hfill\break\indent Institut de Math\'ematiques de Bordeaux, UMR 5251, F. 33405, Talence, France} \email{marc.arnaudon@math.u-bordeaux.fr}
\author[K.A. Coulibaly-Pasquier]{Kol\'eh\`e Coulibaly-Pasquier} \address{Institut \'Elie Cartan de Lorraine, UMR 7502\hfill\break\indent Universit\'e de Lorraine and CNRS} \email{kolehe.coulibaly@univ-lorraine.fr}
\author[L. Miclo]{Laurent Miclo} \address{{
Institut de Math\'ematiques de Toulouse, UMR 5219, \hfill\break\indent
Toulouse School of Economics, UMR 5314,\hfill\break\indent
 CNRS and Universit\'e de Toulouse}} \email{{laurent.miclo@math.cnrs.fr}}
 \thanks{Funding from the  grant ANR-17-EURE-0010 is aknowledged by L.M}
\date{\today\ \emph{ File: }\jobname.tex}
\maketitle

\begin{abstract}
The purpose of this paper is to construct a Brownian motion $X \df (X_t)_{t\geq 0}$ taking values in a Riemannian manifold $M$, together with a compact valued process $D\df (D_t)_{t\geq 0}$ such that, at least for small enough $\SF^D$-stopping time $\tau> 0$ and conditioned by $\SF_\tau^D$, the law of $X_\tau$ is the normalized Lebesgue measure on $D_\tau$. This intertwining result is a generalization of Pitman theorem. We  first construct regular intertwined processes related to Stokes' theorem. Then using several limiting procedures we  construct synchronous intertwined, free intertwined, mirror intertwined processes. The local times of the Brownian motion on the (morphological) skeleton or the boundary of~$D$ plays an important role. Several examples with moving intervals, discs, annulus, symmetric convex sets are investigated. \par\smallskip\noindent
\end{abstract}

\section{Introduction and main results}
\label{Section0}
\setcounter{equation}0

Markov intertwinings were introduced by Rogers and Pitman \cite{MR624684} to give a direct proof of the famous relation between the Brownian motion and the Bessel-3 process due to Pitman \cite{MR0375485}.
These relations were next used by Yor and his coauthors (see e.g.\ \cite{Yor_intertwinings,MR1654531}) to get identities in law and by Diaconis and Fill \cite{MR1071805} to construct strong stationary times.
For a historical account of the subsequent development of the Markov intertwining  technique, consult for instance {Pal} and {Shkolnikov} \cite{2013arXiv1306.0857P}.\par
At an algebraic level, a \textbf{Markov intertwining relation} is a (directed) weak similar relation, from a Markov semi-group $(\bar P_t)_{t\geq 0}$ on a measurable state space $(\bar M,\bar \cM)$ to another Markov semi-group $(P_t)_{t\geq 0}$ on a measurable state space $(M,\cM)$, consisting of a Markov kernel (called the \textbf{link}) $\Lambda$ from $(\bar M,\bar \cM)$ to $(M,\cM)$ such that
\bqn{inter}
\fo t\geq 0,\qquad \bar P_t\Lambda&=&\Lambda P_t\eqn
in the sense of the composition of Markov kernels. 
Depending on non-degeneracy properties of $\Lambda$, such a relation is more or less strong. Especially when Markov semi-groups are described by their generators,  \eqref{inter}
is often replaced by
\bqn{inter2}
\bar L \Lambda&=&\Lambda L\eqn
where $\bar L$ and $L$ are respectively the generators of  $(\bar P_t)_{t\geq 0}$ and  $(P_t)_{t\geq 0}$. But then one has to be more careful with the meaning of generators (e.g.\ in the sense of martingale problems) and their domains,
in particular the domains are transported via \eqref{inter2}.\par
To be more useful from a probabilist point of view, it is convenient to convert \eqref{inter2} into a coupling between $\bar X\df(\bar X_t)_{t\geq 0}$ and $X\df(X_t)_{t\geq 0}$, two Markov processes respectively associated to $\bar L$ and $L$ (called the \textbf{dual} and \textbf{primal processes}), so that the following relations hold for the conditional laws:
\bqn{cou1}
\fo t\geq 0,\qquad \cL(X_t\vert \bar X_{[0,t]})&=&\Lambda(\bar X_t,\cdot)\eqn
\par
In addition, one asks that $(\bar X_t)_{t\geq 0}$ can be constructed from $(X_t)_{t\geq 0}$ in an adapted way, meaning
\bqn{cou2}
\fo t\geq 0,\qquad\cL(\bar X_{[0,t]}\vert X)&=&\cL(\bar X_{[0,t]}\vert X_{[0,t]})\eqn\par
 Yor was wondering about  such couplings  between some piecewise linear Markov processes and squared Bessel processes, in order to simplify his approach to certain properties of the former processes similar to those of the latter, see the end of the introduction of \cite{Yor_intertwinings}.\par
Such couplings are crucial for the constructions of strong stationary times, as explained by Diaconis and Fill \cite{MR1071805} in a discrete time and finite setting.
More precisely, in this situation $X$ is an ergodic Markov chain with invariant probability $\pi$ and $\bar X$ is a Markov chain absorbed in a unique point.
A strong stationary time $\tau$ for $X$ is a finite stopping time for $X$ (and some independent randomness) such that $\tau$ and $X_\tau$ are independent and $X_\tau$ is distributed according to $\pi$. Taking into account \eqref{cou1} and \eqref{cou2}, one can see that the absorption time for $\bar X$ is a strong stationary time for $X$. 
\par
Strong stationary times are important for two reasons (cf.\ Diaconis and Fill \cite{MR1071805}):\par
- They enable to sample \textit{exactly} the invariant probability $\pi$, contrary to the usual approximations provided by Monte Carlo techniques.
\par
- They provide a  probabilistic alternative to functional analysis approaches for the quantitative investigation of convergence to equilibrium. More precisely, for any strong stationary time $\tau$, we have
\bq
\fo t\geq 0,\qquad \fs(\cL(X_t),\pi)&\leq & \PP[\tau>t]\eq
where  the \textbf{separation discrepancy} $\fs(\mu,\pi)$ between two probability measures $\mu$ and $\pi$  is defined by
\bq
\fs(\mu,\pi)&\df&\esssup_{\pi}\lt(1-\f{d\mu}{d\pi}\rt)\eq
(where $d\mu/d\pi$ is the Radon-Nikodym density). The separation discrepancy dominates the total variation norm and gives positivity properties of $\mu$ with respect to $\pi$.
In the context of convergence to equilibrium, it is very difficult to estimate the discrepancy of  $\fs(\cL(X_t),\pi)$ via functional inequality techniques (see e.g.\ the book \cite{MR3155209} of Bakry, Gentil and Ledoux).
\par
In the objective of constructing strong stationary times via intertwining duality, there are particular dual processes $\bar X$ which are taking values in $\cV$, the set  of measurable subsets of $M$, but in general $\bar V$ is only a subset of $\cV$, consisting in some regular subsets.
The absorption set is the whole set $M$. The heuristic goal of intertwining duality is then to construct random subsets $\bar X_t\subset V$ such that $X_t$ is already at equilibrium in $\bar X_t$, for all $t\geq 0$, in such a way that $\bar X$ is itself Markovian and ends up covering the whole state space $M$.
\par
In the diffusion context, set-valued intertwining dual processes started to be constructed in Fill and Lyzinski \cite{MR3571247} and \cite{MR3634282}.
In  \cite{zbMATH07470497}, set-valued  dual processes for diffusions on Riemannian manifolds were identified as stochastic perturbations of mean-curvature flows.
But the coupling of primal and dual processes were not considered in \cite{zbMATH07470497} and this is our present goal, mainly for Brownian motions on Riemannian manifolds.
As we will see, there are numerous ways to  construct such couplings (this is true in  more general contexts, see \cite{Mic2020} for the diversity of such couplings in a finite framework), but none of them is immediate and they are related to fine geometric features of the evolving subsets, such as their skeletons.
We are thus to consider synchronous intertwined, free intertwined, mirror set-valued intertwined dual processes. \par
The reader must be warned that, as it stands now in the context of multidimensional diffusions, the set-valued dual processes are not defined up to the absorption time (except in symmetric settings), and as a consequence the same will be true for our couplings, which will be defined only up to some positive stopping times.
We hope to investigate this point in future works, to end the construction of  strong stationary times for Brownian motion on compact Riemannian manifolds, which remains our remote motivation. Other motivations for the couplings of primal and dual processes in the context of diffusions can be found in Machida \cite{2019arXiv190807559M} and \cite{Mic2020}.
\par\me

Let us now present more precise definitions. Here the state space  $M$ is a $d$-dimensional complete Riemannian manifold.  Denote respectively by $\rho$, $\mu$ and $\usm$, the Riemannian distance,  the Lebesgue measure on $M$ and the corresponding $(d-1)$-Hausdorff measure. The main objective of this paper is to construct couplings of primal diffusions processes with their set-valued dual intertwined processes. This will partially solve  Conjecture~6 in~\cite{zbMATH07470497} in the case of Brownian motion $(X_t)_{t\ge 0}$ and stochastic modified mean curvature flow $(D_t)_{t\ge 0}$ (which were generically denoted $(\bar X_t)_{t\ge 0}$ above). This conjecture says that an intertwined construction in the sense of Definition~\ref{D6.0} is always possible. 

\begin{definition}
\label{D6.0}
Consider a Markov process $D=(D_t)_{t\in[0,\tau]}$, with values in compact subsets of $M$ and continuous with respect to the Hausdorff topology,
and where $\tau$ is an a.s.\ positive stopping time in the filtration $\SF^D$ of $D$, serving as a lifetime for $D$.
We say that a Brownian motion $X=(X_t)_{t\geq 0}$ in $M$ and $D$ are intertwined when for all bounded  $\SF^D$-stopping time $\tau'$ smaller than $\tau$, {conditioned on} $\SF_{\tau'}^D$, $X_{\tau'}$ has uniform law in $D_{\tau'}$ (and in particular $X_{\tau'}\in {D}_{\tau'}$). 
More generally, for any $\SF^D$-stopping time $\wi\tau$ smaller than $\tau$,
we say that $X$ and $D$ are $\wi\tau$-intertwined when $X$ and $(D_t)_{t\in[0,\wi\tau]}$  are intertwined.
\end{definition}
This is a generic definition, below stronger topologies on subsets of $M$ will be considered. Note that the above lifetime is not necessary the explosion time, i.e.\   the exit time from all compact sets for the considered topology. In the infinite dimensional  state space of $D$, compactness does not seem an appropriate notion.
\par\sm
Our main results are Theorems \ref{C6.1},  \ref{T4.1} and \ref{T7.1} presenting such joint constructions of the primal Brownian motion $(X_t)_{t\geq 0}$ and the dual domain-valued $(D_t)_{t\geq 0}$ processes.
The coupling of Theorem \ref{C6.1} consists in the infinite-dimensional stochastic differential equation \eqref{6.5}, based on a function $f\st(x,D)\mapsto f(x,D)$ which is a deformation of the signed distance from $x\in M$ to the boundary of the domain $D$ (see Assumption \eqref{A6.1} for the precise requirements).
Theorem \ref{T4.1} is obtained by specifying some approximating  functions $f$. Given the trajectory $(X_t)_{t\geq 0}$ of the Brownian motion, we construct the domain evolution 
$(D_t)_{t\geq 0}$ using the local time of $(X_t)_{t\geq 0}$ on the skeletons of $(D_t)_{t\geq 0}$ and the mean curvatures of the normal foliations of these domains (see  \eqref{4.10}). Other approximating functions $f$ lead to Theorem~\ref{T7.1}, where the prominent role is played by the local time at the boundary.
This situation is in some sense opposite to the previous one, since
the driving Brownian motion of $(D_t)_{t\geq 0}$  is now independent from $(X_t)_{t\geq 0}$, while it is as correlated as it can be in Theorem \ref{T4.1}.
These theoretical results are illustrated by the fundamental examples of Section \ref{Section5}. First we recover the intertwining relation between the real Brownian motion and the three-dimensional Bessel process. Next we deal with rotationally symmetric manifolds.
Finally we present the application of our results to  symmetric convex domains in the plane, even if the detailed proofs are deferred to a forthcoming paper.
 \par
 To come back to our initial motivation, 
assume that $X$ and $D$ are intertwined, where the lifetime $\tau$ is the hitting/covering time by $D$ of the whole state space $M$. 
If  furthermore $\tau$ is finite (typically true when $M$ is compact), then the Riemannian measure can be normalized into a probability (called the \textbf{uniform distribution}, which is invariant and reversible for  the Brownian motion $X$) and $\tau$ is a strong stationary time for $X$. In this situation, the tail distributions of $\tau$ provide quantitative estimates for the speed of convergence of the Brownian motion toward equilibrium, in the separation sense. These estimates will need geometric ingredients such as Ricci bounds and it will be interesting to see how they will enter the game.

{ 
The needs for couplings between primal and dual processes of a Markovian intertwining relation is illustrated by \cite{ACM3:22}, where strong stationary times $\tau_n$ are constructed for the $n$-dimensional  sphere (when the subset-valued dual is starting from a singleton), satisfying
$$
\EE[\tau_n]\sim \f{\ln(n)}n
$$
and for any $r>0$,
$$
\lim_{n\to\infty}\PP\left[\tau_n>(1+r)\f{\ln(n)}n\right]=\lim_{n\to\infty}\PP\left[\tau_n<(1-r)\f{\ln(n)}n\right]=0.
$$
}

\section{Intertwined dual processes: existence in connection with Stoke's formula}
\label{Section6}
\setcounter{equation}0

In this section we make a construction of intertwined processes $X$ and $D$ based on the Stokes' Formula~\eqref{6.1} below. Consider a  compact domain $D$ in $M$ with $C^2$ boundary. Let $f : \barD \to \RR$ a $C^2$ function such that $\n f|_{\partial D}=N^D$ the normal inward vector on boundary. Then by Stoke's formula, for any $C^2$ function $g : \barD  \to \RR$, 
\begin{equation}
 \label{6.1}
 -\int_{\partial D}g d\usm =\int_{\partial D} g\langle \n f, -N^D\rangle \,d\usm =\int_Dg\Delta f\, d\mu+\int_D\langle \n g,\n f\rangle \, d\mu. 
\end{equation}

For $\a\in(0,1)$, denote by $\SD^{2+\a}$ the set of  compact connected  subsets $D$ of $M$ with $C^{2+\a}$ boundary. 
It will be more convenient to work with this state space (endowed with its natural topology) than with the larger one considered in Definition \ref{D6.0}.
Let us even restrict it further:

We fix a point $o\in M$ for convenience.
\begin{definition}
\label{D6.1}
For a given $\a\in(0,1)$, $\e>0$, 
we denote by $\cF^{\alpha,\e}$ the set of $D\in \SD^{2+\a}$ such that 
\begin{itemize}
\item $D\subset B(o,1/\e)$ the Riemannian ball centered at $o$ with radius $1/\e$;
 \item $\rho(\partial D, S(D))\geq  \e$,
where $S=S(D)$ is the skeleton of $D$ (see appendix~\ref{Section2} for details);
\item $\rho(\partial D, S^{\rm out}(D))\geq  \e$,
where $S^{\rm out}(D)$ is the outer skeleton of $D$, i.e. the skeleton of $(\barD )^c$.
\item  the coefficients of the $\alpha$-H\"olderianity {  of the second fundamental form} of $\pa D$ are bounded by $1/\epsilon$.
\end{itemize}
The set $\cF^{\alpha,\e}$ will serve as the state space of the set-valued process $(\wi D_t)_{t\in[0,{\tau_{\e}}]}$  
and ${\tau_{\e}}\in(0,+\iy]$ will be the exiting time from $\cF^{\alpha,\e}$.
This process will be a diffusion, i.e.\ a Markov process with continuous trajectories (for the topology inherited from $\SD^{2+\a}$),
and its  generator $\wi \SL$ will be defined later in~\eqref{6.7}.
We extend the trajectory $(\wi D_t)_{t\in[0,{\tau_{\e}}]}$ by taking $\wi D_t=\wi D_{\tau_{\e}}$ for any $t> {\tau_{\e}}$.
It amounts to imposing that $\wi \SL$ vanishes outside $\cF^{\alpha,\e}$.
It is possible to define in the same way $(\wi D_t)_{t\in[0,{\tau})}$ on $\cD^{2+\a}$ (which coincides with $\cup_{\e>0}\cF^{\a,\e}$), where $\tau$ is the exiting time from $\cD^{2+\a}$.
But it will be more convenient for us to work with a process with an infinite lifetime (to be able to apply Proposition \ref{pro1} in Appendix \ref{appD}) and whose set of values has a boundary which is well-separated from the skeleton.
\end{definition}

 \par
Let $\b\in \{0,\a\}$. 
For $D_0 \in   \cD^{2+\b}$ and $\d>0$ small enough, a $\d$-neighborhood of $D_0$ is defined as follow:
$$ \mathcal{V}_\d^{2+\b}(D_0) := \left\{ \inte(\exp_{\partial D_0} (f))   , f \in C^{2+ \b} (\partial D_0)  , \Vert f  \Vert_{C^{2 +\b}(\partial D_0)} < \d  \right\}, $$
where for  $f \in  C^{2+ \b} (\partial D_0)$
$$\exp_{\partial D_0} (f) := \left\{ \exp_{x}(f(x)N^{D_0}(x)) , x \in  \partial D_0 \right\} $$
($\exp$ being the exponential map in $M$), 
  and $\inte(\exp_{\partial D_0}(f))$ is the interior of the hypersurface $\exp_{\partial D_0}(f), $ oriented by the orientation of $D_0$. Let $\eta(\partial D_0)>0$ be the radius of the maximal tubular neighborhood of $\partial D_0$. 
Notice that $\d < \eta(\partial D_0)$
garantees that 
all elements of $\mathcal{V}_\d^{2+\b}(D_0)$ are regular deformations of $D_0$. Also notice that all elements $D$ of $\cF^{\alpha,\e}$ have $\eta(\partial D)\ge \e$.
 
 We identify two domains $ D_1 , D_2 \in \mathcal{V}_\d^{2+\b}(D_0) $ with the functions $ f_1, f_2 \in C^{2+ \b} (\partial D_0) $ such that 
 $ D_1 = \inte \{ \exp_{\partial D_0} (f_1)\}$ and $ D_2 = \inte \{ \exp_{\partial D_0} (f_2)\}$  and we define a local distance
\begin{equation}\label{dD}
  d_{\b,D_0}(D_1,D_2) := \Vert f_1 -f_2  \Vert_{C^{2+\b}(\partial D_0)}  . 
\end{equation}
\begin{assumption}
 \label{A6.1}
$ $
\begin{itemize}
\item The function 
\begin{align*}
f :  M\times \cF^{\alpha,\e}&\to \RR\\
 (x,D)&\mapsto f(x,D)=f^D(x)
\end{align*}
 is a $C^{2+\a}$ function in the two variables ({the differential in $D$ is in the sense of  Fr\'echet} with respect to the above local Banach structure {  defined by the distances $d_{\a, D}$}). 
The functions  $f^D$ satisfy 
\begin{equation}
\label{6.2}
\lVe \n f^D\rVe_{\infty}\leq  1,
\end{equation}
and coincide with the signed  distance to the boundary $\rho_{\partial D}^+$ (positive inside $D$ and negative outside) in a neighbourhood of $\partial D$. 
The functions $f^D$ have bounded Hessian, uniformly in $D\in\cF^{\a,\e}$.
Furthermore,  we assume that the coefficients of the $\alpha$-H\"olderianity of ${\rm Hess} f^D$ are   uniformly bounded over $\cF^{\a,\e}$.
\item 
There exists a positive integer  $m$ and a $C^1$ map
\begin{align*}
 \s_c : M\times \cF^{\alpha,\e}&\to \Gamma(TM\otimes (\RR^m)^\ast)\\
 (x,D)&\mapsto \s_c(x,D)=\s_c^D(x)\in L(\RR^m,T_xM)
\end{align*}
where $\Gamma(TM\otimes (\RR^m)^\ast)$ is the set of sections over $M$ of $TM\otimes (\RR^m)^\ast$ and 
$L(\RR^m,T_xM)$ is the set of linear maps from $\RR^m$ to $T_xM$,
such that the linear map 
\begin{equation}
\label{6.4}
\begin{split}
\s^D(x) : \RR\times \RR^m&\to T_xM \\
(w_0, w)&\mapsto\ w_0\n f^D(x)+ \sigma_{\mathrm c}^D(x)(w)
\end{split}
\end{equation}
satisfies 
\begin{equation}
\label{6.4.3}
\forall x\in \barD , \ \ \s^D(\s^D)^\ast (x)={\rm Id}_{T_xM}.
\end{equation}
\end{itemize}
\end{assumption}
\par\hfill $\square$\par\sm
\begin{remark}
The first condition of Assumption~\ref{A6.1} implies that  
\begin{equation}
 \label{6.2bis}
 \begin{split}
 \n f^D|_{\partial D}&=({\n \rho^{+}_{\partial D})}|_{\partial D}(=N^{D}) \quad\hbox{and}\\  \Delta f^D|_{\partial D}&=(\Delta {\rho^{+}_{\partial D})}|_{\partial D}(=-h^{D}).
 \end{split}
\end{equation}
where $h^D$ stands for the mean curvature on $\pa D$.
It also implies that the functions $f^D$ are uniformly Lipschitz and have uniformly bounded Laplacian. Also, for fixed $x\in \partial D$, varying $D$ successively  along a field $K$ normal to the boundary $\partial D$ and along $N^{D}$ for the second derivative: 
\begin{equation}
 \label{6.2ter}
 \begin{split}
 \langle \n f(x, \cdot ), K\rangle(x)&=-\langle N^{D}(x), K(x)\rangle \quad\hbox{and}\\  \n d f (x, \cdot)\left(N^{D},N^{D}\right)&=0
 \end{split}
\end{equation}
where $\n d f (x, \cdot)$ is the Hessian of $f$ in the second variable.

The second condition  of Assumption~\ref{A6.1}  implies that for all $u\in T_xM$, 
\begin{equation}
 \label{6.4.1}
 \|u\|^2=\langle u,\n f^D(x)\rangle^2+ \sum_{i=1}^m\langle u,\s^D_{\mathrm c}(x)(e_i)\rangle^2
\end{equation}
for $e_1,\ldots ,e_m$ an orthonormal basis of $\RR^m$. In particular, if $x\in \partial D$, taking $u=\n f^D(x)=N^D(x)$, we get since $\|N^D(x)\|=1$:
\begin{equation}
 \label{6.4.2}
 0=\langle \n f^D(x),\s(x)(e_i)\rangle, \quad i=1,\ldots m.
\end{equation}
\end{remark}
\begin{prop}
\label{P6.3}
Assumption~\ref{A6.1} can always be realized, with any $\a\in(0,1)$ and $\e>0$. 
\end{prop}
\begin{proof}
We begin with remarking that for $D\in \cF^{\alpha,\e}$, $\rho(\partial D, S(D))\ge \e$. In particular, the distance to $\partial D$ is $C^{2+\a}$ on $D_\e:=\{x\in M,\ \rho(x,\partial D)< \e\}$. Let $h_\e$ be an odd smooth nondecreasing function from $\R$ to $\R_+$ such that $h_\e(r)=r$ for $r\in [0,\e/2]$, $h_\e(r)=(3/4)\e$ for $r\ge \e$ and $\|h_\e'\|_\infty\le 1$. Then $f^D:=h_\e\circ {\rho^{+}_{\partial D}}$ satisfies all the requirements of the first condition of Assumption~\ref{A6.1}. Then for constructing $\s_c^D$ we proceed as in~\cite{Arnaudon-Li:17}, Proposition~3.2 taking $\s_1=\n f^D$. The wanted regularity in $D$ is easily checked. 
\end{proof}

Let $W_t$ and $W_t^m$ two independent Brownian motions with values respectively in $\RR$ and $\RR^m$. 

  The equation we are interested in writes in It\^o form for all $y\in \partial D_t$: 
\begin{equation}
 \label{6.5}
 \lt\{\begin{array}{rcl}
dX_t&=&\left(\n f^{D_t}(X_t) \, dW_t+ \sigma_{\mathrm c}^{D_t}(X_t)\, dW_t^m\right)\\[2mm]
 d\partial D_t(y)&=&N^{D_t}(y)\left(dW_t
+\left(\f12 h^{D_t}(y)+\Delta f^{D_t}(X_t)\right)\,dt\right)
\end{array}\rt.
\end{equation}
started at
a  compact domain $D_0$ with $C^{2+\a}$ boundary and $X_0$ such that  $\SL(X_0)=\SU(D_0)$, where $\SU(D_0)$ is the uniform probability measure on $D_0$.
The notation $d\partial D_t(y)$ stands for an infinitesimal move of the boundary $\pa D_t$ at point $y$ and is rigorously presented in Appendix \ref{Section3}, see \eqref{3.6}.
In fact, as in Definition~\ref{D6.1},  the evolution equation \eqref{6.5} is implicitly considered only up to the exit time $\tau_\e$ of $\cF^{\a,\e}$ for some fixed $\a\in(0,1),\  \e>0$, after which the process
is assumed not to move.
\par
In \eqref{6.5}, the processes $(D_t)_{t\geq 0}$
and $(X_t)_{t\geq 0}$ are fully interacting, since the evolution of one of them depends on the other one.
In particular, they are not Markovian by themselves in general.
\par
Another subset-valued  process $(\wi D_t)_{t\geq 0}$ will be interesting for our purposes. It is solution to the evolution equation
\bqn{6.5mod0}
\nonumber
\lefteqn{\hskip-20mm\fo t\leq \wi\tau_\epsilon,\,  \forall y\in \partial \wi D_t,}\\
d\partial \wi D_t(y)&=&N^{\wi D_t}(y)\left(d\wi W_t+\left(\f12 h^{\wi D_t}(y)-\f{\usm^{\partial\wi D_t}(\partial\wi D_t)}{\mu(\wi D_t)}\right)\,dt\right), \quad\eqn
where $\wi W_t$ is a real-valued Brownian motion and where $\wi\tau_\epsilon$ is the exit time from $\cF^{\alpha,\epsilon}$.
 \par
 Notice that the equation for $\wi D_t$ does no longer depend of $X_t$, so if the solution is unique, $(\wi D_t)_{t\geq 0}$ will be Markovian.
It is Equation (44) in \cite{zbMATH07470497} (up to a time scaling by 2).  Theorem~40 of~\cite{zbMATH07470497} (where (44) has been rewritten as (79)) proves local existence of a solution.
 \begin{thm}\label{6.5mod0hyp}
Fix $\a\in (0,1)$ and $\e>0$. Then  \eqref{6.5mod0} admits a unique 
 global solution. In particular the process $(\wi D_t)_{t\geq 0}$ is Markovian.
  \end{thm}
\begin{proof}
The proof is a consequence of Theorem~22 in~\cite{zbMATH07470497}. It can be found in Appendix~\ref{appC}.
\end{proof}
 \par
To describe the generator $\wi\SL$ of $(\wi D_t)_{t\geq 0}$   we must introduce the following notations.
For any smooth function $k$ on $M$, consider the mapping $F_k$ on $\cD^{2+\alpha}$ by
\bq
\fo D\in \cD^{2+\alpha},\qquad F_k(D)&\df& \int_D k\, d\mu\eq
For any $k,g\in \cC^\iy(M)$ and any $D\in \cD^{2+\alpha}$, define
\bqn{6.7}
\wi \SL[F_k](D)&\df&\usm ^{\partial D}(k)\f{\usm^{\partial D}(\partial D)}{\mu(D)}-\f12 \usm^{\partial D}(\langle \n k, N^D\rangle )\\
\label{6.7b}\Gamma_{\wi\SL}[F_k, F_g](D)&\df& \int_{\pa D} k\, d\usm\int_{\pa D} g\, d\usm
\eqn\par
Next consider $\fA$ the algebra consisting of the 
functionals of the form
$\fF\df\ff(F_{k_1}, ..., F_{k_n})$, where $n\in\ZZ_+$, $k_1, ..., k_n\in \cC^\iy(M)$ and $\ff\st \cR\ri \RR$  is a $\cC^\iy$ mapping, with $\cR$ an open subset of $\RR^n$ containing the image of $\cD^{2+\alpha}$ by $(F_{k_1}, ..., F_{k_n})$.
 For  such a functional $\fF$, define
\bqn{fLfF}
\wi\SL[\fF]
&\df&
\sum_{l=1}^n \pa_j\ff(F_{k_1}, ..., F_{k_n}) \wi\SL[F_{k_l}]\\&&
\nonumber+\sum_{j,l\in\lin 1,n\rin}^n \pa_{j,l}\ff(F_{k_1}, ..., F_{k_n})\Gamma_{\wi\SL} [F_{k_j},F_{k_l}]\eqn
To two elements of $\fA$, $\fF\df\ff(F_{k_1}, ..., F_{k_n})$ and $\fG\df\fg(F_{g_1}, ..., F_{g_m})$, we also associate
\bqn{GfF}
\Gamma_{\wi\SL} [\fF,\fG]&\df&\!\!\!\!\sum_{l\in\lin n\rin, j\in\lin m\rin} \pa_l\ff(F_{k_1}, ..., F_{k_n})\pa_j\fg(F_{g_1}, ..., F_{g_m})\Gamma_{\wi\SL} [F_{k_l},F_{g_j}]
\eqn
\par
\begin{remark}
To see that the above definitions are non-ambiguous, since a priori they could depend on the writing of $\fF\in\fA$ under the form 
$\ff(F_{k_1}, ..., F_{k_n})$ and similarly for $\fG$, see Remark 2 of \cite{zbMATH07470497}. More generally, the forms of \eqref{fLfF} and \eqref{GfF}
are consequences of  the diffusion feature of $\wi\SL$, for more on the subject, see e.g.\ the book of Bakry, Gentil and Ledoux \cite{MR3155209}.
\end{remark}
\par
\begin{remark}\label{rm-stop}
In the above considerations, $\wi\SL$ was defined on $\cD^{2+\a}$, but from now on, $\wi\SL$ will stand for the restriction of this generator
to $\cF^{\a,\e}$ and will be zero on $\cD^{2+\a}\setminus\cF^{\a,\e}$, in accordance with Definition \ref{D6.1}. Similarly, all stochastic differential equations
will be valid only up to the stopping time $\tau_\e$ (which was defined after Definition \ref{D6.1}) or $\wi\tau_\e$ (defined after \eqref{6.5mod0}).
\end{remark}
\par
The interest of Assumption  \ref{A6.1} comes from the following result:
\begin{thm}
 \label{C6.1}
 Let $(x,D)\mapsto f^D(x)$ and  $(x,D)\mapsto \sigma_{\mathrm c}^D(x)$ satisfy Assumption~\ref{A6.1}. Then 
 equation~\eqref{6.5} has a solution $\di \left(X_{t}, D_t\right)_{t\geq 0}$ started at $D_0\in \SF^{\a, \e}$, $X_0\sim \SU(D_0)$. Moreover
 the processes  $\di \left(X_{t}\right)_{t\geq 0}$ and $\di \left(D_{t}\right)_{t\geq 0}$ are $\tau_\e$-intertwined.
\end{thm}
\begin{proof}

We prove here the existence of solution to equation~\eqref{6.5}. The intertwining will be a consequence of Proposition~\ref{P6.2} below.

We begin to prove the existence of a diffusion with modified drift, and then we will get the result by change of probability. The modified equation writes
\begin{equation}
 \label{6.5mod}
 \lt\{
\begin{array}{rcl}
 d\partial D_t(y)&=&N^{D_t}(y)\left(d\widehat W_t+\left(\f12 h^{D_t}(y)-\f{\usm^{\partial D_t}(\partial D_t)}{\mu(D_t)}\right)\,dt\right);\\
dX_t&=&\Biggl(\n f^{D_t}(X_t) \,\left[ d\widehat W_t-\left(\f{\usm^{\partial D_t}(\partial D_t)}{\mu(D_t)}+\Delta f^{D_t}(X_t)\right)\, dt\right]\\&&+ \sigma_{\mathrm c}^{D_t}(X_t)\, dW_t^m\Biggr)
\end{array}\rt.
\end{equation}
for $\widehat W_t$ and $W_t^m$ independent Brownian motions. 
Notice that the first equation is the same as \eqref{6.5mod0}. Thus due to  Theorem \ref{6.5mod0hyp}, $(D_t)_{t\geq 0}$ is a diffusion process with generator 
$\wi\SL$.
Then given $D_t$, the equation for $X_t$ 
\begin{equation}
 \label{6.10}
 \begin{split}
 dX_t&=\Biggl(\n f^{D_t}(X_t) \,\left[ d\widehat W_t-\left(\f{\usm^{\partial D_t}(\partial D_t)}{\mu(D_t)}+\Delta f^{D_t}(X_t)\right)\, dt\right]\\&+ \sigma_{\mathrm c}^{D_t}(X_t)\, dW_t^m\Biggr)
 \end{split}
\end{equation}
can also be solved, since the coefficients in front of $d\widehat W_t$ and $dW_t^m$ are Lipschitz, 
\linebreak
$\s^D(\s^D)^\ast(x)={\rm Id}_{T_xM}$ and $\D f^D$ is bounded and uniformly H\"older continuous (due to Assumption \ref{A6.1}).
 Notice that $X_t$ remains in $D_t$, since when $X_t\in \partial D_t$, we have, using~\eqref{6.4.2} which yields on boundary $\langle N^{D_t}(X_t), \sigma_{\mathrm c}^{D_t}(X_t) dW_t^m\rangle=0$, 
 \begin{equation}\label{6.6}
 \begin{split}
 &d(\rho^+_{\partial D_t}(X_t))\\&=\langle \n\rho^+_{\partial D_t},dX_t\rangle -\f{1}2h^{{D_t}}(X_t)\,dt-\langle d\partial D_t(X_t), N^{D_t}(X_t)\rangle\\
 &=\langle N^{D_t}(X_t),dX_t\rangle -\f{1}2h^{{D_t}}(X_t)\,dt-\langle d\partial D_t(X_t), N^{D_t}(X_t)\rangle=0.
 \end{split}
 \end{equation}
where we used~\eqref{6.5mod} and~\eqref{6.2bis}.  We also {have no covariation since} the martingale part of $d\partial D_t$ acts on the normal flow only, and any normal flow $$r\mapsto D(r):=\{x\in M,\ \rho^+(x)\ge r\}$$ satisfies $\rho^+_{\partial D(r)}(x)=\rho^+_{\partial D(0)}(x)-r$ for $x\in D(0)$ and $|r|$ small, (see Appendix \ref{Section2}).

Once we have a solution to~\eqref{6.5mod}, make by Girsanov theorem a change of probability such that $(W_t,W_t^m)$ is a Brownian motion where
\begin{equation}
 \label{6.11}
 W_t:=\widehat W_t-\int_0^t\left(\f{\usm^{\partial D_s}(\partial D_s)}{\mu(D_s)}+\Delta f^{D_s}(X_s)\right)\, ds.
\end{equation}
We get a solution to~\eqref{6.5} in the new probability.
\end{proof}

\begin{prop}
 \label{P4.5} 
 Let $D_t$ satisfy
 \begin{equation}
  \label{4.9.0}
  d\partial D_t(y)=N^{D_t}(y)\left(dW_t+\left(\f12 h^{D_t}(y)+b_t\right)\,dt\right),\quad \forall y\in \partial D_t
 \end{equation}
 for some Brownian motion $W_t$ and some adapted locally bounded real-valued process $b_t$.
  Let $\mu_t=\mu^{D_t}$ be the Lebesgue measure on $D_t$ and $\di \bar\mu_t=\bar \mu^{D_t}=\SU(D_{t})=\f{\mu^{D_t}}{\mu(D_t)}$. Denote by $\usm_t=\usm^{\partial D_t}$ the Lebesgue measure on $\partial D_t$ and $\di \bar{\usm}_t=\bar{\usm}^{\partial D_t}=\f{\usm^{\partial D_t}}{\mu(D_t)}$.
 Let $k$ be a smooth function of $M$. Then 
 \begin{equation}
  \label{4.9.1}
  d\mu_t(k)=-{\usm}_t(k)\,dW_t-\f12 \left(2b_t\usm_t(k)+{\usm}_t\left(\langle dk,N^{D_t}\rangle \right)\right)\,dt
 \end{equation}
 and
 \begin{equation}
  \label{4.9.2}
  \begin{split}
  d\bar\mu_t(k)&=\left(-\bar{\usm}_t(k)+\bar\mu_t(k)\bar{\usm}_t(\partial D_t)\right)\,dW_t-\f12 \bar{\usm}_t\left(\langle dk,N^{D_t}\rangle \right)\,dt\\
  &+\left(\bar{\usm}_t(\partial D_t)+b_t\right)\left(-\bar{\usm}_t(k)+\bar\mu_t(k)\bar{\usm}_t(\partial D_t)\right)\,dt
  \end{split}
 \end{equation}
 In particular, if $b_t=-\bar{\usm}_t(\partial D_t)$ we get
 \begin{equation}
  \label{4.9}
  d\bar\mu_t(k)=\left(-\bar{\usm}_t(k)+\bar\mu_t(k)\bar{\usm}_t(\partial D_t)\right)\,dW_t-\f12 \bar{\usm}_t\left(\langle dk,N^{D_t}\rangle \right)\,dt.
 \end{equation}
\end{prop}
\begin{proof}
Let us first work at fixed time $t\geq 0$. Denote $D=D_t$ and adopt the corresponding notations presented in Appendix \ref{Section2}.
 For $k$ a smooth function on $M$ and $r\in \RR$ sufficiently close to $0$ so that $\partial D(r)$ (defined in ~\eqref{2.3} and~\eqref{2.3.1}) is a smooth manifold without boundary, let 
\begin{equation}
 \label{4.11}
 F(r,k)=\int_{D(r)}k\,d\mu.
 \end{equation}
 We have 
 \begin{equation}\label{4.12}
 F(r,k)=\int_{\partial D}\left(\int_r^{\tau(y)}k\left(\psi(s)(y)\right)e^{-\int_0^sh^{{D}}(\psi(u)(y))\,du }{\, ds}\right)\,\usm(dy)
\end{equation}
with $\tau(y)$ the hitting time of $S(D)$ by the inward normal flow started at $y$ (defined in~\eqref{2.1}) and  $\psi(s)(y)\df\psi(0,s)(y)\df\exp_y(sN_y)$ defined in~\eqref{2.4}. 
The mapping $h^D$ is defined in \eqref{2.5} and is an extension of the mean curvature on the boundary $\pa D$: it corresponds to the mean curvature for the foliation induced by the $\pa D(r)$, $r\in\RR$ sufficiently small.
With this formulation we can differentiate with respect to $r$, to obtain
\begin{equation}
 \label{4.13}
 F'(r,k)=-\int_{\partial D} k(\psi(r,y))e^{-\int_0^rh^{{D}}(\psi(s)(y)\,ds}\,\usm(dy).
\end{equation}
Differentiating again we get 
\begin{equation}
 \label{4.14}
 F''(r,k)=-\int_{\partial D}\left( \langle dk,\partial_r\psi(r,y)\rangle - (kh)(\psi(r,y))\right)e^{-\int_0^rh^{{D}}(\psi(s)(y))\,ds}\,\usm(dy).
\end{equation}
In particular,
 \begin{equation}
  \label{4.15}
  F'(0,k)=-\usm(k)\quad\hbox{and}\quad F''(0,k)=\usm(kh-\langle dk,N\rangle).
 \end{equation}
 This allows us to compute 
 \begin{equation}
  \label{4.16}
  d(F(W_t,k))=F'(W_t,k)\,dW_t+\f12 F''(W_t,k)\, dt
 \end{equation}
 and then, since $dW_t$ and $\langle d\partial D_t, N^{D_t}\rangle(\cdot)$ differ only by a finite variation process
 \begin{equation}
  \label{4.17}
  d\mu_t(k)=\int_{\partial D_t}-{k}(y)\langle d\partial D_t(y),N^{D_t}(y)\rangle +\f12 \left(kh^{{D_t}}-\langle dk,N^{D_t}\rangle\right)(y)\,\usm_{{t}}(dy).
 \end{equation}
This yields
\begin{equation}
  \label{4.18}
  d\mu_t(k)=\int_{\partial D_t}{k}(y)\left(-dW_t-b_t\,dt\right)-\f12 \langle dk,N^{D_t}\rangle(y)\,\usm_{{t}}(dy)\,dt,
 \end{equation}
 which gives~\eqref{4.9.1}.
In particular, taking $k \equiv 1$ we obtain
\begin{equation}
\label{4.19}
d\mu(D_t)=\usm_t(\partial D_t)\left(-dW_t -b_t\,dt\right).
\end{equation}
Now we can compute 
\begin{align*}
&d\bar\mu_t(k)\\&=d\left(\f{\mu_t(k)}{\mu(D_t)}\right)\\
&{=\f1{\mu(D_t)}d\mu_t(k)-\f{\mu_t(k)}{\mu(D_t)^2}d\mu(D_t)+\f{\mu_t(k)}{\mu(D_t)^3}d\lan\mu(D_\cdot)\ran_t
-\f1{\mu(D_t)^2}d\lan \mu_\cdot(k),\mu(D_\cdot)\ran_t}\\
&=\f1{\mu(D_t)}d\mu_t(k)-\f{\mu_t(k)}{\mu(D_t)^2}d\mu(D_t)+\f{\mu_t(k)}{\mu(D_t)^3}\usm(\partial D_t)^2dt-\f{1}{\mu(D_t)^2}\usm_t(k)\usm_t(\partial D_t)dt\\
&=-\bar{\usm}_t(k)\left(dW_t+b_t\,dt\right)-\f12\bar{\usm}_t(\langle dk,N^{D_t}\rangle )\,dt+\bar\mu_t(k)\bar{\usm}_t(\partial D_t)(dW_t+b_t\,dt)\\&+\bar\mu_t(k)\bar{\usm}_t(\partial D_t)^2\,dt
-\bar{\usm}_t(k)\bar{\usm}_t(\partial D_t)\,dt.
\end{align*}
This yields~\eqref{4.9.2}.
\end{proof}
\par
Denote $\tau_\e$ the exiting time of $(D_t)_{t\geq 0}$ from $\cF^{\a,\e}$.
As in Definition \ref{D6.1}, we stop $(X_t,D_t)_{t\geq 0}$ at $\tau_\e$.
\begin{prop}
\label{P6.1}
Any solution of equation~\eqref{6.5} stopped at $\tau_\e$ is a Markov process   solution to a martingale problem associated to a generator $\SL$ acting in the following way: for any $g,k$ smooth functions on $M$ and 
\begin{equation}
\label{6.13}
F_k (D):=\int_Dkd\mu,
\end{equation}
we have for $(x,D)\in M\times\cF^{\a,\e}$,
\begin{equation}
\label{6.14}
\begin{split}
\SL(gF_k)(x,D)&=-g(x)\Delta f^D(x)\usm^{\partial D}(k)-\f12 g(x)\usm^{\partial D}(\langle \n k, N^D\rangle)+\f12 F_k(D)\Delta g(x)\\
&-\usm ^{\partial D}(k)\langle\n  g, \n f^D\rangle (x).
\end{split}
\end{equation}
\end{prop}
\begin{proof}
From~\eqref{6.5} and~\eqref{4.9.1} with $b_t=\D f^{D_t}(X_t)$ we have 
\begin{equation}
\label{6.15}
\begin{split}
dF_k(D_t)=&-\usm^{\partial D_t}(k)\left(\,dW_t +\D f^{D_t}(X_t)\,dt\right){ -}\f12\usm^{\partial D_t}\left(\langle \n k,N^{D_t}\rangle \right)\,dt.
\end{split}
\end{equation}
This implies that 
\begin{equation}
\label{6.16}
\SL(F_k)(x,D)=-\usm^{\partial D}(k)\D f^D(x)-\f12\usm^{\partial D}\left(\langle \n k,N^{D}\rangle \right),
\end{equation}
and the covariation of $g(X_t)$ and $F_k(D_t)$ is $\Gamma_\SL[g,F_k](X_t,D_t)\,dt$ with 
\begin{equation}
\label{6.17}
\Gamma_\SL[g,F_k](x,D)=-\usm^{\partial D}(k)\langle \n g,\n f^D\rangle (x).
\end{equation}
Consequently, using
\begin{equation}
\label{6.18}
\SL(gF_k)(x,D)= g(x)\SL(F_k)(x,D)+F_k(D)\f12\D g(x)+\Gamma_\SL[g,F_k](x,D)
\end{equation}
we get~\eqref{6.14}.
\end{proof}
\par
It is possible to extend the description of $\SL$ to more general functions on $M\times\cF^{\a,\e}$ (it vanishes on its complementary set), by replacing $F_k$ in \eqref{6.14} by a mapping $\fF$ from $\fA$, as presented before Theorem \ref{C6.1}.\par
Let $(\SP_t)_{t\geq 0}$ be the Markovian semi-group  associated to the processes $(X_t,D_t)_{t\geq 0}$ solution to \eqref{6.5} stopped at $\tau_\e$.
This semi-group is associated to $\SL$ in the weak sense of martingale problems, as described in Appendix \ref{appD}.
\par
Let $(\wi D_t)_{t\geq 0}$ be a diffusion process with generator $ \wi\SL$ stopped outside $\cF^{\a,\e}$,
started at $\wi D_0=D_0$ (due to  Theorem \ref{6.5mod0hyp}, this process can be obtained as a solution to the evolution equation \eqref{6.5mod0}),
$\wi \nu_t$ its law at time $t$ and let
\begin{equation}
\label{6.8}
\nu_t(dD, dx):=\wi \nu_t(dD)\SU(D)(dx).
\end{equation}

\begin{prop}
\label{P6.2}
We have for all smooth functions $g,k$ on $M$:
\begin{equation}
\label{6.15.1}
\partial_t\nu_t(gF_k)=\nu_t(\SL (gF_k)).
\end{equation}
As a consequence, if $(D_0,X_0)$ has law $\nu_0$ then for all $t\ge 0$, the solution $(D_t,X_t)$ to equation~\eqref{6.5} has law $\nu_t$, implying that $(X_t)_{t\ge 0}$ and $(D_t)_{t\ge 0}$ are $\tau_\e$-intertwined. Moreover $D_t$ is a diffusion with generator $\wi \SL$.
\end{prop}
\begin{proof}
Integrating~\eqref{6.14} in $x$ with respect to the uniform law $\bar\mu^D:=\SU(D)$ in $D$ yields 
\begin{equation}
\label{6.19}
-\bar\mu^D\left(g\D f^D\right)\usm^{\partial D}(k)-\f12 \bar\mu^D(g)\usm^{\partial D}(\langle \n k,N^D\rangle)+\f12 F_k(D)\bar\mu^D(\D g)-\usm^{\partial D}(k)\bar\mu^D(\langle \n g,\n f^D\rangle ).
\end{equation}
By Stokes theorem, 
\begin{equation}
\label{6.20}
\bar\mu^D\left(g\D f^D+\langle \n g,\n f^D\rangle\right)=\bar\usm ^{\partial D}\left(g\langle \n f^D, -N^D\rangle \right)=-\underline{\bar \mu}^{\partial D}(g),
\end{equation}
so the expression~\eqref{6.19} writes
\begin{equation}
\label{6.21}
H(D):=\usm^{\partial D}(k)\underline{\bar \mu}^{\partial D}(g)-\f12\bar\mu^D(g)\usm^{\partial D}(\langle \n k,N^D\rangle )+\f12F_k(D)\bar \mu^D(\D g)
\end{equation}
On the other hand 
\begin{equation}
\label{6.22}
\nu_t(gF_k)=\wi\nu_t[\bar \mu^{D_t}[g] F_k]
\end{equation}
which implies that 
\begin{equation}
\label{6.23}
\partial_t\nu_t(gF_k)=\partial_t\wi\nu_t(\left(\bar \mu^{D_t}(g) F_k\right)=\wi\nu_t\left(\wi\SL\left(\bar \mu^{D_t}(g) F_k\right)\right).
\end{equation}
By~\eqref{4.9}, 
\begin{equation}
\label{6.24}
\wi\SL\left(\bar \mu^{D_t}(g)\right)=-\f12\bar\usm ^{\partial D_t}(\langle \n g,N^{D_t}\rangle),
\end{equation}
so, taking into account \eqref{6.7b},
\begin{align*}
&\wi\SL\left(\bar \mu^{D_t}(g)F_k\right)\\&=\bar \mu^{D_t}(g)\wi \SL(F_k)+F_k\wi\SL\left(\bar \mu^{D_t}(g)\right)+\Gamma_{\wi\SL}\left[\bar \mu^{D_t}(g), F_k\right]
\\
&=\bar\mu^{D_t}(g)\Big\{\usm^{\partial D_t}(k)\bar\usm ^{\partial D_t}(\partial D_t)-\f12 \usm^{\partial D_t}(\langle \n k, N^{D_t}\rangle)\Big\}-\f12 \mu^{D_t}(k)\bar\usm ^{\partial D_t}(\langle \n g,N^{D_t}\rangle )\\
&-\left(-\bar\usm ^{\partial D_t}(g)+\bar\mu^{D_t}(g)\bar\usm ^{\partial D_t}(\partial D_t)\right)\usm^{\pa D_t}(k)\\
&=-\f12\bar\mu^{D_t}(g)\usm^{\partial D_t}(\langle \n k,N^{D_t}\rangle)-\f12 \mu^{D_t}(k)\bar\usm^{\pa D_t}(\langle \n g, N^{D_t}\rangle )+\bar\usm ^{\partial D_t}(g)\usm^{\partial D_t}(k).
\end{align*}
But  $\di \bar\mu^{D_t}(\D g)=-\bar\usm ^{\partial D_t}(\langle\n g,N^{D_t}\rangle)$ and $F_k(D_t)=\mu^{D_t}(k)$, so
\begin{equation}
\label{6.25}
H(D_t)=\wi\SL\left(\bar \mu^{D_t}(g)F_k\right),
\end{equation}
which together with~\eqref{6.23} proves~\eqref{6.15.1}. 
\par
Let us now prove that for any $t\geq 0$,  $\SP_t$ transports $\nu_0$ into $\nu_t$, where $(\SP_t)_{t\geq 0}$ is the semi-group introduced after the proof of Proposition \ref{P6.1}. Consider the map 
\begin{equation}
\label{6.26}
G(g,k,t)(s)=\nu_s\left(\SP_{t-s}(gF_k)\right), \quad s\in[0,t].
\end{equation}
We compute
\begin{equation}
\label{6.27}
\begin{split}
G(g,k,t)'(s)&=(\partial_s\nu_s)\left(\SP_{t-s}(gF_k)\right)-\nu_s\left(\partial_t\SP_{t-s}(gF_k)\right)\\
&=\nu_s\left(\SL \SP_{t-s}(gF_k)\right)-\nu_s\left(\SL \SP_{t-s}(gF_k)\right)=0
\end{split}
\end{equation}
where we used Proposition \ref{pro1} in Appendix \ref{appD}  to justify the differentiations (as well as the fact that $\SL \SP_{t-s}(gF_k)=\SP_{t-s}\SL (gF_k)$
is bounded to be able to use differentiation under the integral $\nu_s$).
So we get $ G(g,k,t)(0)=G(g,k,t)(t)$ which rewrites as
\begin{equation}
\label{6.28}
\nu_0\SP_t(gF_k)=\nu_t(gF_k),
\end{equation}\par
More generally, by similar arguments, we can replace in this formula $F_k$ by any mapping $\fF$ from $\fA$.
This in turn implies that $\nu_0\SP_t=\nu_t$.
\par
To finish, by iteration, we see that if $X_0\sim \bar\mu^{D_0}$ then $(D_t)_{t\ge0}$ has the same finite time marginals as $(\wi D_t)_{t\ge0}$, proving that $(D_t)$ is a diffusion with generator $\wi \SL$.
\end{proof}

\section{Intertwined dual processes: a  generalized Pitman theorem }
\label{Section4}
\setcounter{equation}0

In this section we will consider the case where $f^D$ is the distance to boundary. It is not covered by Section~\ref{Section6} since distance to boundary is not smooth, it is singular on the skeleton of $D$. We will make an  approximation of it, and then go to the limit in law.

Let $\wi W_t$ be a real-valued Brownian motion and $\wi D_t$ be the solution of~\eqref{6.5mod0} started at $\wi D_0$, with driving Brownian motion $\wi W_t$.
\begin{assumption}
\label{A4.0}
Fix $\a\in (0,1)$ and $\e>0$.  There exists a closed bounded subset $\wi \cF^{\a,\e}$ of $\cF^{\a,\e}$ in which the process $(\wi D_t)_{t\ge 0}$ a.s.\ takes its values, such that  the map $D\mapsto S(D)$ is {  continuous} from $\wi \cF^{\a,\e}$ with the $C^{2}$ metric to $\cK(M)$, the set of compact subsets of $M$ endowed with the Hausdorff metric. Moreover {  Brownian motions with probability one never hit the singular part of $S(\tilde D_t)$}.
\end{assumption}
\begin{conj}
\label{C4.0}
We conjecture that Assumption~\ref{A4.0} is always realized, for any $\a\in (0,1)$, $\e>0$, $\wi D_0 \in \cF^{\a,\e}$.
\end{conj}
{  Notice that Theorem 1.1 in \cite{Albano:16} proves the first part of the conjecture, i.e. the continuity of $D\mapsto S(D)$, in the case where $M=\R^d$ endowed with a possibly varying Riemannian metric. }
All examples together with the study of the motion of the skeleton in Appendix~\ref{Section3} make us believe that Conjecture~\ref{C4.0} is true. However a better knowledge of skeletons is necessary to solve it. {  We believe that the process  $(S(\tilde D_t))_{t\ge 0}$ takes its values in a set of regular stratified spaces, and that it has absolutely continuous variation in this space}.

Let us begin with some preparatory results. 
To describe the approximation of $\rho(x,\partial D)$ we are interested in, let us introduce some notations.\par
$\bullet$ 
Let $(x,D)\mapsto \ell_\e(x,D)\df(h_\e\circ \rho_{\partial D})(x)$ where  $h_\e\equiv 1$ in $[0,\e/2]$, $h_\e\equiv 0$ in $[3\e/4,\infty)$ and $h_\e$ is smooth and nonincreasing in $[0,\infty)$. When $D$ is fixed by the context, we will denote $\ell_\e(x)\df\ell_\e(x,D)$. \par
$\bullet$ 
For any  $\d\in (0,\e)$,  let $\varphi_\d:\RR_+\to \RR$ be a nonnegative function with support in $[0,\delta]$, such that
the mapping $\RR^d\ni u\mapsto \varphi_\d(\vert u\vert)$ is smooth and $\int_{\RR^d} \varphi_\d(\vert u\vert)\, du=1$ (in the sequel, $\vert \cdot\vert$ will stand for the usual Euclidean norm or for the Riemannian norm on any tangent space of $M$, depending on the context) .
\par
$\bullet$ 
Let $g_\d$ be a smooth, $1$-Lipschitz and odd function defined on $\RR$, with $g_\d(r)=r$ on $[0,\e/4]$,  $0\leq g_\d(r)\leq r$ for any $r\geq 0$,   and $g_\d(r)=c_\delta r$ on $[3\e/8,\infty)$, for an appropriate constant $c_\delta\leq 1$ very close to~$1$ that will be defined below in \eqref{cdelta}. We write $\rho_\d(x,\partial D)\df g_\d(\rho(x,\partial D))$.\par
The approximation of $\rho(x,\partial D)$ we choose is
\begin{equation}
\label{5.4.5}
\begin{split}
f_\d(x,D)&=\ell_\e(x,D)\rho_\d(x,\partial D)+(1-\ell_\e(x,D))\int_{T_xM}\varphi_\d(\vert v\vert)\rho_\d(\exp_x(v),\partial D)\, dv
\end{split}
\end{equation}
(where $dv$ stands for the Lebesgue measure on $T_xM$).
\par
Define
\bq
e(\d)&\df&\sup\{\vvvert (\nabla\exp)(u)]\vvvert,\ x\in B(o,1/\e),\ u\in B_x(0,\d)\subset T_xM\}
\eq
where $\nabla\exp(u) : T_xM\to T_{\exp_x(u)}M$ is the covariant derivative of $\exp$ with respect to the base point, $\vvvert\cdot\vvvert$ is the operator norm, when $T_{x}M$ and $T_{\exp_x(u)}M$ are endowed with their Euclidean structures, and $B_x(0,\d)$ is the open ball in $T_xM$ with center $0$ and radius $\delta$. Recall that  $\e$ is fixed as in Assumption \ref{A4.0}.
The previously mentioned constant $c_\delta$ is given by
\bqn{cdelta}
c_\delta&\df&e^{-1}(\d)\left(1-\d \|\n_1\ell_\e\|_\infty\right)\eqn\par
Notice that $c_\d$ does not depend on $D$ and is as close as we want to $1$.\par
More precisely, we have
\begin{lemma}
There exists  two constants $C_1',C_1''>0$, depending only on $\e$, such that for $\delta>0$ sufficiently small,
\bq
0\ \le\  e(\d)- 1&\le&C_1'\d\\
\vert c_\d-1\vert&\leq & C_1''\d
\eq
\end{lemma}
\begin{proof}
The inequalities of the first line are well-known properties of the exponential mapping.
The second bound follows, since $\|\n_1\ell_\e\|_\infty=\|h'_\e\|_\infty$ is independent of $D$ (and of order $1/\e$).   
\end{proof}
From the second bound, we can and will assume that the function $g_\d$  is furthermore chosen so that $g_\d(r)$ converges uniformly to $r$ on compact sets of $\RR_+$, as well as the corresponding derivatives up to order~$2$ as $\d\searrow 0$. 
In addition, we choose $\d>0$ sufficiently small so that the map $(x,y)\mapsto \exp_x^{-1}(y)$ is well-defined and smooth in the $\d$-neighborhood the diagonal of $B(o,1/\e)\times B(o,1/\e)$. 
Then, for any $x\in M$, we can rewrite \eqref{5.4.5} under the form
\begin{equation}
\label{5.4.5b}
\begin{split}
f_\d(x,D)
&= \ell_\e(x,D)\rho_\d(x,\partial D)\\&+(1-\ell_\e(x,D))\int_{M}\varphi_\d(\vert \exp_x^{-1}(y)\vert )\rho_\d(y,\partial D)\, J\exp_x^{-1}(y) dy,
\end{split}
\end{equation}
where $J\exp_x^{-1}$ is the absolute value of the determinant of the Jacobian of $\exp_x^{-1}(\cdot)$.
\par
The interest of all these preparations is:
\begin{prop}
\label{A4.0bis}
For all $\d>0$ sufficiently small, the function $(x,D)\mapsto f_\d(x,D):= f_\d^D(x)$ has the following properties
\begin{itemize}
 \item $f_\d$ satisfies the conditions of Assumption~\ref{A6.1};
 \item there exists $C_1>0$ such that $\forall D\in \wi\cF^{\a,\e}$ and $x\in D$, we have 
\begin{equation}
\label{Approxdist0}
|f_\d( x,D)-\rho(x,\partial D)|\le C_1\d;
\end{equation}
 \item the differential and the Hessian of $f_\d$ with respect to the second variable  $D$ satisfy $\forall D\in \wi \cF^{\a,\e}$, $\forall x\in D\backslash S(D)$, for all vector fields $K$ normal to $\partial D$:
\begin{equation}
\label{Approxdist}
\left\langle d_2 f_\d(x, D), K\right\rangle \le C_4\|K\|_\infty \quad \hbox{and}\quad \left\| \n_2 d_2 f_\d(x,D)\left(N_{\partial D},N_{\partial D}\right)\right\|\le C_4
\end{equation}
for a $C_4$ not depending on $x, D,\d$. The second term is the second derivative along the inward normal flow on $D$.
\end{itemize}
\end{prop}
\begin{proof}
 We first prove
$\|d_1 f_\d(x,D)\|\le 1$, $d_1$ denoting the differential with respect to the first or the $x$ variable. 
For $x\in B(o,1/\e)$ 
we have 
\begin{equation}
\label{5.4.6}
\begin{split}
d_1f_\d(x,D)=&\ell_\e(x,D)d_1\rho_\d(x,\partial D)\\&+(1-\ell_\e(x,D))d_1\left(\int_{T_xM}\varphi_\d(\vert u\vert )\rho_\d(\exp_x(u),\partial D)\, du\right)\\
&+d_1\ell_\e(x,D)\int_{T_xM}\varphi_\d(\vert u\vert )\left(\rho_\d(x,\partial D)-\rho_\d(\exp_x(u),\partial D)\right)\, du.
\end{split}
\end{equation}
Notice that if $x'$ is close to $x$ and $\imath_{x,x'} :T_{x}M\to T_{x'}M$ is the parallel transport along the minimal geodesic from $x$ to $x'$, then 
$$
\int_{T_{x'}M}\varphi_\d(\vert u\vert )\rho_\d(\exp_{x'}(u),\partial D)\, du=\int_{T_{x}M}\varphi_\d(\vert u\vert )\rho_\d(\exp_{x'}(\imath_{x,x'}(u),\partial D)\, du.
$$
Taking the differential with respect to $x'$ at $x'=x$ and using $\n_{x'}\vert_{x'=x}\imath_{x,x'}=0$ by definition of parallel transport yields
$$
d_1\left(\int_{T_{x}M}\varphi_\d(\vert u\vert )\rho_\d(\exp_{x}(u),\partial D)\, du\right)=\int_{T_{x}M}\varphi_\d(\vert u\vert )d_1\rho_\d((\n\exp)(u),\partial D)\, du.
$$
If $\rho(x,\partial D)\le \e/2$ then $\ell_\e(x,D)=1$,  $\n\ell_\e(x,D)=0$ and 
\begin{align*}
\|d_1f_\d(x,D)\|&\le \ell_\e(x,D)\|d_1\rho_\d(x,\partial D)\|\le 1.
\end{align*}
If $\rho(x,\partial D)\ge \e/2$ then for $\d\le \e/8$, we have, for $u\in T_xM$ with $\vert u\vert \leq \delta$, $\rho(\exp_x(u),\pa D)\geq 3\e/8$. It follows 
\begin{align*}
\|d_1f_\d(x,D)\|\le&\ell_\e(x)e^{-1}(\d)\left(1-\d \|d_1\ell_\e\|_\infty\right)\\&+(1-\ell_\e(x))\int_{T_xM}\varphi_\d(\vert u\vert )c_d\|(\n\exp)(u)\|\, du\\
&+\|d_1\ell_\e(x)\|_\iy\int_{T_xM}\varphi_\d(\vert u\vert )\d\, du\\&\le 1.
\end{align*}
It is easily checked that the function $f_\d$ satisfies the other properties of Assumption~\ref{A6.1}. Let us check that it also satisfies~\eqref{Approxdist0}.
\par
We have 
\begin{equation}\label{5.4.6.1}
f_\d(x,D)-\rho_\d(x,\partial D)=(1-\ell_\e(x,D))\int_{T_xM}\varphi_\d(\vert u\vert )\left(\rho_\d(\exp_x(u),\partial D)-\rho_\d(x,\partial D)\right)\, du
\end{equation}
which implies 
$$
\left|f_\d(x,D)-\rho_\d(x,\partial D)\right|\le \d.
$$
On the other hand
$$
\vert \rho(x,\partial D)-\rho_\d(x,\partial D)\vert\le
(1-c_\delta) \max \lt(\f2\epsilon , \f{3\epsilon}8\rt)
\le C_1'''\d
$$
for some constant $C_1'''>0$ (depending on $\e$).
This yields~\eqref{Approxdist0} with $C_1\df 1+C_1'''$.

For proving~\eqref{Approxdist}, we take a vector field $K(y)=k(y)N(y)$, $y\in\partial D$ and compute
\begin{equation}
\label{5.4.7}
\left\langle d_2\rho(x,\partial D),K\right\rangle =\left\langle -N(P(x)), K(P(x))\right\rangle=-k(P(x))
\end{equation}
where $P(x)$ is the projection of $x$ onto $\partial D$, and 
\begin{equation}
\label{5.4.9}
\begin{split}
 \n_2 d_2 \rho(x, \partial D)\left(N_{\partial D}, N_{\partial D}\right)=0.
\end{split}
\end{equation}
Remarking that $\|d_2\ell_\e(x,D)\|$ is bounded by $\|h_\e'\|_\infty$, we get~\eqref{Approxdist} via a straightforward computation.
\end{proof}

\begin{thm}
\label{T4.1}
Fix $D_0=\tilde D_0\in  \wi\cF^{\a,\e}$ and let $X_0\sim\SU(D_0)$.
 Under Assumption~\ref{A4.0}, 
there exists a pair $(X_t,D_t)_{t\ge 0}$ of $\tau_\e$ intertwined processes in the sense of Definition~\ref{D6.0}, such that  the process $(D_t)_{t\ge 0}$ satisfies
\begin{equation}
 \label{4.0}
 \begin{split}
  d\partial D_t(y) 
  & =N^{D_t}(y) \Biggl( \left  \langle dX_t, N^{D_t}(X_t)\right\rangle +\left(\f12 h^{D_t}(y)-h^{D_t}(X_t){{   \un_{D_t  \backslash S_t}(X_t) }} \right)\,dt \\
 & -2\sin(\theta^{S_t}(X_t))\, dL_t^{S_t}(X)\Biggr)
  \end{split}
\end{equation}
Here  $\di \theta^{S_t}(x)=\pi/2-\varphi^{S_t}(x)$, $\varphi^{S_t}(x)$ being the angle between  the orthogonal line to $S_t$ at $x$ and any of the two minimal geodesics from $\partial D_t$ to $x\in S_t$ (recall $S_t$ is the \textit{regular} skeleton of $D_t$, see Appendix \ref{Section2}). In other words $\theta^{S_t}(x)$ is the smallest angle between $S_t$ and the geodesics. 
  The process $L^{S_t}$ is the local time of $X_t$ at $S_t:=S(D_t)$:
\begin{equation}
 \label{4.2}
 L_t^{S_t}(X)=\lim_{\b\searrow 0}\f1{2\b}\int_0^t 1_{\{X_s\in S_s^\b\}}\,ds,
\end{equation}
$S_s^\b$ being the thickening of the regular part of $S_s$ in normal direction, of thickness $\b$ in both directions. 
\end{thm}
\begin{remark}
Compared to Section~\ref{Section6} with $f^D$ replaced by distance to boundary $\rho_{\partial D}$, we have 
outside the skeleton  $S^D$
\begin{equation}
\label{4.1.1}
\n \rho_{\partial D}(x)=N^D(x)\quad\hbox{and}\quad \D\rho_{\partial D}(x)=-h^D(x)
\end{equation}
and we will see that on the moving skeleton $S_t=S^{D_t}$:
\begin{equation}
\label{4.1.2}
``\D\rho_{\partial D_t}(X_t)\, dt"=-2\sin(\theta^{S_t}(X_t))\, dL_t^{S_t}(X).
\end{equation}

\end{remark}

\begin{proof}
£Under Assumption~\ref{A4.0},
Proposition~\ref{A4.0bis} allows us to construct for
 each $\d>0$, intertwined processes $(X_t^\d,D_t^\d)_{t\ge 0}$ started at $(X_0^\d,D_0^\d)=(X_0,D_0)$, associated with the functions $f_\d^D$, stopped at $\tau_\e^\d$, the exit time from $\wi \cF^{\alpha,\e}$. 
We have from Equation~\eqref{6.5}
 \begin{equation}\label{4.1.4}
 \begin{split}
 d\partial D^\d_t(y)&=N^{D^\d_t}(y)\left(dW_t^\d
+\left(\f12 h^{D^\d_t}(y)+\Delta f_{\d}^{D^\d_t}(X_t^\d)\right)\,dt\right)
\end{split}
\end{equation}
for some Brownian motion $W_t^\d$. On the other hand,
from Proposition~\ref{P6.2} and~\eqref{6.1},
 \begin{equation}\label{4.1.3}
 (\wi D_t^\d)_{t\ge 0}:=(D_t^\d)_{t\ge 0}
 \end{equation}
 satisfies equation~\eqref{6.5mod0}:
\begin{equation}\label{4.1.5}
 \begin{split}
d\partial D^\d_t(y)&= N^{D_t^\d}(y)\left( d\wi{W}^\d_t
+\left( \f12 h^{D^\d_t}(y) - \f{\usm^{\partial D^\d_t}(\partial D^\d_t)}{\mu(D^\d_t)} \right)  \,dt \right) \\
\end{split}
\end{equation}
where  $\wi{W}_t^\d$ is the $ \SF^{D^\d}_t-$Brownian motion
\begin{equation}\label{4.1.6} d\wi{W}^\d_t=dW_t^\d +\Delta f_{\d}^{D^\d_t}(X_t)\,dt + \f{\usm^{\partial D^\d_t}(\partial D^\d_t)}{\mu(D^\d_t)} \,dt.
 \end{equation}

 A remarkable fact about all $(X_t^\d,D_t^\d)_{t\ge 0}$ is that their marginals are constant in law. Notice that also $((D_t^\d)_{t\ge 0},\tau_\e^\d)$ is constant in law since $\tau_\e^\d$ is a functional of $(D_t^\d)_{t\ge 0}$ independent of~$\d$. As a consequence, the family 
\begin{equation}
\label{4.4.5}\left((X_t^\d,D_t^\d,  W_t^\d, \wi W_t^\d, W_t^{\d,m})_{t\ge 0},\tau_\e^\d\right)
\end{equation}
is tight (in~\eqref{4.4.5} the Brownian motions $W_t^\d$ and $W_t^{\d,m}$ are the ones defined by equation~\eqref{6.5}). Denote by 
\begin{equation}\label{4.4.6}
\left((X_t,D_t, W_t, \wi W_t, W_t^{m})_{t\ge 0},\tau_\e\right)
\end{equation}
a limiting point. Let us prove the intertwining. 

Using Proposition~\ref{P6.2}, for any smooth functions $g$ and $k$ on $M$, 
 \begin{align*}
 \EE[g(X^{\d}_t)F_k(D^{\d}_t)]&=  \EE [\EE[g(X^{\d}_t)F_k(D^{\d}_t)\vert \SF^{D^{\d}}_t]]\\
 &=   \EE [ \SU(D^{\d}_t)(g) F_k(D^{\d}_t)] \\ 
 &= \EE \left[\frac{F_g(D^{\d}_t)}{F_1(D^{\d}_t)}  F_k(D^{\d}_t) \right]\\
 \end{align*}
and passing to the limit yields the intertwining. 
  
  This property of $(D_t^\d, \wi W_t^\d)_{t\ge 0}$ being constant in law passes to the limit, and we have 
  \begin{equation}\label{4.1.7}
 \begin{split}
d\partial D_t(y)&= N^{D_t}(y)\left( d\wi{W}_t
+\left( \f12 h^{D_t}(y) - \f{\usm^{\partial D_t}(\partial D_t)}{\mu(D_t)} \right)  \,dt \right).
\end{split}
\end{equation}
We need to work with real-valued processes: 
we have from~\eqref{4.19}, for all $\d>0$, 
\begin{equation}
\label{4.17a}
\int_0^t\f{d\mu(D_s^\d)}{\underline{\mu}(\partial D_s^\d)}=-W_t^\d-\int_0^t\Delta_1f_\d (X_s^\d, D_s^\d)\, ds.
\end{equation}
This together with~\eqref{4.1.6} yields 
\begin{equation}\label{4.1.9}
 \begin{split}
d\partial D^\d_t(y)&= N^{D_t^\d}(y)\left( -\f{d\mu(D_s^\d)}{\underline{\mu}(\partial D_s^\d)}
+ \f12 h^{D^\d_t}(y)   \,dt \right) \\
\end{split}
\end{equation}
Again by constantness in law: 
\begin{equation}\label{4.1.10}
 \begin{split}
d\partial D_t(y)&= N^{D_t}(y)\left( -\f{d\mu(D_s)}{\underline{\mu}(\partial D_s)}
+ \f12 h^{D_t}(y)   \,dt \right).
\end{split}
\end{equation}
 So to prove our result we only need to prove that 
 \begin{equation}
\label{4.18a}
\int_0^t\f{d\mu(D_s)}{\underline{\mu}(\partial D_s)}=-W_t+\int_0^t h^{D_s} (X_s)\, ds+\int_0^t2\sin\left(\theta^{S_s}(X_s)\right)\,dL_s^{S_s}(X)
\end{equation}
and that 
\begin{equation}
\label{4.18b}
W_t=\int_0^t\langle N^{D_s}(X_s), dX_s\rangle. 
\end{equation}

Let us prove~\eqref{4.18b}. In all this paragraph we consider $M$ as isometrically embedded in some Euclidean space. In particular we are allowed to integrate vectorial quantities. We use  the fact that $dX_t^\d\otimes dW_t^\d$ converges in law to $dX_t\otimes dW_t$ (where $\otimes$ stands for bracket of semimartingales). But $dX_t^\d\otimes dW_t^\d$ is equal to $\n_1 f_\d(X_t^\d,D_t^\d)\, dt$. 
 Then by Lemma~\ref{Zheng} applied to $\n_1 f_\d(X_t^\d,D_t^\d)$ (which is uniformly bounded) and $U=\{(x,D), \ x\notin S(D)\}$ defined in~\eqref{H3}  we see that   the integral of $\n_1 f_\d(X_t^\d,D_t^\d)\, dt$ converges to the one of $N^{D_t}(X_t)\, dt$. But almost surely  $N^{D_t}(X_t)$ has norm~$1$ $dt$-a.e., implying that $dW_t=\langle N^{D_t}(X_t), dX_t\rangle$.

Let us now establish~\eqref{4.18a}.  It will be a consequence of the convergence of $(f_\d(X_t^\d, D_t^\d))_{t\ge 0}$ to $(\rho(X_t,\partial D_t)_{t\ge 0}$. 

 Write the It\^o formula for $f_\d(X_t^\d, D_t^\d)$: 
\begin{equation}
\label{4.14bis}
\begin{split}
d\left(f_\d(X_t^\d, D_t^\d)\right)=&\langle d_1f_\d(X_t^\d, D_t^\d), dX_t^\d\rangle +\f12\Delta_1f_\d(X_t^\d, D_t^\d)\,dt\\
&+\langle d_2f_\d(X_t^\d, D_t^\d), d\partial D_t^\d\rangle +\f12\n_2d_2f_\d(X_t^\d, D_t^\d)(d\partial D_t^\d,d\partial D_t^\d)\,dt\\
&+\langle \n_2 d_1f_\d(X_t^\d, D_t^\d),  d\partial D_t^\d \otimes dX_t^\d\rangle.
\end{split}
\end{equation}
From Proposition~\ref{A4.0bis}, possibly by extracting a subsequence,
 \begin{equation}
\label{4.8}
\left(f_\d(X_t^\d, D_t^\d)\right)_{t\ge 0}\overset{\SL}{\longrightarrow}\left(\rho(X_t, \partial D_t)\right)_{t\ge 0}.
\end{equation}

From~\eqref{5.4.6.1} we get for $i=1,2$,
\begin{equation}
 \label{5.4.6.2}
 \begin{split}
 &d_if_\d(x,D)-d_i\rho_\d(x,\partial D)\\
&=-d_i\ell_\e(x,D)\int_{T_xM}\varphi_\d(\vert u\vert )\left(\rho_\d(\exp_x(u)),\partial D)-\rho_\d(x,\partial D)\right)\, du\\&
 +(1-\ell_\e(x,D))\int_{T_xM}\varphi_\d(\vert u\vert )\left(d_i\rho_\d(\exp_x(u)),\partial D)-d_i\rho_\d(x,\partial D)\right)\, du
\end{split}
 \end{equation}
From this we see that   $d_1f_\d(\cdot , D)$ converges, locally uniformly outside $S(D)$,  to $d_1\rho(\cdot , \partial D)$ with respect to the distance $d_0$ of Appendix~\ref{appH}. We obtain, with Lemma~\ref{Zheng}, possibly by again extracting a subsequence, that
\begin{equation}
\label{4.9a}
\left(\int_0^t\langle d_1f_\d(X_s^\d, D_s^\d), dX_s^\d\rangle\right)_{t\ge 0}\overset{\SL}{\longrightarrow}\left(\int_0^t\langle d_1\rho (X_s,\partial  D_s), dX_s\rangle\right)_{t\ge 0}.
\end{equation}
More precisely, we have a sequence of martingales converging in law to a martingale $M_t$ which is a Brownian motion by  Theorem~3 in~\cite{Zheng:85}. For identifying the limiting martingale we use the convergence of $\langle d_1f_\d(X_s^\d, D_s^\d), dX_s^\d\rangle\otimes dX_s^\d$ to $  dM_s\otimes dX_s$ obtained again by Theorem~3 in~\cite{Zheng:85} (here again we use an isometric embedding of $M$). But Lemma~\ref{Zheng} proves that the limit is equal to $\n_1\rho (X_s,\partial  D_s)\, ds$, yielding~\eqref{4.9a}.

Next we prove that 
\begin{equation}
\label{4.10}
\left(\int_0^t\langle d_2f_\d(X_s^\d, D_s^\d), d\partial D_s^\d\rangle\right)_{t\ge 0}\overset{\SL}{\longrightarrow}\left(\int_0^t\langle d_2\rho (X_s,\partial  D_s), d\partial D_s\rangle\right)_{t\ge 0}.
\end{equation}
The argument is similar except that as we see with~\eqref{4.1.4}, the drift part of $d\partial D_s^\d$ is not well controlled as $X_t^\d$ approaches the skeleton. So one cannot proceed exactly the same way. But fortunately, for $x$  outside a $3\e/4$-neighbourhood of $\partial D$ and outside $S(D)$, we have 
\begin{equation}
\label{4.10a}
\begin{split}
&\langle d_2 f_\d(x,D), N|_{\partial D}\rangle \\&=c_\d\int_{T_xM}\varphi_\d(\vert u\vert )\langle -N\left(P(\exp_x(u)\right), N\left(P(\exp_x(u)\right)\rangle \, du=-c_\d
\end{split}
\end{equation} 
where $c_\d$ is defined in~\eqref{cdelta}.
This together with~\eqref{4.1.9} suggests to write 
\begin{align*} 
\int_0^t\langle d_2f_\d(X_s^\d, D_s^\d), d\partial D_s^\d\rangle=&\left(\int_0^t\langle d_2f_\d(X_s^\d, D_s^\d), d\partial D_s^\d\rangle+ c_\d \int_0^t  {{   \langle N^{D^{\d}_s}, d \partial D^{\d}_s \rangle  }}  \right)\\
 & -c_\d\int_0^t {{   \langle N^{D^{\d}_s}, d \partial D^{\d}_s \rangle  }} .
\end{align*}
The second line clearly converges. The 
 right hand side in the first line can be written 
 \begin{equation}
  \label{4.10b}
  \int_0^t\tilde\ell_{\e}{{   (X_s^\d, D_s^\d) }} \left\langle d_2f_\d(X_s^\d, D_s^\d)+c_\d {{   N^{D^{\d}_s} }}, d\partial D_s^\d\right\rangle
 \end{equation}
 with 
 $(x,D)\mapsto \tilde\ell_\e(x,D)\df(\tilde h_\e\circ \rho_{\partial D})(x)$ where  $\tilde h_\e\equiv 1$ in $[0,3\e/4]$, $\tilde h_\e\equiv 0$ in $[\e,\infty)$ and $\tilde h_\e$ is smooth and nonincreasing in $[0,\infty)$.

With this last integral we can proceed as for~\eqref{4.9a}, {{   after passing to the limit, and  since $ \lim_{\d \to 0} c_\d = 1$, we get  \eqref{4.10a}.}}

Similarly we obtain the two following convergences for the second derivatives. 

\begin{equation}
\label{4.12a}
\begin{split}
&\left(\int_0^t \n_2d_2f_\d(X_s^\d, D_s^\d)(d\partial D_s^\d,d\partial D_s^\d)\right)_{t\ge 0}\\&\overset{\SL}{\longrightarrow}\left(\int_0^t \n_2d_2\rho(X_s, \partial D_s)\left(N(P^{\partial D_s}(X_s) ,N(P^{\partial D_s}(X_s)\right)\,ds\right)_{t\ge 0}\equiv 0
\end{split}
\end{equation}
where $P^{\partial D_s}(X_s)$ is the orthogonal projection of $X_s$ on $\partial D_s$ (which is defined $ds$-almost everywhere), 
\begin{equation}
\label{4.13a}
\begin{split}
&\left(\int_0^t \langle \n_{2}d_1f_\d(X_s^\d, D_s^\d), d\partial D_t^\d \otimes dX_t^\d\rangle )\right)_{t\ge 0}\\&\overset{\SL}{\longrightarrow}\left(\int_0^t \langle \n_{2}d_1\rho (X_s,\partial  D_s), d\partial D_s\otimes dX_s\rangle\right)_{t\ge 0}\equiv 0
\end{split}
\end{equation}
since $d_1\rho(X_s, \partial D_s)= {{  +}}\langle N^{D_s}(X_s),\cdot \rangle$ which implies that the covariant derivative in the second variable with respect to $N^{D_s}$ is equal to~$0$.
On the other hand, by It\^ o-Tanaka formula
(see Proposition \ref{PropF1} in Appendix \ref{AppendixF}
using  that $\rho(x,\partial D)$ is almost everywhere the minimum of two smooth functions) {  together with Assumption~\ref{A4.0} which allows to only consider the regular skeleton, together with Theorem~\ref{T3.1} which says that the latter has absolutely continuous variation (useful for the term $dL_t^{S_t}(X)$),} 
we have 
\begin{equation}
\label{4.14a}
\begin{split}
d\left(\rho(X_t,\partial  D_t)\right)=&
\langle d_1\rho (X_t,\partial  D_t), dX_t\rangle  -\f12h^{D_t}(X_t){{   \un_{D_t  \backslash S_t}(X_t) }} \,dt+\langle d_2\rho (X_t,\partial  D_t), d\partial D_t\rangle\\&+0+0-\sin\left(\theta^{S_t}(X_t)\right)\,dL_t^{S_t}(X).
\end{split}
\end{equation}

Using \eqref{4.14bis}, \eqref{4.8}, \eqref{4.9a}, \eqref{4.10}, \eqref{4.12a}, \eqref{4.13a}, \eqref{4.14a} we obtain that 
\begin{equation}
\label{4.16a}
\left(\int_0^t \D_1f_\d(X_s^\d, D_s^\d)\, ds\right)_{t\ge 0}\overset{\SL}{\longrightarrow}\left(\int_0^t -h^{D_s} (X_s) {{   \un_{D_s  \backslash S_s}(X_s) }} \, ds-\int_0^t2\sin\left(\theta^{S_s}(X_s)\right)\,dL_s^{S_s}(X) \right)_{t\ge 0}.
\end{equation}
It remains to pass in the limit as $\d$ goes to zero in \eqref{4.17a}, to deduce \eqref{4.18a}.
\end{proof}

\begin{remark}
From \eqref{4.14a}, it can be deduced that
 \begin{equation}
 \label{4.15a}
 d\left(\rho(X_t,\partial  D_t)\right)=
 \f12\left(h^{D_t}(X_t){{   \un_{D_t  \backslash S_t}(X_t) }}-h^{D_t}\left(P^{\partial D_t}(X_t)\right)\right)\, dt +\sin\left(\theta^{S_t}(X_t)\right)\,dL_t^{S_t}(X).
 \end{equation} 
Indeed, \eqref{4.18b} implies that
\bq \langle d_1\rho (X_t,\partial  D_t), dX_t\rangle&=& dW_t\eq
and due to \eqref{4.10}, we have
\bq
\lefteqn{\langle d_2\rho (X_t,\partial  D_t), d\partial D_t\rangle}\\&=&
\lim_{\d\ri0}\langle d_2\rho (X_t^\d,\partial  D_t^\d), d\partial D_t^\d\rangle
\\
&=&
\lim_{\d\ri0}
-dW^\d_t-\lt(\Delta_1f_\d(P^{\pa D_t^\d}(X^\d_t),D_t^\d)+\f12h^{D_t^\d}(P^{\pa D_t^\d}(X^\d_t))\rt)dt\eq
where we used
\eqref{4.17a} in conjunction with \eqref{4.1.9}. 
\par Taking into account \eqref{4.16a}, we identify the last limit with
\bq
-dW_t+\lt(h^{D_t}(X_t){{   \un_{D_t  \backslash S_t}(X_t) }}-\f12h(P^{\pa D_t}(X_t))\rt)dt +{{  2\sin\left(\theta^{S_t}(X_t)\right)\,dL_t^{S_t}(X) }}
\eq
\end{remark}

\section{Intertwined dual processes: decoupling and reflection on boundary}
\label{Section7}
\setcounter{equation}0

 In this section we consider another canonical and extremal situation, the case where $f^D$ vanishes almost everywhere. More precisely, it is the limiting situation where $f^D$ is constant outside  a $\e$-neighbourhood of the boundary. This situation is completely opposite to the one of Section~\ref{Section4} where the coupling is maximal. 
 
 \begin{thm}
 \label{T7.1}
 There exists a pair $(X_t,D_t)_{t\ge 0}$ of $\tau_\e$-intertwined processes in the sense of Definition~\ref{D6.0}  satisfying
\begin{equation}
 \label{7.0}
 {{   d\partial D_t(y)}}=N^{D_t}(y)\left( dW_t +\f12 h^{D_t}(y)dt-dL^{\partial D_t}_t(X)\right)
 \end{equation}
 where $X_t$ is a $M$-valued Brownian motion started at uniform law in $D_0$, $W_t$ is a real-valued Brownian motion independent of $X_t$, $L^{\partial D_t}_t(X)$ is the local time of $X_t$ on the moving boundary $\partial D_t$.
\end{thm}
\begin{remark}
Equation~\eqref{7.0} can be considered as a limiting case of~\eqref{6.5}. Here Assumption~\ref{A4.0} is not needed since the morphological skeleton of $D$ does not play a role, and  the map $D\mapsto \partial D$ is already sufficiently regular.
 \end{remark}

\begin{proof}
The proof is quite similar to the one of Theorem~\ref{T4.1}, but with another family of functions $f_\d^D$, namely  $f_\d^D:=h_\d\circ \rho_{\partial D}$ where $h_\d$ is defined in the proof of Proposition~\ref{P6.3}: $h_\d$ is a smooth nondecreasing function from $[0,\infty)$ to $\R_+$ such that $h_\d(r)=r$ for $r\in [0,\d/2]$, $h_\d(r)=(3/4)\d$ for $r\ge \d$ and $\|h_\d'\|_\infty\le 1$. But here, as $\e$ is fixed, we will let $\d\searrow 0$. Again we  construct for
 each $\d>0$, an intertwined processes $(X_t^\d,D_t^\d)_{t\ge 0}$ stopped at $\tau_\e^\d$. Again
  all $(X_t^\d,D_t^\d)_{t\ge 0}$  are tight, and a limiting process $(X_t,D_t)_{t\ge 0}$ stopped at $\tau_\e$ provides an intertwining. The proof of~\eqref{7.0} goes along the same lines as the one of~\eqref{4.0}. 
\end{proof}

We end this section with another canonical construction, where the functions $f_\d^D$ approximate $-\rho_{\partial D}$. 
\begin{thm}
 \label{T7.2}
 Under assumption~\ref{A4.0bis}, 
 there exists an intertwining $(X_t,D_t)_{t\ge0}$ stopped at $\tau_\e$, satisfying
\begin{equation}
 \label{7.2}
 \begin{split}
 d\partial D_t(y)=&N^{D_t}(y)\Biggl(-\left\langle dX_t, N^{D_t}(X_t)\right\rangle +\left(\f12 h^{D_t}(y)+h^{D_t}(X_t) {{   \un_{D_t  \backslash S_t}(X_t) }} \right)\,dt\\&+2\sin(\theta^{S_t}(X_t))\, dL_t^{S_t}(X) -2dL_t^{\partial D_t}(X)\Biggr)
 \end{split}
\end{equation}
\end{thm}
\begin{proof}
It is completely similar to the ones of Theorems~\ref{T4.1} and~\ref{T7.1}.
\end{proof}
\section{Some fundamental examples}
\label{Section5}
\setcounter{equation}0

\subsection{Real Brownian motion and three-dimensional Bessel process}
\label{Subsection5.1}

{We come back to} the case where $M=\RR$. Assume that the Brownian motion $X$ starts from 0
(to respect rigorously the above framework, $X$ should start from the uniform distribution on $D_0\df[-\epsilon,\epsilon]$
and next we should let $\epsilon$ go to $0_+$).
Due to the invariance by symmetry of \eqref{4.0}, for any $t>0$, $D_t$  remains a symmetric interval, let us write it $[-R_t,R_t]$.
In this simple setting, we have $N^{D_t}(\cdot)=-\sign(\cdot)$ on $\RR\backslash{\{0\}}$, $h^{D_t}=0$ and $S_t=\{0\}$, for any $t>0$.
Thus {\eqref{4.0}} writes
\begin{equation}
\label{5.1.1}
dR_t=\sign(X_t)dX_t+2dL_t
\end{equation}
where $L\df(L_t)_{t\geq 0}$ is the local time of $X$ at 0.
Namely we get that
\bq
\fo t\geq 0,\qquad R_t&=&\int_0^t\sign(X_s)\,dX_s+2L_t\\
&=&\vert X_t\vert +L_t
\eq
by Tanaka's formula. 
It is well-known that $R\df(R_t)_{t\geq 0}$ is a Bessel process of dimension 3 (cf.\ e.g.\ Corollary 3.8 of Chapter 6 of Revuz and Yor \cite{MR1725357}).
In particular, we get that with the notation {introduced in \eqref{2.3.1}},
\bq
\fo t\geq 0,\qquad \rho^+_{\partial D_t}(X_t)&=&\min(X_t+R_t, R_t-X_t)\eq
But except at time $t=0$, this quantity is always positive: a.s.\ $X_t$ never touch the boundary of $D_t$ for $t>0$.
Indeed, if for some $t>0$ we have $\vert X_t\vert =R_t$, we deduce that $L_t=0$, namely a contradiction, since $X_0=0$.
\par
In particular, we see that the  intertwining coupling we have constructed is different from the one proposed by Pitman \cite{MR0375485},
which is a.s.\ touching (the upper) boundary repeatedly. {Instead we end up with the intertwining dual constructed in \cite{Mic2020} via stochastic flows.
 It is mentioned there how to deduce the classical Pitman's dual, via L\'evy's theorem.}
 
 Here is an alternative approach.
While Equation~\eqref{5.1.1} is obtained from approximating $x\mapsto |r-x|$  outside an $\e$-neighbourhood of~$0$ when $D=[-r,r]$ by smooth functions $f^D$ satisfying Assumption~\ref{A6.1}, 
we are able to recover Pitman theorem by rather approximating $x\mapsto -x$ in $D=[-r,r]$ outside the only $\e$-neighbourhood of~$-r$. 
In the limit of \eqref{6.5} as $\epsilon$ goes to zero, on the one hand we have 
\begin{equation}
\label{5.1.2.1}
\un_{\left\{X_t\not=R_t\right\}}dR_t=dX_t,
\end{equation}
on the other hand we have $X_t+R_t\ge 0$, so that $X_t+R_t$ is the solution to the Skorohod problem associated to $2X_t$. 
We get 
\begin{equation}
 \label{5.1.2.2}
 R_t+X_t=2X_t-2\min_{0\le s\le t}X_s.
 \end{equation}
which is equivalent to
 \begin{equation}
 \label{5.1.2}
 R_t=X_t-2\min_{0\le s\le t}X_s.
 \end{equation}
 
 The answer to the question: what would be a symmetric construction with local time at the two ends of $[-R_t,R_t]$ is given by Theorem~\ref{T7.2}. We obtained intertwined processes with 
 \begin{equation}
 \label{5.1.3}
 R_t=-\int_0^t\sign(X_s)\,dX_s-2L_t^0(X)+2L_t^0(R-X)+2L_t^0(R+X).
 \end{equation}
 
\subsection{Brownian motion and disks in rotationally symmetric manifolds}
\label{Subsection5.2}
This is the simpler example since the skeleton is never hit by the Brownian motion. Consider a complete $d$-dimensional manifold with $d\ge 2$, rotationally symmetric around a point $o\in M$. Denote by $(r,\Theta)$ polar coordinates with $r(x)=\rho(o,x)$ and 
\begin{equation}
 \label{5.2.1}
 ds^2=dr^2+f^2(r)\,d\Theta^2
\end{equation}
the metric in polar coordinates. Then the radial Laplacian is 
\begin{equation}
 \label{5.2.2}
 \Delta_r=\f{\partial^2}{(\partial r)^2} +b(r)\f{\partial}{\partial r}\quad\hbox{with}\quad b=(d-1)(\ln f)'.
\end{equation}
 
 We will investigate set-valued processes $D_t=B(o,R_t)$ where $B(o,r)$ is the open geodesic ball centered at~$o$, with radius $r$. The skeleton of $B(o,R_t)$ is the point~$o$. 
 
 Let $X_t$ be a Brownian motion in $M$ satisfying $X_0\sim \SU(D_0)$ for some $D_0=B(o,r_0)$. Denote by $\rho_t\df r(X_t)$ the radial part of $X_t$. Then 
\begin{equation}
 \label{5.2.3}
 d\rho_t=d\b_t+\f12 b(\rho_t)\,dt, \qquad \rho_0\sim \SU^f((0,r_0))
\end{equation}
where $(\b_t)_{t\geq 0}$ is a real Brownian motion and 
\begin{equation}
 \label{5.2.4}
 \SU^f(dr)\df \f{f(r)}{\int_0^{r_0}f(s)\,ds}\,dr.
\end{equation}

The evolution equation~\eqref{4.0} for $D_t$ shows by symmetry that for all $t\ge 0$, $D_t=B(0,R_t)$ for some real-valued process $(R_t)_t$. Moreover it writes
\begin{equation}
 \label{5.2.5}
 \begin{split}
 d\rho_t&=d\b_t+\f12 b(\rho_t)\,dt\\
 dR_t&= d\beta_t+\left[-\f12 b(R_t)+b(\rho_t)\right]\,dt.
 \end{split}
\end{equation}
\begin{prop}
 \label{P5.2.1}
 The system of equations~\eqref{5.2.5} has a solution up to explosion time of $(R_t)_t$
 \begin{equation}
  \label{5.3.9}
  \tau^D\df \inf\{ t\ge 0, R_t\not\in (0,\infty)\},
 \end{equation} 
 which satisfies for all $t<\tau^D$,
 \begin{equation}
  \label{5.3.10}
  0< \rho_t< R_t.
 \end{equation}
 The corresponding set-valued process $D_t=B(o,R_t)$ is solution to equation~\eqref{4.0}, and in particular, for all $\SF^D$-stopping time $\tau$, 
\begin{equation}
 \label{5.3.11}
 \SL(X_\tau|\SF_\tau^D)=\SU(D_\tau)\quad\hbox{as well as}\quad \SL(\rho_\tau|\SF_\tau^D)=\SU^f((0,R_\tau)).
\end{equation}
\end{prop}
\begin{proof}
We only have to check~\eqref{5.3.10}. By~\eqref{5.2.5},
\begin{equation}
\label{5.2.6}
d(R_t-\rho_t)=\f12\left[b(\rho_t)-b(R_t)\right]\,dt,
\end{equation}
which vanishes on $\{R_t=\rho_t\}$, and since $b$ is smooth, if $\rho_0<R_0$, then $\rho_t<\R_t$ for all times. 
\end{proof}
\subsection{Brownian motion and annulus in $2$-dimensional rotationally symmetric manifolds}
\label{Subsection5.3}

Let $M$ be a complete $2$-dimensional Riemannian manifold, rotationally symmetric around a point $o\in M$. Denote by $(r,\theta)$ polar coordinates with $r(x)=\rho(o,x)$ and 
\begin{equation}
 \label{5.3.1}
 ds^2=dr^2+f^2(r)\,d\theta^2
\end{equation}
the metric in polar coordinates. Then the radial Laplacian is 
\begin{equation}
 \label{5.3.2}
 \Delta_r=\f{\partial^2}{(\partial r)^2} +b(r)\f{\partial}{\partial r}\quad\hbox{with}\quad b=(\ln f)'.
\end{equation}
If $0\le r^-\le r^+$, let
\begin{equation}
 \label{5.3.3}
 A(r^-,r^+) \df \{x\in M,\  r^-\leq r(x)\leq  r^+\}\quad\hbox{if}\quad r^-<r^+, \quad A(r^-,r^+)\df\emptyset,
\end{equation}
the closed annulus delimited by the radius $r^-$ and $r^+$. 

In the following we will investigate set-valued processes $D_t=A(R_t^-, R_t^+)$. 
The skeleton of $A(R_t^-,R_t^+)$ is the circle 
\begin{equation}
 \label{5.3.4}
 S_t=C(o,R_t^0)\quad \hbox{with}\quad R_t^0:=\f12(R_t^-+R_t^+).
\end{equation}
Let $X_t$ be a Brownian motion in $M$ satisfying $X_0\sim \SU(D_0)$ for some $D_0=A(r_0^-,r_0^+)$. Denote by $\rho_t\df r(X_t)$ the radial part of $X_t$. Then 
\begin{equation}
 \label{5.3.5}
 d\rho_t=d\b_t+\f12 b(\rho_t)\,dt, \qquad \rho_0\sim \SU^f((r_0^-,r_0^+))
\end{equation}
where $\b_t$ is a real Brownian motion and 
\begin{equation}
 \label{5.3.6}
 \SU^f((r_0^-,r_0^+))(dr)\df \f{f(r)}{\int_{r_0^-}^{r_0^+}f(s)\,ds}\,dr.
\end{equation}

The evolution equation~\eqref{4.0} for $D_t$ shows by symmetry that for all $t\ge 0$, $D_t=A(R_t^-,R_t^+)$ for some real-valued processes $R_t^-\le R_t^+$. Moreover it writes
\begin{equation}
 \label{5.3.7}
 \begin{split}
 d\rho_t&=\sign(\rho_t-R_t^0)\,dW_t+\f12 b(\rho_t)\,dt\\
 dR_t^+&= dW_t+\left[-\f12 b(R_t^+)+\sign(\rho_t-R_t^0)b(\rho_t)\right]\,dt +2L_t^{R_t^0}(\rho)\\
 dR_t^-&= -dW_t+\left[-\f12 b(R_t^-)-\sign(\rho_t-R_t^0)b(\rho_t)\right]\,dt -2L_t^{R_t^0}(\rho)\\
 R_t^0&=\f12\left(R_t^-+R_t^+\right)
 \end{split}
\end{equation}
and these equations imply 
\begin{equation}
 \label{5.3.8}
 dR_t^0=-\f14\left[b(R_t^+)+b(R_t^-)\right]\,dt.
\end{equation}

\begin{prop}
 \label{P5.3.1}
 The system of equations~\eqref{5.3.7} has a solution up to explosion time
 \begin{equation}
  \label{5.3.9.1}
  \tau^D\df \inf\{ t\ge 0, (R_t^-,R_t^+)\not\in (0,\infty)^2\},
 \end{equation} 
 which satisfies for all $t<\tau^D$,
 \begin{equation}
  \label{5.3.10.1}
  R_t^-\le\rho_t\le R_t^+.
 \end{equation}
 The corresponding set-valued process $D_t=A(R_t^-,R_t^+)$ is solution to equation~\eqref{4.0}, and in particular, for all $\SF^D$-stopping time $\tau$, 
\begin{equation}
 \label{5.3.11.1}
 \SL(X_\tau|\SF_\tau^D)=\SU(D_\tau)\quad\hbox{as well as}\quad \SL(\rho_\tau|\SF_\tau^D)=\SU^f((R_\tau^-,R_\tau^+)).
\end{equation}
\end{prop}
\begin{proof}
Fix $\e>0$ and $\a\in (0,1)$. We will first solve the system of equations until the exit time $\tau_\e$ and then let $\e\searrow 0$. Let us construct functions $f_\d^D(x)$ which satisfies equation~\eqref{5.4.5}. It will be easier here because there is no need of functions $\ell_\e$ and $g_\d$.

For $\d\in (0,\e)$, let $\varphi_\d :\RR\to \RR$ be the function with support equal to $[-\d/2,\d/2]$, satisfying for $-\d/2<r<\d/2$:
\begin{equation}
 \label{5.3.12}
 \varphi_\d (r)\df \f1{c(\d)}\exp\left(-\f1{\left(\f{\d}{2}\right)^2-r^2}\right)\quad \hbox{with}\quad c(\d):=\int_{-\d/2}^{\d/2}\exp\left(-\f1{\left(\f{\d}{2}\right)^2-s^2}\right)\,ds,
\end{equation}
and let 
\begin{equation}
 \label{5.3.13}
 \begin{split}
  \sign_\d : \RR&\to\RR\\
  r&\mapsto -1+2\int_{-\infty}^r\varphi_\d(s)\,ds.
 \end{split}
\end{equation}
The functions $\varphi_\d$ and $\sign_\d$ are both smooth and Lipschitz, and they respectively approximate $\d_0$ and $\sign$. For $0<r^-<r^+$ satisfying $r^+-r^-\ge 2\e$, defining $\di r^0 :=\f12(r^-+r^+)$, for $x\in A(r^-,r^+)$ let 
\begin{equation}
\label{5.3.14}
f^{A(r^-,r^+)}(x)=f(x,r^-,r^+)=g(r(x))\quad \hbox{with}\quad g(r)=g(r,r^-,r^+)=\int_{r^-}^{r} -\sign_\d(s-r^0)\, ds.
\end{equation}
Clearly $f(x,r^-, r^+)$ is $1$-Lipschitz in the first variable. A computation shows that   
\begin{equation}
\label{5.3.14.1}
\partial_{r^+}g(r,r^-,r^+)=\int_{-\e}^{r^0}\varphi_\d(v)\,dv\quad \hbox{and}\quad \partial_{r^-}g(r,r^-,r^+)=-\int_{r-r^0}^{\e}\varphi_\d(v)\,dv
\end{equation}
showing  that $g$ and $f$ are $1$-Lipschitz. Then the vector $N:=N_{\partial A(r^-,r^+)}$  is equal to $-\un_{\{r(x)=r^+\}}\partial_{r^+}+\un_{\{r(x)=r^-\}}\partial_{r^-}$ so that 
\begin{equation}
\label{5.3.14.2}
\langle \n f, N\rangle \equiv 1\quad \hbox{and}\quad  \n df(N,N)\equiv 0.
\end{equation}
This yields an elementary proof of  the properties of Proposition~\ref{A4.0bis}. We can use  Theorem~\ref{T4.1} to solve equation~\eqref{5.3.7} until the stopping time $\tau_\e$. 

We are left to prove that $\tau_\e\nearrow \tau^D$ a.s.\ as $\e\searrow 0$. This is a direct consequence of the fact that the volume of $A(R_t^-,R_t^+)$ is a time changed Bessel process of dimension $3$ (by~\cite{zbMATH07470497} Theorem~5), proving that $A(R_t^-,R_t^+)$ cannot collapse onto its skeleton.
\end{proof}
\begin{remark}
 \label{R5.3.1}
 After the hitting time of $0$ by $R_t^-$, the processes can continue to evolve under the regime of Section~\ref{Subsection5.2}.
\end{remark}

We recover from Proposition \ref{P5.3.1} a result from \cite{MR3634282} stating that $([R^-_t,R^+_t])_{t\geq 0}$ is an intertwining dual process for the real diffusion $(\rho_t)_{t\geq 0}$.
In particular, we deduce that if $(\rho_t)_{t\geq 0}$ is positive recurrent and if $+\iy$ is an entrance boundary, then $([R^-_t,R^+_t])_{t\geq 0}$ reaches $[0,+\iy]$ in finite time and this finite time is a strong stationary time for $(\rho_t)_{t\geq 0}$, see \cite{MR3634282} for more details.

\subsection{Brownian motion and symmetric convex sets in $\RR^2$}
\label{Subsection5.4}

In this section we take $M=\RR^2$ endowed with the Euclidean metric.
For any integer $n\ge 2$, let $G_n$ the group of isometries of $\RR^2$ generated by the rotation of angle $\di \f{2\pi}n$ and the symmetry with respect to the horizontal axis. 
 Consider a smooth strictly convex bounded set $D_0\subset M$ with smooth boundary, stable by the action of $G_n$.  Also assume that its skeleton has the form $S_0=G_nH_0$, $H_0$ being  an horizontal interval $H_0=[0,x_0]\times \{0\}$ for some $x_0>0$. An example of such a set when $n=2$ is the interior of an ellipse, the skeleton being the interval between the two foci. Assume that $X_t$ is a Brownian motion in $\RR^2$ satisfying $X_0\sim \SU(D_0)$. Let us investigate the evolution of $(X_t,D_t)$. Notice that it is the first example where we really have to deal with infinite dimensional processes.
By conservation of the convexity by the normal and mean curvature flows, $D_t$ will stay convex. It will also stay symmetric. All the results of this subsection will be proved in the forthcoming paper  \cite{ACM:21}:
\begin{prop}
\label{P5.4.1}
The  skeleton of $D_t$  always takes the form  $S_t=G_nH_t$ with  $H_t=[0,x_t]\times\{0\}$ an horizontal interval.
\end{prop}
\begin{proof}
 See \cite{ACM:21}
\end{proof}

Denote by $(\imath, \jmath)$ the canonical basis of $\RR^2$, and $X_t=(X_t^{(1)},X_t^{(2)})$. In this notation, when $-\pi/n < \theta^{S_t}(X_t) < \pi/n$,  the vector $N^{D_t}(X_t)$ of Equation~\eqref{4.0} writes 
\begin{equation}
 \label{5.4.1}
 N^{D_t}(X_t)=-\sign(X_t^{(1)})\cos(\theta^{S_t}(X_t))\imath-\sign(X_t^{(2)})\sin(\theta^{S_t}(X_t))\jmath
\end{equation}
where $\theta^{S_t}(x)$ is naturally extended to $D_t$ by being constant on lines normal to the boundary (see \cite{ACM:21}). Notice that $x\mapsto \theta^{S_t}(x)$ is locally Lipschitz on $D_t$ and is equal to $0$ on $D_t\cap ( [x_t,\infty)\times\{0\})$. Also notice that the function $h^{D_t}$ is locally Lipschitz on $D_t\backslash S_t$.    With these notations, equation~\eqref{4.0} writes (again when $-\pi/n < \theta^{S_t}(X_t) < \pi/n$)
\begin{equation}
 \label{5.4.2}
 \begin{split}
 d\partial D_t(y)=&-N^{D_t}(y)\Biggl(\sign(X_t^{(1)})\cos(\theta^{S_t}(X_t))\,dX_t^{(1)}+\sin(\theta^{S_t}(X_t))\sign(X_t^{(2)})\,dX_t^{(2)}\\
 & +\left(\f12h^{D_t}(y)-h^{D_t}(X_t)\right)\,dt-2\sin(\theta^{S_t}(X_t))dL_t(X^{(2)})\Biggr).
 \end{split}
\end{equation}
This equation written for $-\pi/n < \theta^{S_t}(X_t) < \pi/n$  is enough to describe  the whole coupling, thanks to the symmetry properties of $D_t$.

Let us investigate the motion of the skeleton $\tilde S_t$ of the solution $\tilde D_t$ of equation~\eqref{6.5mod0} (garanteed by Theorem~\ref{6.5mod0hyp}). 
\begin{prop}
\label{P5.4.2}
 The process $\di \left(\tilde D_t\right)_{t\ge 0}$ takes its values in a closed subset $\wi \cF^{\a,\e}$ of $ \cF^{\a,\e}$, invariant by $G_n$, such that on $\wi \cF^{\a,\e}$,   the map $D\mapsto h^D|_{\partial D}$ is continuous from $\wi \cF^{\a,\e}$ (with the $C^{2+\a}$ metric) to $C^2(\partial D)$. Its skeleton $\tilde S_t$ satisfies $\tilde S_t=G_n\left([0,\tilde x_t]\times \{0\}\right)$ for some process $\tilde x_t$.
\end{prop}
\begin{proof}
 See \cite{ACM:21}
\end{proof}
In the next result we prove that the skeleton has finite variation and is monotonly decreasing.
\begin{prop}
\label{P5.4.3}
   The right endpoint $(\tilde x_t,0)$ in the horizontal axis of the skeleton $\tilde S_t$ satisfies
\begin{equation}
 \label{5.4.3}
 \f{d \tilde x_t}{dt}=\f{\rho^2((\tilde x_t,0),\tilde y_t)}2(h^{\tilde D_t})''(\tilde y_t), 
\end{equation}
$\tilde y_t$ being the point of $\partial \tilde  D_t$ in the horizontal line with the greatest abscissa, and the second derivative being calculated with curvilinear coordinates on $\partial\tilde  D_t$.  Notice that $(h^{\tilde D_t})''(\tilde y_t)\le~0$, proving that the process $S(\tilde D_t)$ is monotonly decreasing.
\end{prop}
\begin{proof}
Let us investigate the motion of a point in $\tilde S_t$ close to  $(\tilde x_t,0)$. This point has two closest points in $\partial \tilde D_t$, which we call $ \wi y_{1,t}$ and $ \wi y_{2,t}$, the first one having positive second coordinate. We will use Theorem~\ref{T3.1} and~\eqref{3.22}. 
Call $\hat x_t$ the point in the skeleton corresponding to $ \wi y_{1,t}$ and $ \wi y_{2,t}$. We have $N_1(\hat x_t)=-\cos \theta(\hat x_t)\imath -\sin \theta(\hat x_t)\jmath$, $N_2(\hat x_t)=-\cos \theta(\hat x_t)\imath +\sin \theta(\hat x_t)\jmath$, $N_1^S(\hat x_t)=-\jmath$. Denote $T( \wi y_{1,t})$ the tangent vector to $\partial \tilde D_t$ at $ \wi y_{1,t}$, corresponding to increasing of $\theta$: $T( \wi y_{1,t})=-\sin \theta(\hat x_t)\imath +\cos \theta(\hat x_t)\jmath$. Write $h'(\wi y_{1,t})$ the curvilinear derivative of $h(\wi y_{1,t})$ in the direction of $T( \wi y_{1,t})$. Then the vector $J_1^\perp(1)$ of~\eqref{3.22} is equal to $\di -\f12\rho_S( \wi y_{1,t})h'(\wi y_{1,t})T( \wi y_{1,t})$. So we get from~\eqref{3.22}:
\begin{equation}
\label{5.4.4}
\begin{split}
\f{d}{dt}\hat x_t&=\f12\rho_S( \wi y_{1,t})h'(\wi y_{1,t})\left(\sin \theta(\hat x_t)+\f{\cos^2 \theta(\hat x_t)}{\sin \theta(\hat x_t)}\right)\imath\\
&=\f{\rho_S( \wi y_{1,t})h'(\wi y_{1,t})}{2\sin \theta(\hat x_t)}\imath\\
&=\f{\rho_S^2( \wi y_{1,t})h'(\wi y_{1,t})}{2\wi y_{1,t}^{(2)}}\imath\quad \hbox{with}\quad \wi y_{1,t}=(\wi y_{1,t}^{(1)},\wi y_{1,t}^{(2)}).
\end{split}
\end{equation} 
In the limit, as $\wi y_{1,t}^{(2)}$ goes to zero, we obtain the motion of $\tilde x_t$ and using the symmetry of the convex set, we have $h'(\wi y_t)=0$ so that  we can replace $\di \f{h'(\wi y_{1,t})}{\wi y_{1,t}^{(2)}}$ by $\di h''(\wi y_t)$. This yields~\eqref{5.4.3}.
\end{proof}

In particular a Brownian motion $X_t$ will never meet the ends of $\tilde S_t$.

A  solution to~\eqref{5.4.2} can be found with the help of 
Theorem~\ref{T4.1}.  The  family of functions $f_\d(x,D)$ defined in~\eqref{5.4.5} takes the form: 
\begin{equation}
\label{5.4.5bis}
\begin{split}
f_\d(x,D)&=\ell_\e(x)\rho_\d(x,\partial D)+(1-\ell_\e(x))\int_{\RR^2}\varphi_\d(\vert x-y\vert)\rho_\d(y,\partial D)\, dy\\
&= \ell_\e(x)\rho_\d(x,\partial D)+(1-\ell_\e(x))\int_{\RR^2}\varphi_\d(\vert y\vert )\rho_\d(x-y,\partial D)\, dy.
\end{split}
\end{equation}

The investigation of the lifetime of the solution to~\eqref{5.4.2} is not easy. In  \cite{ACM:21} we prove that the lifetime is the time when $D_t$ meets its skeleton $S_t$. So it is enough to investigate the  time $\tilde\tau$ when $\tilde D_t$ meets its skeleton $\tilde S_t$.  We have no example where this happens. The next proposition yields examples where the lifetime is infinite, together with nice properties related to the symmetry group $G_n$.

\begin{prop}
\label{P5.4.4}
\begin{enumerate}
\item the process $\di \left(\f{\usm^{\partial\wi D_t}(\partial\wi D_t)}{\mu(\wi D_t)}\right)_{0\le t<\tilde \tau}$ is a supermartingale;
\item when $\tilde S_0$ is $G_n$-symmetric with $n\ge 3$, then the entropy process $\left(\widetilde{\rm{Ent}}_t \right)_{0\le t<\tilde \tau}$ defined as the integral of $\rho\log \rho$ with respect to the curvilinear abscissa in $\partial \tilde D_t$, $\rho$ being the curvature of $\partial \tilde D_t$, is a supermartingale;
\item when $\tilde S_0$ is $G_n$-symmetric with $n\ge 7$, then $\tilde \tau=\infty$ a.s. Consequently, when $ S_0$ is $G_n$-symmetric with $n\ge 7$,
Equation~\eqref{5.4.2} provides an intertwining with infinite lifetime.
\end{enumerate}
\end{prop}
\begin{proof}
 See \cite{ACM:21}
\end{proof}

\appendix
\renewcommand{\theequation}{\Alph{section}.\arabic{equation}}

\section{An integration by parts on domains with boundary}
\label{Section2}
\setcounter{equation}0

Our goal here is to obtain an extension of Stokes's formula  on a domain with a smooth boundary, for functions which degenerate on the skeleton.
We take the opportunity to recall this notion, as well as related   geometric concepts.
\par\me
Let $M$ be a $d$-dimensional Riemannian manifold 
{and} $D\subset M$ a {compact and} connected domain with smooth boundary $\partial D$. For $y\in \partial D$, let $N(y)$ be the inward normal vector. Denote by $S'$ the inward (morphological) skeleton of $D$: $S'$ is the set of points in $D$ such that (i) the distance to $\partial D$ is not smooth 
and (ii) there are points  around them where the distance to $\pa D$ is smooth with a non vanishing gradient. Denote
\begin{equation}
 \label{2.1}
 \tau(y)=\inf\{t>0,\ \exp_y(t N(y))\in S'\}. 
\end{equation}
Let $S$ be the set of regular points of $S'$, which we can describe as follows:
 if $x\in S$, then there exists a unique {couple $(y_1,y_2)$ of distinct points from {$\partial D$} }such that 
\begin{equation}
 \label{2.2}
 x=\exp_{y_1}\left(\tau(y_1)N(y_1)\right)=\exp_{y_2}\left(\tau(y_2)N(y_2)\right).
\end{equation}
 We have $\tau(y_1)=\tau(y_2)$, and for $i=1,2$, the differential at $(\tau(y_i),y_i)$ of the map $\R_+\times \partial D\ni (t,y)\mapsto \exp_y(tN(y))$ is nondegenerate. The set $S$ is a codimension $1$ submanifold of $M$
 and $S'\backslash S$ has Hausdorff dimension smaller than or equal to~$d-2$. It is the 
 union of the focal set which is the set of points $x=\exp_y(\tau(y)N(y))$  such that  $(t,y')\mapsto \exp_{y'}(tN(y'))$ is degenerate at $(\tau(y), y)$, and the union of the sets defined like $S$ but with{strictly} more than two points $y_1$, $y_2$, {$y_3$,...} 
For $r\ge 0$, let
\begin{equation}
 \label{2.3}
 D(r)=\{ z\in D\backslash S', \  \rho_{\partial D} (z)\ge r\}. 
\end{equation}
{where $\rho$ is the Riemannian distance.}
The set $D(r)$ is a {(possibly empty)} manifold with smooth boundary $\partial D(r)$ on which one can define an inward normal $N(y)$ and an orientation by parallel transporting oriented basis of $\partial D$ along normal geodesics. So we have for all $y\in D\backslash S'$: $N(y)=\n \rho_{\partial D}(y)$. 

We will also need the sets $D(r)$ for all $r\in\RR$. We will let for $r<0$ 
\begin{equation}
 \label{2.3.1}
 D(r)=\{ z\in M, \  \rho^+_{\partial D} (z)\ge r\}
\end{equation}
where $\rho^+_{\partial D}$ is the signed distance to $\partial D$, positive inside $D$, negative outside $D$.

Define for $s, t\in \RR$ 
\begin{equation}
 \label{2.4}
 \begin{split}
 \psi(s,t) : \partial D(s)&\to \partial D(t)\\
 y&\mapsto \exp_y\left((t-s)N(y)\right)
 \end{split}
\end{equation}
and $\psi(t)=\psi(0,t)$. We will indifferentely write $\psi(t)(x)=\psi(t,x)$.
The function $\psi(s,t)$ is not defined for all points of $\partial D(s)$ because we ask $\psi(s,t)(y)\in \partial D(t)$,
 nor is $N(\cdot)$. However for $|s|$ and $|t|$ small it is a map, defined for all $y\in\partial D(s)$, and is is also a diffeomorphism with inverse $\psi(t,s)$.

 We have for $0\le s\le t$, 
$
\psi(t)=\psi(s,t)\circ\psi(s)$, which implies \begin{equation}\label{2.7}\det T\psi(t)=\det T\psi(s,t)\times \det T\psi(s).\end{equation} 
Notice that thanks to the orientation of the sets $\partial D(r)$ we get an orientation of $D\backslash S'$ by adding $N$ as first vector to oriented basis, consequently $\det T\psi$ is well defined and always positive. 
{It is well-known} that
\begin{equation}
 \label{2.5}
 \f{d}{dt}\Big|_{t=s}\det T\psi(s,t)(y)=-h(y)
\end{equation}
where $h(y)$ is the inward mean curvature of $\partial D(s)$ 
(the minus sign of the r.h.s.\ of \eqref{2.5} insures that $h$ is non-negative on $\pa D(s)$ when $D(s)$ is convex). This together with~\eqref{2.7} yields 
\begin{equation}
 \label{2.6}
 \f{d}{dt}{\Big|_{t=s}}\det T\psi(t)(y)=-h\left(\psi(s)(y)\right)\det T\psi(s)(y)
\end{equation}
and consequently, using $\psi(0)=\id$ and $\det T\psi(0)\equiv 1$,
\begin{equation}
 \label{2.8}
 \det T\psi(t)(y)=\exp\left(\int_0^t -h\left(\psi(s)(y)\right)\,ds\right).
\end{equation}
Denote by $\mu$ the volume measure of $D$ and by $\usm$ the volume measures of the manifolds $\partial D(s)$ and of $S$.
Then 
\begin{equation}
 \label{2.9}
 \mu(D)=\int_0^\infty \usm\left(\partial D(r)\right)\, dr.
\end{equation}
But for $r\ge 0$
\begin{equation}
 \label{2.10}
 \usm\left(\partial D(r)\right)=\int_{\partial D}\det T\psi(r)(y)\,\usm(dy)
\end{equation}
with convention $\det T\psi(r)(y)=0$ if $r\ge \tau(y)$. We get 
\begin{equation}
 \label{2.11}
 \usm\left(\partial D(r)\right)=\int_{\partial D}\exp\left(-\int_0^rh(\psi(s)(y))\,ds \right)1_{\{r<\tau(y)\}}\,\usm(dy)
\end{equation}
which yields with~\eqref{2.9}
\begin{equation}
 \label{2.12}
 \mu\left(D\right)=\int_{\partial D}\left(\int_0^{\tau(y)}\exp\left(-\int_0^rh(\psi(s,y))\,ds \right)\, dr\right)\,\usm(dy).
\end{equation}
More generally, for a measurable function $g : D\to \RR$ bounded below,
 \begin{equation}
 \label{2.13}
 \int_Dg\,d\mu=\int_{\partial D}\left(\int_0^{\tau(y)}g\left(\psi(r,y)\right)\exp\left(-\int_0^rh(\psi(s,y))\,ds \right)\, dr\right)\,\usm(dy).
\end{equation}
Applying this formula to the function $gh$ which we assume to be bounded below or integrable, we get by integration by parts
\begin{align*}
 \int_Dgh\,d\mu&=\int_{\partial D}\left(\int_0^{\tau(y)}-g\left(\psi(r,y)\right)\f{d}{dr}\exp\left(-\int_0^rh(\psi(s,y))\,ds \right)\, dr\right)\,\usm(dy)\\
 &=\int_{\partial D}\left[-g\left(\psi(r,y)\right)\exp\left(-\int_0^rh(\psi(s,y))\,ds\right)\right]_0^{\tau(y)}\,\usm(dy)\\
 &+\int_{\partial D}\left(\int_0^{\tau(y)}\langle dg,N\rangle\left(\psi(r,y)\right)\exp\left(-\int_0^rh(\psi(s,y))\,ds \right)\, dr\right)\,\usm(dy)\\
 &=\int_{\partial D}g(y)\,\usm(dy)-\int_{\partial D}g(\psi(\tau(y),y))e^{-\int_0^{\tau(y)}h(\psi(u,y))\,du}\,\usm(dy)\\&+\int_D\langle dg,N\rangle \, d\mu.
\end{align*}
Define the map 
\begin{equation}
 \label{2.14}
 \begin{split}
  \varphi :  \partial D &\to S'\\
  y&\mapsto \psi(\tau(y),y).
 \end{split}
\end{equation}
For $z=\psi(\tau(y_i),y_i)\in S$ ($i=1,2$) define $\theta(z){\in(0,\pi/2]}$ the angle 
between $N(\psi(\tau(y_i)-,y_i))$ and $S$. 
{In the sequel we assume that $\theta(z)\not=\pi/2$ (the case $\theta(z)=\pi/2$ is simpler to deal with and Proposition \ref{P1} is always valid).}
Notice that this angle does not depend on $i$, {this is a consequence of $z\in S$ staying at the same distance to $y_1$ and $y_2$ by infinitesimal variation}. For later use, let also $\theta(z)=0$ when $z\in S'\backslash S$. Let us prove that for $z=\psi(\tau(y_i),y_i))\in S$, 
\begin{equation}
 \label{2.15}
 \det T\psi (\tau(y_i),y_i)=\sin \theta(\varphi(y_i)) \det T\varphi(y_i),\qquad i=1,2.
\end{equation}
Set $y=y_1$. Let $e_1=N(y)$, $e_1^S=N(\psi(\tau(y)-,y))$, $N^S(z)$ the normal to $S$ at $z$ such that $\langle N^S(z),e_1^S\rangle>0$,  let $e''=(e_3,\ldots, e_d)$ be a family of  orthonormal normalized vectors in $T_y\partial D$ such that letting $\di e_2=\f{\n\tau(y)}{\|\n\tau(y)\|}$ {(we have $\na\tau(y)\not=0$, since $\theta(z)\not=\pi/2$)}, 
$e':=(e_2, e'')$ is an orthonormal basis of $T_y\partial D$, let $(e^S)''=(e_3^S,\ldots , e_d^S)$ be an orthonormal basis of $T_y\varphi({\rm Vect}(e''))$,  let $e_2^S$ such that $(e^S)':=(e_2^S,\ldots , e_d^S)$ is an orthonormal basis of $T_zS$. Finally let $e_2^\theta\in T_zM$ be such that $\langle e_2^\theta,{N(z)}\rangle <0$ {($e_2^\theta$ and $N^S(z)$ are not orthogonal, since $\theta(z)\not=\pi/2$)} and $(e_1^S, e_2^\theta, (e^S)'')$ is an orthonormal basis of $T_zM$. 
{Figure \ref{fig1} shows the configuration of $e_1^S, N^S(z), e_2^S$ and $e_2^{\theta}$ on an example of dimension 2.}
In the sequel we will denote for instance $\di T\varphi(e')=\left(\begin{matrix}T\varphi(e_2)\\ \vdots\\ T\varphi(e_d)\end{matrix}\right)$, so that $\langle T\varphi(e'), (e^S)'\rangle$ will be the matrix of all scalar products. 
We have
\begin{align*}
& \langle T\varphi(e'), (e^S)'\rangle \\&=\langle d\tau,e'\rangle\langle \partial_t\psi(\tau(y),y), (e^S)'\rangle+ \langle T\psi(e'),(e^S)'\rangle\\
 &=\left(\begin{matrix}
          \langle d\tau,e_2\rangle\langle \partial_t\psi, e_2^S\rangle+\langle T\psi(e_2), e_2^S\rangle&
          \langle T\psi(e_2),(e^S)''\rangle\\
          \langle d\tau,e''\rangle\langle\partial_t\psi, e_2^S\rangle+\langle T\psi (e''), e_2^S\rangle
          & \langle T\psi(e''),(e^S)''\rangle
         \end{matrix}
\right).
\end{align*}
Let us simplify and make more explicit this expression. 
We have $\langle d\tau,e''\rangle=0$. Also $e_2^\theta\perp (e^S)''$ and $e_2^S\perp (e^S)''$ so $e_2^S\in {\rm Vect}(e_1^S, e_2^\theta)$ and more precisely 
\begin{equation}
 \label{2.16}
 e_2^S=\cos(\theta(z)) e_1^S+\sin(\theta(z))e_2^\theta.
\end{equation}
\begin{figure}[!h]
 \includegraphics[width=10cm]{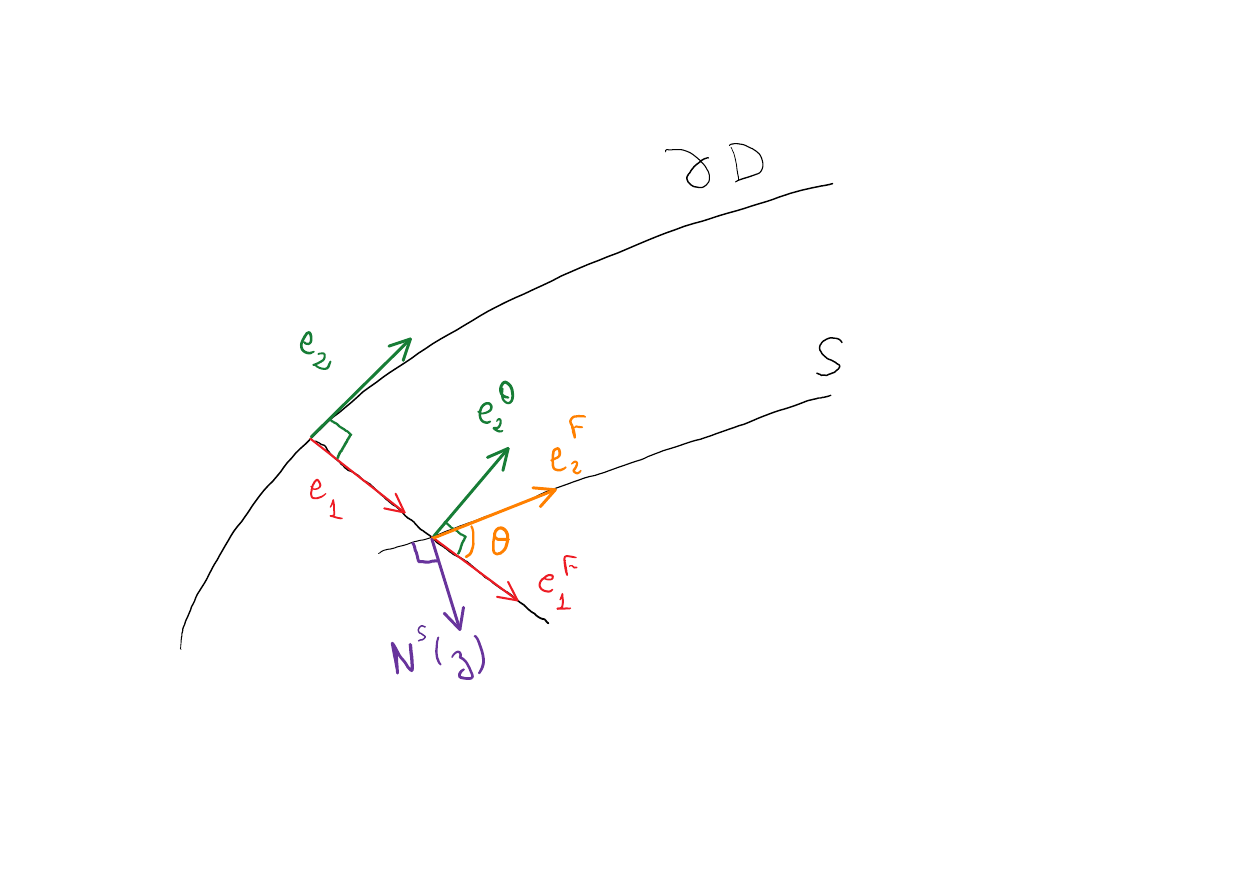}
   \caption{ The vectors $e_1^S, N^S(z), e_2^S$ and $e_2^{\theta}$}\label{fig1}
\end{figure}
On the other hand $T\psi(e')\perp e_1^S$ which implies 
\begin{equation}
 \label{2.17}
 \langle T\psi(e'),e_2^S\rangle =\sin(\theta(z))\langle T\psi(e'), e_2^\theta\rangle. 
\end{equation}
Also $\langle \partial_t\psi, e_2^S\rangle =\cos(\theta(z)).$ We arrive at 
\begin{equation}
 \label{2.18}
 \begin{split}
& \det\langle T\varphi(e'),(e^S)'\rangle\\&=\sin\theta(z)\det\left(
 \begin{matrix}
  \langle T\psi(e_2), e_2^\theta\rangle &\langle T\psi(e''),e_2^\theta\rangle\\
  \langle T\psi(e_2),(e^S)''\rangle & \langle T\psi(e''),(e^S)''\rangle 
 \end{matrix}
 \right)\\&+\cos\theta(z)\det\left(\begin{matrix}
                                 \langle d\tau, e_2\rangle &{0}\\
                                 \langle T\psi(e_2), (e^S)''\rangle & \langle T\psi(e''), (e^S)''\rangle
                                \end{matrix}
\right)\\
&= \sin\theta(z)\det T\psi+\cos\theta(z) \langle d\tau, e_2\rangle \det \langle T\psi(e''),(e^S)''\rangle.
\end{split}
\end{equation}
For the last equation we used the fact that  $\det T\psi= \det \langle T\psi(e'), (e_2^\theta, (e^S)'')\rangle$, {since $e'$ and $(e_2^\theta, (e^S)'')$
are orthonormal bases. Note that by definition, $\langle T\psi(e''),e_2^\theta\rangle=0$,
so we also get $\det T\psi=\det \langle T\psi(e''),  (e^S)''\rangle\times \langle T\psi(e_2), e_2^\theta\rangle$.}
On the other hand, we have
\begin{equation}
 \label{2.19}
 \langle d\tau,e_2\rangle =\langle T\psi (e_2), e_2^\theta\rangle \cot\theta{(z)}.
\end{equation}
{Indeed, note that
\bq
0&=&\lan T\varphi(e_2), N^S\ran\\
&=&\lan d\tau, e_2\ran \lan e_1^S, N^S\ran+\lan T\psi(e_2), N^S\ran\\
&=&\lan d\tau, e_2\ran \sin(\theta(z))-\cos(\theta(z))\lan T\psi(e_2), e_2^\theta\ran\eq
where the last term is obtained by taking into account that $ T\psi(e_2)$ is parallel to $e_2^\theta$.}
This is the change of length of the geodesic needed to stay in $S$. 
We obtain 
\begin{align*}
 \det T\varphi&=\sin\theta(z)\det T\psi +\cos\theta(z)\cot\theta(z)\det T\psi\\
 &=\f{\sin^2\theta(z)+\cos^2\theta(z)}{\sin\theta(z)}\det T\psi.
\end{align*}
This yields~\eqref{2.15}.

We arrived at 
\begin{equation}
 \label{2.20}\begin{split}
 \int_Dgh\,d\mu&=\int_{\partial D}g(y)\,\usm(dy)-\int_{\partial D}g(\psi(\tau(y),y))\det T\psi(\tau(y),y)\,\usm(dy)\\&+\int_D\langle dg,N\rangle \, d\mu.
 \end{split}
\end{equation}
this yields with~\eqref{2.15}
\begin{equation}
 \label{2.21}\begin{split}
 \int_Dgh\,d\mu&=\int_{\partial D}g(y)\,\usm(dy)-\int_{\partial D}g(\varphi(y))\sin\theta(\varphi(y))\det T\varphi(y)\,\usm(dy)\\&+\int_D\langle dg,N\rangle \, d\mu.
 \end{split}
\end{equation}
Using the change of variable $y\mapsto\varphi(y)$ and the fact that all $z\in S$ is equal to $\varphi(y_i)$, $i=1,2$, we obtain the key formula
\begin{prop}
 \label{P1}
 {With the above notations, for any smooth function $g$ defined on $D$ {such that $gh$ is integrable or bounded below}, we have:}
 \begin{equation}
 \label{2.22}\begin{split}
 \int_Dgh\,d\mu&=\int_{\partial D}g(y)\,\usm(dy)-2\int_{S}g(z)\sin\theta(z)\,\usm(dz)+\int_D\langle dg,N\rangle \, d\mu.
 \end{split}
\end{equation}
\end{prop}

\section{Moving sets}
\label{Section3}
\setcounter{equation}0

In this section we describe how to move a domain with smooth boundary by deformation of its boundary. We will investigate the deformation of its skeleton The deformation we will consider will have a general absolutely continuous finite variation part, together with a very specific martingale part and singular finite variation part. First we introduce some notation. 

For a domain $D$ with smooth boundary $\partial D$, $s\in \RR$, define 
\begin{equation}
 \label{3.1}
 \begin{split}
 \psi^D(s)= \psi^D(0,s) : \partial D&\to \partial D(s)\\
  y&\mapsto \psi^D(s)(y)=\psi^D(s,y)=\exp_y\left(sN^D(y)\right).
 \end{split}
\end{equation}
Here $N^D=N$ is the inward normal defined in {S}ection~\ref{Section2}.
Consider a moving domain $t\mapsto D_t$. {Be careful not to confound $D(t)$ with $D_t$, since in general they are quite different subsets.} We first assume that the deformation is sufficiently regular so that 
for all $0\le s\le t$, we can write $D_t$ as 
\begin{equation}
 \label{3.2}
 D_t=\left\{\psi^{D_s}([Z_t^{D_s}(y),\tau_{D_s}(y)],y), \quad y\in \partial D_s\right\}. 
\end{equation}
{In particular, we must have $S'_s\subset D_t$.}
Notice that in the special case where the real valued function $t\mapsto Z_t^{D_s}(y)$ does not depend on $y$, {for any $0\leq s\leq t$,} then we have 
\begin{equation}
 \label{3.3}
 D_t=D_s(Z_t^{D_s})=D_0(Z_t^{D_0}), \qquad Z_t^{D_0}=Z_t^{D_s}+Z_s^{D_0}
\end{equation}
where $D(r)$ is defined in~\eqref{2.3}, replacing distance to $\partial D$ by signed distance with positive sign inside $D$ and negative sign outside. In this situation, the skeleton is not moving, at least as long as $\pa D_t$ remains smooth {(i.e.\ until $\pa D_t$ hits $S_0'$ {or is too far outside $D_0$}),} and $t\mapsto Z_t^{D_0}$ can be allowed to be a semimartingale with singular continuous drift. 

When $t\mapsto Z_t^{D_s}(y)$ depends on $y$ the situation is a little bit more complicated. Starting from $(t,y)\mapsto Z_t^{D_0}(y)$ which is assumed to be defined on $[0,\e)\times \partial D_0$, the sets $D_t$ are defined for $0\le t<\e$, as well as the $Z_t^{D_s}(y)$, ${0\leq\, }s\le t$, $y\in D_s$. In fact, if $(y,t)\mapsto Z_t^{D_0}(y)$ is $C^1$, then one can reconstruct all $Z_t^{D_s}(y)$ with the only knowledge of $\dot Z_t^{D_t}(z)$, $z\in\partial D_t$. Let us do it for $s=0$: the map $(t,y)\mapsto \psi^{{D_0}}(t,y)$ from $(-\a,\a)\times \partial D_0$ to $M$ is a diffeomorphism on its range, for $\a{\,>0}$ sufficiently small. Let us denote $z\mapsto (\tau_0(z), \varphi_0(z))$ its inverse. Then a variation  $z+N^{D_t}(z)dZ_t^{D_t}$ corresponds to a variation $(\tau_0(z), \varphi_0(z))+(d\tau_0,T\varphi_0)N^{D_t}(z)dZ_t^{D_t}$ of the coordinates in $(-\a,\a)\times \partial D_0$.
 But this is not convenient at all, since it is not intrinsic.  {Moreover, when passing to stochastic processes and Stratonovich equations, it will involve} second derivatives of $z\mapsto (\tau_0(z), \varphi_0(z))$. So we prefer to leave the reference to $D_0$ and to always stay at the level of the moving $D_t$. 

For all $y\in \partial D_0$ we define a stochastic process $t\mapsto Y_t(y)$ representing the motion of $D_t$ satisfying $Y_0(y)=y$ and the It\^o equation in manifold with respect to the Levi Civita connection $\n$
\begin{equation}
 \label{3.4}
 dY_t(y)=d^\n Y_t(y)=\partial_1\psi^{D_t}(\cdot,Y_t(y))(dZ_t^{D_t}(Y_t(y)))=N^{D_t}(Y_t(y))dZ_t^{D_t}(Y_t(y)).
\end{equation}
Recall that fomally $d^\n Y_t(y)$ is a vector which writes in local coordinates $(y^1,\ldots, y^d)$ with the Christoffel symbols $\Gamma_{j,k}^i$:
\begin{equation}
\label{3.4.1}
d^\n Y_t(y)=\left(dY_t^i(y)+\f12 \Gamma_{j,k}^i(Y_t(y))\, d\langle Y_t^j(y),Y_t^k(y)\rangle \right)D_i(Y_t(y))
\end{equation}
where $D_i(Y_t(y))$ is the vector $\di \f{\partial}{\partial y^i}$ taken at point $Y_t(y)$.
We will always assume that the martingale part $dm_t$ of $dZ_t^{D_t}(y)$ does not depend on $y$. In this situation, the It\^o equation is equivalent to the Stratonovich one: indeed, using~\eqref{3.3}  the It\^o to Stratonovich convertion term is  $$ \f12\n_{N^{D_t}(Y_t(y))dm_t}N^{D_t}(\cdot)dm_t=\f12\n_{N^{D_t}(Y_t(y))}N^{D_t}(\cdot)d\langle m,m\rangle_t=0
$$
since $N^{D_t}(Y_t(y))$ is the speed at time $a=0$ of the geodesic $a\mapsto \psi^{D_t}(a)(Y_t(y))$.

More precisely, we will let $dZ_t^{D_t}(y)$ be of the form 
\begin{equation}
\label{3.5}
dZ_t^{D_t}(y)=H^{D_t}(Y_t(y))\,dt +dz_t
\end{equation}
where $H^{D_t}$ is a smooth function on $\partial D_t$ (which later on will be chosen to be $h^{D_t}/2$, where $h^{D_t}$ is the mean curvature of $\pa D_t$) and {$(z_t)_{t\geq 0}$} is a real valued continuous semimartingale. {We assume that Equation~\eqref{3.4} has a strong solution up to some positive stopping  time.}
 Moreover, since $dY_t(y)$ represents the motion of $\partial D_t$ and for small time the map $y'\mapsto Y_t(y')$ is a diffeomorphism from $\partial D_0$ to $\partial D_t$, writing $Y_t(y')=y$, equation~\eqref{3.4} rewrites as 
\begin{equation}
\label{3.6}
d\partial D_t(y){\df}dY_t(y')=N^{D_t}(y)\left(H^{D_t}(y)\, dt +dz_t\right).
\end{equation}

Let us now investigate the motion of the skeleton $S_t$ under this motion of $D_t$. First we remark that by local inversion theorem, at regular points of the skeleton, the variation {in Stratonovich sense} is linear and the sum of all variations 
of the concerned point at the boundary. 
 As we already remarked, the motion $dz_t$ does not change $S_t${, so this together with the linearity just mentioned implies that we have a finite variation of the skeleton}. 

Recall the situation of~\eqref{2.2} in Section~\ref{Section2}. We consider a domain $D$, $x\in S$, $y_1,y_2$ the two elements of $\partial D$ such that $\exp_{y_1}\left(\tau(y_1)N(y_1)\right)=\exp_{y_2}\left(\tau(y_2)N(y_2)\right)$, with $\tau(y_1)=\tau(y_2)$.
For $i=1,2$, we will consider a variation of the minimal geodesic from $y_i$ to $x$, represented by a Jacobi field $J_i$ satisfying $J_i(0)\in T_{y_i}M$, $J_1(1)=J_2(1)\in T_xM$, 
\begin{equation}
 \label{3.11}
 J_i(0)=\l_i N(y_i)+J_i^\perp(0), \quad J_i'(0)=\l_i'N(y_i)+(J_i^\perp)'(0),
\end{equation}
with $J_i^\perp$ orthogonal to $N(y_i)$. The motion of $S$ corresponding to the motion of $y_1$ and $y_2$ will be represented by $J_1(1)$. Since $S$ has a boundary, the observation of the orthogonal part {to $S$ of $J_1(1)$} is not sufficient. 
 
Let $\g_i$ be the projection on $M$ of $J_i$. It is the geodesic in time $1$ from $y_i$ to $x$ {(as usual in the computations of Jacobi fields, the speed is not normalized)}. Denote {$N_i(x)=\dot \g_i(1)/\|\dot\g_i(1)\|$}. Recall that the angle between $N_i(x)$ and $T_xS$ is $\theta(x)\in (0,\pi/2]$. We will also let 
\begin{equation}
 \label{3.12}
 N_1^S(x)=\f{1}{2\sin\theta(x)}(N_1(x)-N_2(x)).
\end{equation}
Figure \ref{fig2} shows the configuration of the points $x,y_1,y_2$ and the vectors $N_1(x)$, $N_2(x)$, $N_1^S(x)$.
\begin{figure}[!h]
  \includegraphics[width=10cm]{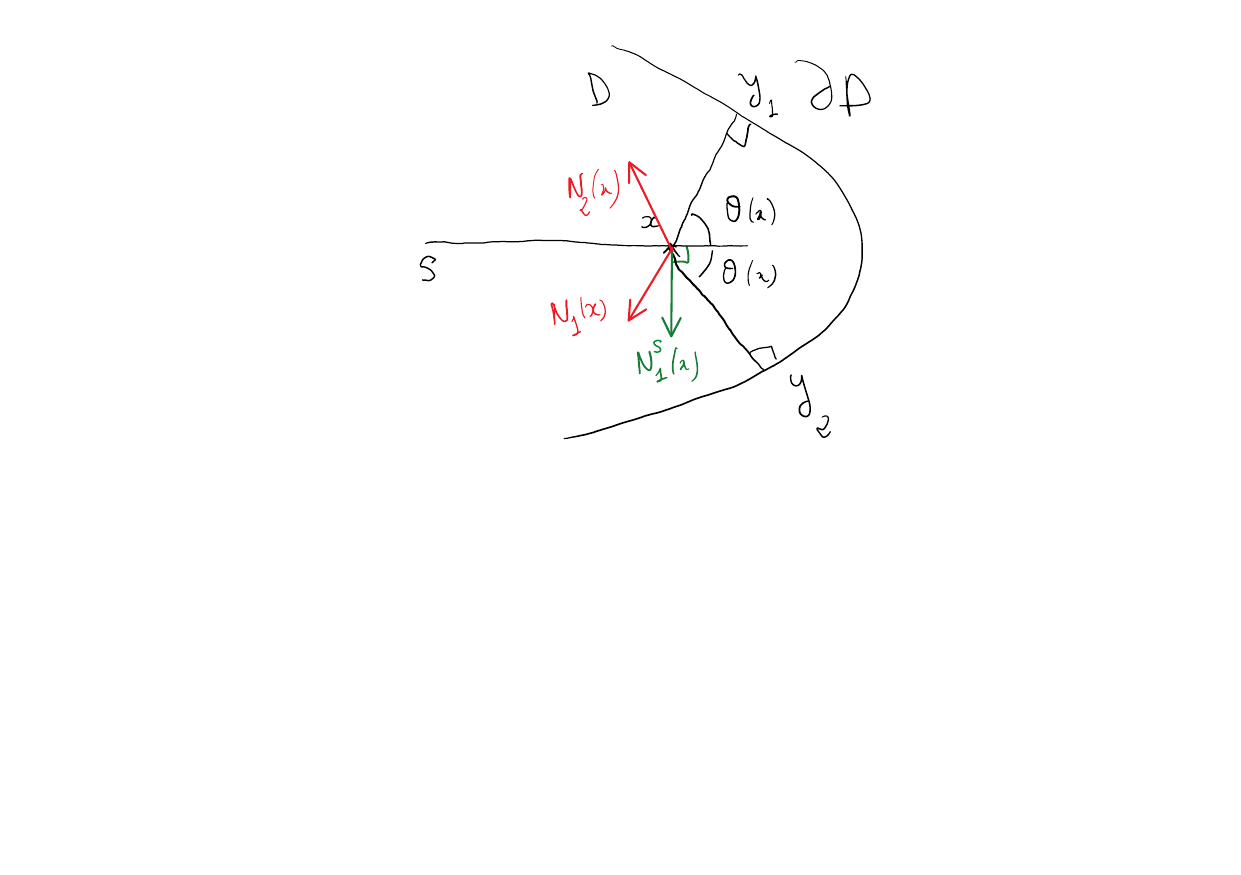}
   \caption{The points $x,y_1,y_2$ and the vectors $N_1(x), N_2(x), N_1^S(x)$}\label{fig2} 
\end{figure}
The vector
$N_1^S(x)$ is
 is the normal vector to $S$ at point $x$, in the same side as $N_1(x)$. We will  consider variations of geodesics with same  final value: 
\begin{equation}
 \label{3.13}
 J_1(1)=J_2(1)=\l N_1^S(x)+ J_1^T(1)
\end{equation}
for some $\l\in \RR$, where $J_1^T(1)\in T_xS$. Writing $\l N_1^S(x)=\f{\l}{2\sin\theta(x)} (N_1(x)-N_2(x))$ we have 
\begin{equation}
 \label{3.14}
 \begin{split}
 \langle J_1(1), N_1(x)\rangle &= \f{\l}{2\sin\theta(x)}\left(1-\cos(2\theta(x)\right)+ \langle J_1^T(1), N_1(x)\rangle\\&=\l\sin\theta(x)+ \langle J_1^T(1), N_1(x)\rangle
 \end{split}
\end{equation}
and 
\begin{equation}
 \label{3.15}
 \begin{split}
 \langle J_1(1), N_2(x)\rangle &= -\f{\l}{2\sin\theta(x)}\left(1-\cos(2\theta(x)\right)+ \langle J_1^T(1), N_{{2}}(x)\rangle\\&=-\l\sin\theta(x)+ \langle J_1^T(1), N_2(x)\rangle
 \end{split}
\end{equation}
On the other hand we require that the variation of length of the two geodesics are the same. This writes as
\begin{equation}
 \label{3.16}
 \langle J_1(1), N_1(x)\rangle -\langle J_1(0), N(y_1)\rangle=\langle J_2(1), N_2(x)\rangle -\langle J_2(0), N(y_2)\rangle
\end{equation}
or 
\begin{equation}
 \label{3.17}
\l\sin\theta(x)+ \langle J_1^T(1), N_1(x)\rangle-\l_1=-\l\sin\theta(x)+\langle J_1^T(1), N_2(x)\rangle-\l_2,
\end{equation}
which finally, with $\langle J_1^T(1), N_1(x)-N_2(x)\rangle=0$, yields $\di \l=\f{\l_1-\l_2}{2\sin\theta(x)}$, so the normal variation of $S$ is given by 
\begin{equation}
 \label{3.18}
\langle J_1(1), N_1^S(x)\rangle N_1^S(x)=\f{\l_1-\l_2}{2\sin\theta(x)}N_1^S(x).
\end{equation}
Next we will compute the tangential displacement $J^T(1)$ of $x$ in $S$. As we will see later, we will only need a Jacobi field $J_1$ such that $J_1^\perp(0)$
and  $(J_1^\perp)'(0)$ are known and
\begin{equation}
\label{3.18.1}
J_1(0)=\l_1N(y_1), \ \hbox{{i.e.}}\quad J_1^\perp(0)=0
.\end{equation}
So we know $J_1^\perp(1)$: and 
\begin{equation}
\label{3.18.2}
J_1^\perp(1)=J\left(1,0,(J_1^\perp)'(0)\right)
\end{equation}
where $J(1,u,v)$ is the value at time $1$ of the Jacobi field $J$ with $J(0)=u$ and $J'(0)=v$.
From
\begin{equation}
 \label{3.18.3}
 \begin{split}
  J_1(1)&=J_1^T(1)+\langle J_1(1),N_1^S(x)\rangle N_1^S(x)\\
  J_1(1)&=J_1^\perp(1)+\langle J_1(1),N_1(x)\rangle N_1(x)
 \end{split}
\end{equation}
we get 
\begin{equation}
 \label{3.18.4}
 J_1^T(1)=J_1^\perp (1)+\langle J_1(1),N_1(x)\rangle N_1(x)-\langle J_1(1),N_1^S(x)\rangle N_1^S(x).
\end{equation}
On the other hand we have
\begin{equation}
 \label{3.18.5}
 \begin{split}
  \langle J_1(1),N_2(x)\rangle &=\langle J_1^\perp(1),N_2(x)\rangle +\langle J_1(1),N_1(x)\rangle\langle N_1(x),N_2(x)\rangle\\
  \langle J_1(1),N_2(x)\rangle &=\langle J_1(1),N_1(x)\rangle -(\l_1-\l_2)
 \end{split}
\end{equation}
where the second equation is a direct consequence of~\eqref{3.18}.
Substracting the second equation to the first one yields
\begin{equation}
 \label{3.18.6}
 (1-\cos(2\theta(x)))\langle J_1(1),N_1(x)\rangle=\langle J_1^\perp(1),N_2(x)\rangle+\l_1-\l_2.
\end{equation}
Replacing $\langle J_1(1),N_1(x)\rangle$ in~\eqref{3.18.4} and after simplification, using~\eqref{3.12} {and \eqref{3.18},} we finally obtain the horizontal displacement
\begin{equation}
 \label{3.19}
 (J_1^T)(1)=J_1^\perp(1)+\f1{4\sin^2\theta(x)}\left({2}\langle J_1^\perp(1),N_2(x)\rangle {N_1(x)}+(\l_1-\l_2)(N_1(x)+N_2(x))\right).
\end{equation}

We are now in position to write the motion of the skeleton $S_t$ when the motion of the boundary is given by~\eqref{3.6}. For $x\in S_t$ with corresponding points $y_1$ and $y_2$ in $\partial D_t$, 
\begin{equation}\label{3.10}
dS_t^\perp (x)=\f{1}{2\sin\theta^{S_t}(x)}\left(H^{D_t}(y_1)-H^{D_t}(y_2)\right)N_1^{S_t}(x)\, dt
\end{equation}
which has finite variation. Observe that, as already mentioned, the term $dz_t$ disappears. 

Here we wrote $dS_t^\perp (x)$ for the normal variation of the regular skeleton. But as we already remarked, since $S_t$ is not a closed manifold, it can expand via the motion of its boundary. So we have to investigate the horizontal motion $dS^T(x)$. 

{Notice that $J_1^\perp)'(0)$ is the perpendicular part of the time derivative of the speed at $y_1$ of the geodesic in time~$1$ from~$y_1$ to~$x$. So 
 from} equation~\eqref{3.6} we deduce the rotation 
 \begin{equation}\label{3.20}
(J_1^\perp)'(0)\,dt= {{\rho_S(y_1)}}\n_{t}N^{D_t}(y_1)=-{{\rho_S(y_1)}}\n H^{D_t}(y_1)\,dt.
 \end{equation}
{(in the r.h.s.\ the gradient corresponds to the tangential gradient on $\pa D_t$, recall that $H^{D_t}$ is only defined on this hypersurface).}\par
 We conclude that the horizontal displacement of $x$ is $J^T_1(1)\,dt$ 
 \begin{equation}
 \begin{split}
  \label{3.21}
  &J^T_1(1)\,dt=J_1^\perp(1)\,dt+\f1{4\sin^2\theta^{S_t}(x)}\Biggl(2\langle J_1^\perp(1),N^{D_t}_2(x)\rangle N^{D_t}_1(x)\\&+(H^{D_t}(y_1)-H^{D_t}(y_2))(N^{D_t}_1(x)+N^{D_t}_2(x))\Biggr)\,dt 
  \end{split}
 \end{equation}
 where $J_1^\perp(1)=J(1,0,-{\rho_S(y_1)}\n H^{D_t}(y_1))$.
Again the processus $z_t$ does not play a role. 

{To summarize, we have the following result for the evolution of $S_t$:
\begin{thm}
\label{T3.1}
When $D_t$ evolves as~\eqref{3.6}
\begin{equation}
\label{3.7.bis}
d\partial D_t(y)=N^{D_t}(y)(H^{D_t}(y)\, dt +dz_t),
\end{equation}
the regular skeleton $S_t$ has the normal evolution~\eqref{3.10}
\begin{equation}
\label{3.10.bis}
dS_t^\perp (x)=\f{H^{D_t}(y_1)-H^{D_t}(y_2)}{4\sin^2\theta^{S_t}(x)}\left(N_1^{D_t}(x)-N_2^{D_t}(x)\right)\, dt
\end{equation}
and the tangential evolution~\eqref{3.21} which can be rewritten as 
\begin{equation}
\label{3.22}
\begin{split}
&dS_t^T(x)\\
&=p_S(J_1^\perp(1))\, dt\\
&+\left(-\f{\langle J_1^\perp(1), N_1^S(x)\rangle}{2\sin\theta^{S_t}(x)}+\f{H^{D_t}(y_1)-H^{D_t}(y_2)}{4\sin^2\theta^{S_t}(x)}\right)(N^{D_t}_1(x)+N^{D_t}_2(x))\,dt
\end{split}
\end{equation}
where $p_S$ denotes the orthogonal projection on $TS$,  $J_1^\perp(1)=J(1,0,-{\rho_S(y_1)}\n H^{D_t}(y_1))$, and $y_1$, $y_2$ are defined in Figure~\ref{fig2}.
\end{thm}
\begin{remark}
\label{R3.1}
The points $y_1$ and $y_2$ do not play the same role in Theorem~\ref{T3.1}.  As formula~\eqref{3.10.bis} is symmetric in $y_1$ and $y_2$, formula~\eqref{3.22} is not. The reason is that if we assume the motion of $y_1$ to be normal to the boundary $\partial D_t$ and to have speed given by~\eqref{3.7.bis}, the motion of $y_2$ has no reason to be normal to the boundary: $J_2^\perp(0)$ does not vanish.
\end{remark}
}
\section{Doss-Sussman representation of It\^o's equation    \eqref{6.5mod0} }
\label{appC}
\setcounter{equation}0

In this section we adapt the results of \cite{zbMATH07470497} to our notations.
Let  the stochastic mean curvature flow be  a solution of :
\bqn{sflow}
\fo t\in[0,\tau),\,\fo y\in C_t, \qquad d \partial D_t(y)&=&\lt(dW_t+\f12 h^{D_t}(y)dt\rt)N^{D_t}(y)
\eqn
where  $C_t\df\pa D_t$, starting at $D_0$.

Let $\partial G_t$ be a solution of
\bqn{sflow2_aux}
 \lt\{\begin{array}{rcl}
G_0&=&D_0\\
\fo t\in[0,\wi\epsilon),\,\fo x\in \partial G_t, \qquad \pa_t x&=&\alpha_{\pa G_t,-W_t}(x)N^{G_t}(x)
\end{array}\rt.\eqn
for some $\wi\epsilon>0$ small enough, where $\alpha$ is defined by
\bqn{alpharhobeta}
\fo r>0,\,\fo D\in\cD_r,\,\fo x\in C,\qquad \alpha_{C,r}(x)&\df& \f12h^{\Psi(C,r)}(\psi_{C,r}(x))\eqn
and $\Psi(C,r)$  is the normal (exterior) flow starting at $C$ at time $r$ (c.f. Chapter 3 and 4 of \cite{zbMATH07470497} for notations).

Similarly to the proof of Theorem 17 from \cite{zbMATH07470497}, we  show that $D_t = \Psi(G_t,-W_t)$ is a solution of the stopped martingale problem associated to the generator $(\mathcal{D},\wi{\mathcal{L}})$ where for $f \in C^{\infty}(M) $ and $ \mathbb{F}_f (D) = \int_D f \,d\mu$, $\nu = -N $ is the exterior normal
$$ \wi{\mathcal{L}} \mathbb{F}_f (D) := \frac12 \int_{\partial D} \langle \nabla f , \nu \rangle \,d\underline{\mu} = \mathbb{F}_{ \frac12 \Delta f}(D). $$

Recall that  the equation  \eqref{sflow2_aux}, is in fact a quasiparabolic equation with coefficients that  depend on  trajectory of the Brownian motion (the meaning is trajectory by trajectory). Similarly to Section 4.1 from \cite{zbMATH07470497}, we show that the solution of  \eqref{sflow2_aux} have a regularity $ C^{1+\frac{\alpha}{2},2 + \alpha}$, for all $\alpha < 1. $

\begin{prop}
Let $ \partial{G}_t$ be a solution of \eqref{sflow2_aux}. Then $\partial D_t =  \Psi( \partial G _ t , -W_t)$ is a solution of \eqref{sflow} in the It\^o sense. 
\end{prop}
\begin{proof}
Let $ x \in \Psi( \partial G _ t , -W_t)$, we have :
\begin{equation}
\begin{split}
&d \Psi( \partial G _ t , -W_t) (x) = \\
& = T_1\Psi_{( \partial G _ t , -W_t)} (\frac{d}{dt} \partial G_t) (\Psi^{-1}(\partial G_ t,-W_t  )(x) \,d t \\
&- \nu^{\Psi( \partial G _ t , \-W_t)}(x) dW_t \\
& =\lt(dW_t+\f12h^{\Psi( \partial G _ t , -W_t)}(x)dt\rt)N^{\Psi( \partial G _ t , -W_t)}(x),
\end{split}
\end{equation}
where in the first equality we use the It\^o formula, the fact that $ t \mapsto \partial G_t $ is $C^{1 + \frac{\alpha}{2}} $, $\frac{d^2}{d^2r} \Psi(x,r) = 0$, and in the second equality we used Lemma 13 in \cite{zbMATH07470497}, i.e. 
$\partial D_t$ is a solution in 
the It\^o form :
\bqn{Ito}
 \lt\{\begin{array}{rcl}
 d \partial D_t (x) &=& (dW_t+\f12h^{\partial D_t}(x)dt)N^{\partial D_t}(x)\\
 x &\in & \partial D_t .\\
 \end{array}
 \rt.
 \eqn

\end{proof}
\begin{prop}
Conversely, if $\partial D_t $ is  a solution of \eqref{Ito} then  $\partial G _ t = \Psi (\partial D_t, W_t) $ is a solution of \eqref{sflow2_aux}.
\end{prop}

\begin{proof}
Let $x \in \partial  \Psi (\partial D_t, W_t)  $
\begin{equation}
\begin{split}
&d \Psi( \partial D _ t ,  W_t) (x) \\ 
&= T_1\Psi_{( \partial D _ t , W_t)} (\circ d \partial D_t) (x) +\nu^{\Psi( \partial D _ t , W_t)}(x) dW_t \\
&=    T_1\Psi_{( \partial D _ t , W_t)} ( (dW_t+\f12h^{\partial D_t}dt)N^{\partial D_t}) (x) \\
&-N^{\Psi( \partial D _ t ,W_t)}(x) dW_t \\
&= \lt(\f12h^{\partial D_t}(  \Psi^{-1}( \partial D_t , W_t)(x)  )N^{\partial G_t}(x) dt\rt) \\
&=\f12h^{\Psi( \partial G_t , -W_t) }(  \Psi( \partial G_t , -W_t)(x)  )N^{\partial G_t}(x) dt \\
\end{split}
\end{equation}
where we use that in this case, the Stratonovich differential is equal to the It\^o's one (c.f. Appendix \ref{Section3}),  i.e. $\circ d \partial D_t(x) =  d \partial D_t  $, and $\frac{d^2}{d^2r} \Psi(x,r) = 0$.
So $\partial G _t$ is a solution of \eqref{sflow2_aux}.
\end{proof}

By the uniqueness of the solution of \eqref{sflow2_aux} (c.f. Theorem~22 in~\cite{zbMATH07470497}) and the fact that it is adapted to the filtration of $B $  we deduce that the solution of \eqref{Ito} is unique and is a strong solution. Similarly we have the uniqueness of the solution of 
$$
  d \partial D_t (x) =\left(dW_t+\f12h^{\partial D_t}(x)dt - \frac{\underline{\mu}( \partial D_t)}{\mu(D_t)} dt\right)N^{\partial D_t}(x).
$$
 Moreover, since we could also make a change of time in the It\^o equation, Equation \eqref{6.5mod0} has a unique strong solution. 

\section{Weak semi-group theory in the martingale problem sense}\label{appD}
\setcounter{equation}0

This theory has been developed in several books, see for instance Stroock and Varadhan \cite{MR2190038} or Ethier and Kurtz \cite{MR838085}.
Here we present a minimal version suitable for our purposes.\par\me
Let $V$ be a measurable state space and consider $\Omega$ a set of trajectories from $\RR_+$ to $V$.
The canonical coordinates on $\Omega$ are denoted by the $X_t$, for $t\geq 0$: for $\omega\in\Omega$, $X_t(\omega)$ is the position at time $t$ of $\omega$.
The set $\Omega$ is endowed with the sigma-field generated by the $X_t$, for $t\geq 0$.
Our first assumption is that the mapping
\bq
\Omega\times\RR_+\ni(\omega,t)&\mapsto & X_t(\omega)\in V\eq
is measurable, which usually means that ``$\Omega$ is not too big''.
\par
For $t\geq 0$, we define \bq
\cF_t&\df&\sigma( X_s\st s\in[0,t])\eq
\par
For $t\geq 0$, we will also need the time shift $\Theta_{t}$  associating to any $\omega\in\Omega$ the trajectory $\Theta_{t}(\omega)$ defined by
\bq
\fo s\geq 0,\qquad X_s(\Theta_{t}(\omega))&=&X_{s+t}(\omega)\eq
\par
We assume that $\Theta_t(\Omega)\subset \Omega$.
\par
A given family $\PP\df (\PP_x)_{x\in V}$ of probability measures on $\Omega$ is said to be \textbf{Markovian} if for any $x\in V$ and any $t\geq 0$, the image by $\Theta_t$ of  $\PP_x$ conditioned by 
$\cF_t$ is $\PP_{X_t}$. In particular, it is assumed that $\PP$ has the regularity of a Markov kernel from $V$ to $\Omega$.
\par
From now on, we suppose that a Markovian family  $\PP$ is given.
Let $\cB$ be the space of bounded and measurable functions defined on $V$.
The \textbf{semi-group} $P\df (P_t)_{t\geq 0}$ associated to $\PP$ is the family of operators acting on $\cB$ via
\bq
\fo t\geq 0,\,\fo f\in\cB,\,\fo x\in V,\qquad P_t[f](x)&\df&\EE_x[f(X_t)]\eq
\par
The Markovianity of $\PP$ implies at once the semi-group property
\bq
\fo s,t\geq 0,\qquad P_tP_s&=&P_{t+s}\eq
and in particular the elements of $P$ commute.\par
A subclass of ``regular'' functions that will be important for our purposes is $\cR$ defined as
\bq
\cR&\df&\lt\{f\in \cB\st \fo x\in V,\, \lim_{t\ri 0_+}P_t[f](x)\,=\, f(x)\rt\}\eq
Exceptionally in the above limit, we assumed that $t\geq 0$ (i.e.\ not only that $t>0$), so that by definition, for any $f\in\cR$ and $x\in V$,
$P_0[f](x)=f(x)$.\par
Let us observe that $\cR$ is left stable by the semi-group:
\begin{lemma}\label{lem1}
For any $t\geq 0$, we have $P_t[\cR]\subset\cR$. Thus for any given $f\in\cR$ and $x\in V$, the mapping
\bq
\RR_+\ni t&\mapsto& P_t[f](x)\eq
is right continuous.
\end{lemma}
\begin{proof}
Indeed, fix $t\geq 0$ and $f\in\cR$, we have for any $x\in V$ and $s\geq 0$, 
\bq
P_s[P_t[f]](x)&=&P_t[P_s[f]](x)\\
&=&\EE_x[ P_s[f](X_t)]]\eq
\par
We have for any $s\geq 0$,
$\lVe P_s[f]\rVe_{\iy}\leq \lVe f\rVe_{\iy}$ (where $\lVe \cdot\rVe_{\iy}$ stands for the supremum norm on $\cB$)
and since $f\in\cR$, we  get
 everywhere 
\bq
\lim_{s\ri 0_+} P_s[f](X_t)&=&f(X_t)\eq
\par
Dominated convergence implies that
\bq
\lim_{s\ri 0_+} \EE_x[ P_s[f](X_t)]]&=&\EE_x[f(X_t)]\\
&=&P_t[f]\eq
as desired.\end{proof}
\par\sm
The \textbf{generator} $L$ associated to $P$ is the operator
\bq
L\st \cD(L)&\ri&\cR\eq
defined in the following way:
the space $\cD(L)$ is the set of functions $f\in\cR$
for which there exists a function $g\in\cR$ such that
the process $M^{f,g}\df(M^{f,g}_t)_{t\geq 0}$ defined by
\bq
\fo t\geq 0,\qquad M^{f,g}_t&\df& f(X_t)-f(X_0)-\int_0^t g(X_s)\, ds\eq
is a martingale under $\PP_x$, for all $x\in V$.
\par
Let us remark that $g$ is then uniquely determined. Indeed, we have for any $x\in V$ and $t\geq 0$,
\bq
\EE_x[f(X_t)]-\EE[f(X_0)]-\EE\lt[\int_0^t g(X_s)\, ds\rt]&=&0\eq
\par
Using Fubini's lemma (applicable due to our measurability requirement on $\Omega$)  and taking into account the definition of $P$, we get
\bq
P_t[f](x)-P_0[f](x)-\int_0^tP_s[g](x)\,ds&=&0\eq
namely, recalling that we required that $g\in\cR$,
\bqn{der}
\nonumber g&=&P_0[g]\\
\nonumber &=&\lim_{t\ri 0_+} \frac1t \int_0^tP_s[g](x)\,ds\\
&=&\lim_{t\ri 0_+} \frac{P_t[f](x)-f(x)}{t}\eqn
(we came back to the usual convention that $t>0$ in the above limit) and as a by-product, we are assured of the existence of the latter limit.
\par
We define $L[f]\df g$ and $M^f\df M^{f,g}$.\par
The differentiation property \eqref{der} can be extended into
\begin{lemma}
For any $f\in\cD(L)$, $x\in V$ and $t\geq 0$, we have
\bqn{lem2}
\pa_t P_t[f](x)&=&P_t[L[f]](x)\eqn
\end{lemma}
\begin{proof}
For any $f\in\cD(L)$, $x\in V$ and $t,s\geq 0$, we have
\bq
\EE_x\lt[ M_{t+s}^f-M_t^f\rt]&=&\EE_x\lt[\EE_x\lt[ M_{t+s}^f-M_t^f\vert\cF_t\rt]\rt]\\
&=&0\eq
\par
We compute that
\bq
M_{t+s}^f-M_t^f&=&f(X_{t+s})-f(X_t)-\int_t^{t+s} L[f](X_u)\, du\eq
so that
\bq
\EE_x\lt[ M_{t+s}^f-M_t^f\rt]&=&
P_{t+s}[f](x)-P_t[f](x)-\int_0^sP_{t+u}[L[f]](x)\, du\eq\par
Since $L[f]\in\cR$, the mapping $[0,s]\ni u\mapsto P_{t+u}[L[f]](x)$ is right continuous, according to Lemma \ref{lem1},
and the same argument as in \eqref{der} enables to conclude to \eqref{lem2}.\end{proof}
\par
We can now come to the main goal of this appendix:
\begin{prop}\label{pro1}
For any $t\geq 0$, $\cD(L)$ is stable by $P_t$ and on $\cD(L)$ we have $LP_t=P_tL$.
\end{prop}
\begin{proof}
Fix $f\in\cD(L)$ and $x\in V$, the assertion of the lemma amounts to checking that the process $N\df (N_s)_{s\geq 0}$ defined by
\bq
(N_s)_{s\geq 0}&\df&\lt(P_t[f](X_s)-P_t[f](X_0)-\int_0^s P_t[L[f]](X_u)\,du\rt)_{s\geq 0}\eq
is a martingale under $\PP_x$.
Consider $s'\geq s\geq 0$, we have to prove that
\bqn{zero}
\EE_x[N_{s'}-N_s\vert\cF_s]&=&0\eqn
\par
The l.h.s.\ is equal to
\bq
\lefteqn{\EE_x\lt[ P_t[f](X_{s'})-P_t[f](X_s)-\int_s^{s'} P_t[L[f]](X_u)\,du \Big\vert\cF_s\rt]}\\
&=&
\EE_x\lt[ P_t[f](X_{s'-s}\circ \Theta_s)-P_t[f](X_0\circ \Theta_s)-\int_0^{s'-s} P_t[L[f]](X_u\circ \Theta_s)\,du \Big\vert\cF_s\rt]\\
&=&\EE_y\lt[ P_t[f](X_{s'-s})-P_t[f](X_0)-\int_0^{s'-s} P_t[L[f]](X_u)\,du\rt]
\eq
where $y=X_s$. By Fubini's lemma, the previous r.h.s.\ can be written 
\bq
\lefteqn{\EE_y\lt[ P_t[f](X_{s'-s})\rt]-\EE_y\lt[ P_t[f](X_0)\rt]-\int_0^{s'-s} \EE_y[P_t[L[f]](X_u)]\,du}\\
&=&P_{t+s'-s}[f](y)-P_t[f](y)-\int_0^{s'-s} P_{t+u}[L[f]](y)\,du
\eq
\par
Taking into account \eqref{lem2}, the last integral is equal to
\bq
\int_0^{s'-s} \pa_u P_{t+u}[f](y)\,du&=&P_{t+s'-s}[f](y)-P_t[f](y)\eq
which ends the proof of \eqref{zero}.\end{proof}\par
\sm
The advantage of the above approach is that it is quite sable by optional stopping, as it is the case for martingales.
Let us succinctly give a simple example in the spirit of Section \ref{Section6}.
\par
Assume that in the above framework,  $V$ is a metric space, endowed with its Borelian measurable structure, and that $\Omega$ is the set of continuous trajectories $\cC(\RR_+,V)$. 
Furthermore, we suppose that $P$ is \textbf{Fellerian}, in the sense that it preserves $\cC_{\mathrm{b}}(V)$, the set of bounded and continuous real functions on $V$.
 \par
Let be given $A\subset V$ a closed set.
We consider $\tau$ the hitting time of $A$:
\bq
\tau&\df& \inf\{t\geq 0\st X_t\in A\}\ \in\ \RR_+\sqcup\{+\iy\}\eq
\par
Define the ``new'' process $\wi X\df(\wi X_t)_{t\geq 0}$ via
\bq
\fo t\geq 0,\qquad \wi X_t&\df& X_{t\wedge \tau}\eq
and for $x\in V$, let $\wi \PP_x$ be the image of $\PP_x$ by $\wi X$, it is still a probability measure on $\cC(\RR_+,V)$.
All notions corresponding to $\wi\PP\df(\wi \PP_x)_{x\in V}$, which is still a Markovian family,  receive a tilde.
It appears without difficulty that $\wi \cR$ is the set of functions $\wi f\in \cB$ such that there exists $f\in \cR$
with $\wi f$ coinciding with $f$ on $V\setminus A$.
The domain $\cD(\wi L)$ is the set of $\wi f\in\wi\cR$ such that there exists $f\in \cD(L)$
with $\wi f$ coinciding with $f$ on $V\setminus A$. In addition, we have
\bq
\fo x\in V,\qquad \wi L[\wi f](x)&=&\lt\{\begin{array}{ll}
L[f](x)&\hbox{, when $x\not\in A$}\\
0&\hbox{, when $x\in A$}\end{array}\rt.
\eq
This expression does not depend on the choice of $f$, due to the fact that $\PP$ is a diffusion, i.e.\ that $\Omega=\cC(\RR_+,V)$, which implies that
$L$ is a local operator (see for instance Theorem 7.29 of Schilling and Partzsch \cite{MR3234570},
they are working with Euclidean spaces, but the result can be extended to metric spaces).
\par
According to \eqref{lem2} and Proposition \ref{pro1},
we get 
\bq
\fo \wi f\in \cD(\wi L),\,\fo x\in V,\,\fo t\geq 0\qquad \pa_t\wi P_t[\wi f](x)&=&\wi P_t[\wi L[\wi f]](x)\ =\ \wi L[\wi P_t[\wi f]](x)\eq
\par
Such relations are not so obvious if we had chosen to work in a Banach setting (cf.\ e.g.\ the book of Yosida \cite{MR96a:46001}), 
considering for instance semi-groups acting on the space $\cC_{\mathrm{b}}(V)$  (endowed with the supremum norm),
since in general $\wi L$ would not naturally take values in $\cC_{\mathrm{b}}(V)$.

\section{An It\^o-Tanaka formula}\label{AppendixF}
\setcounter{equation}0

Let $M$ be a $d$-dimensional Riemannian manifold 
{and} $D\subset M$ a {compact and} connected domain with $C^2$ boundary $\partial D$, and $S$ be the regular skeleton of $D$, and $ \rho^{+}_{\partial D}$ the signed  distance to $\partial D$, which is positive inside $D$ and negative outside $D$. The notations will be the same as in Appendix \ref{Section2}.

\begin{prop}\label{PropF1}
Let $X_t$ a Brownian motion in $M$. We have the following It\^o-Tanaka formula :
$$d \rho^{+}_{\partial D} (X_t) = \langle N^D(X_t), dX_t \rangle -\frac{1}{2} h^D(X_t) dt -  \sin\left(\theta^{S}(X_t)\right)\,dL_t^{S}(X),$$
in the above formula, $N^D(x) = \n  \rho^{+}_{\partial D} (x) $ and 
$ - h^D(x) = \D \rho^{+}_{\partial D} (x)$ for $x \notin S $ , and define to be $ 0$ elsewhere, $L_t^S(X)$ is the local time defined  as in \eqref{4.2}. 
\end{prop}
\begin{proof}
The formula is a consequence of the It\^o formula outside the skeleton.
Since the non regular part of the skeleton has   Hausdorff dimension smaller than or equal to $d-2$, it is not visited by the Brownian motion. So we only focus on the regular skeleton. For all $x \in S$, the distance to the boundary is  the minimum of two $C^2$ functions $ f,g$ defined on some neighborhood $U$ of $ x$ in  $M$. The function $f$ (resp. $g$ ) is the distance function to a piece of $\partial D $ containing $y_1 $ (resp. $ y_2$) as in \eqref{2.2}. 
We have locally, 
$$ \rho^{+}_{\partial D}  = f \wedge g = \frac12 (f  +g ) -\frac12 \vert f-g \vert .$$
Using It\^o formula and Tanaka formula we have
\begin{align*}
d\rho^{+}_{\partial D}(X_t) &= \frac12  \Big(\frac12 \D (f+g)(X_t) dt +\langle \n (f+g) (X_t), dX_t \rangle \Big)\\ 
&- \frac12 \Big(\sign ((f-g)(X_t))d((f-g)(X_t)) + dL^{0,+}_t((f-g)(X_.))\Big),\\
\end{align*}
where $L^{0,+}_t((f-g)(X_.))= \lim_{\e \to 0^+} \frac{1}{\e}\int_0^t \un_{[0,\e]}  ((f-g)(X_s)) d\langle  (f-g)(X),(f-g)(X)  \rangle_s $.
Since locally $S = \{f-g =0 \}$ and $ \mu(S)=0$, we have
\begin{align*}
d\rho^{+}_{\partial D}(X_t) &= \frac12 \un_{X_t \notin S} \D \rho^{+}_{\partial D}(X_t) dt + \un_{X_t \notin S} \langle \n\rho^{+}_{\partial D}(X_t), dX_t \rangle -\frac12 dL^{0,+}_t((f-g)(X_.)) .\\
\end{align*}
After changing the role of  $ f$ and $g$ we get 
\begin{equation}\label{F1}
d\rho^{+}_{\partial D}(X_t) = \frac12 \un_{X_t \notin S} \D \rho^{+}_{\partial D}(X_t) dt + \un_{X_t \notin S} \langle \n\rho^{+}_{\partial D}(X_t), dX_t \rangle -\frac12 dL^{0}_t((f-g)(X_.)) ,
\end{equation}
where $$L^{0}_t((f-g)(X_.))= \lim_{\e \to 0^+} \int_0^t\frac{1}{2\e} \un_{[-\e,\e]}  ((f-g)(X_s)) \Vert \n  (f-g)\Vert^2 (X_s)\, ds . $$
In Appendix \ref{Section2} it is shown that for  $x \in S$ , $\Vert \n (f-g)(x) \Vert = 2 \sin\left(\theta^{S}(x)\right) $.

Using  the flow $ \frac{d}{dt} \gamma(t) = - \frac{\n (f-g)(\gamma(t))}{\Vert \n (f-g)(\gamma(t))\Vert^2} $ that starts at $y\in U$, we get 
$$ \{y \in M, \text { s.t. } \vert f-g \vert(y) \le \e  \} \subset   \{y \in M, \text { s.t. } \vert d_S(y) \vert \le \frac{\e}{2\sin\left(\theta^{S}(\gamma(g(y)))\right)} + o(\e)\},$$ where $d_{S}$ is the distance to $ S$.
On the other hand, using the minimal geodesic from $S$ to $y \in U$
we get 
$$ \{y \in M, \text { s.t. } \vert d_S(y) \vert  \le \e  \} \subset   \{y \in M, \text { s.t. } \vert f-g \vert(y) \le 2 \e \sin\left(\theta^{S}(P^S(y))\right) + o(\e)\}.$$
Hence $$ dL^{0}_t((f-g)(X_.)) = 2\sin\left(\theta^{S}(X_t)\right) L_t^S(X_.). $$
Together with   \eqref{F1}, this yield the Proposition. 
\end{proof}

\par
\section{Uniqueness in law of $\tilde{\mathcal{L}}$ diffusion}\label{AppendixG}
\setcounter{equation}0

Let us consider the following  generator $\widehat{\SL}$ of a stochastic modified mean curvature flow. The action of this generator and its carr\'e du champs on elementary observables are defined  as follows.
For any smooth function $k$ on $M$, consider the mapping $F_k$ on $\cD^{2+\alpha}$ defined  by
\bq
\fo D\in \cD^{2+\alpha},\qquad F_k(D)&\df& \int_D k\, d\mu\eq
For any $k,g\in \cC^\iy(M)$ and any $D\in \cD^{2+\alpha}$, 
\bqn{G.6.7}\left\{\begin{array}{rcl}
 \widehat{\SL}[F_k](D)&\df&  -\f12 \usm^{\partial D}(\langle \n k, N^D\rangle )= F_{\frac12 \D k} (D)\\
\Gamma_{\widehat{\SL}}[F_k, F_g](D)&\df& \int_{\pa D} k\, d\usm\int_{\pa D} g\, d\usm . \\
\end{array}\right.
\eqn\par
Note that $\widehat{\SL} $ has the same carr\'e du champs as the carr\'e du champs associated to  $\wi\SL $.
From now the generator  $\widehat{\SL} $ is defined as in \eqref{fLfF}.
\begin{prop}\label{prop-G1} The martingale problem associated $\widehat{\SL}$ is well-posed. 
\end{prop}

\begin{proof}

We have already shown the existence result in  \cite{zbMATH07470497}, so it remains  to prove the uniqueness in law.
Let us  first consider the two-dimensional Euclidean case, namely $M = \mathbb{R}^2 $. For all $ \lambda \in \mathbb{R}$	 and for any  function $k_{\lambda} \in {\rm vect} (  e^{\lambda x},e^{\lambda y} ) $  we have $ \frac12 \D k_{\lambda}(x,y)= \frac{\lambda^2}{2}k_{\lambda}(x,y)$. Let  $f_{\lambda} ((x,y),D) := k_{\lambda}(x,y) F_{k_{\lambda}}(D) $, for  $(x,y) \in \mathbb{R}^2$ and $ D \in \cD^{2+\alpha} $. This function satisfies the following property:
\begin{align*}
\widehat{\SL} f_{\lambda}((x,y),D) & =  k_{\lambda}(x,y)  \widehat{\SL} F_{k_{\lambda}}(D) \\
&=  k_{\lambda}(x,y)   F_{\frac12 \D k_{\lambda}}(D)\\
&= k_{\lambda}(x,y)    F_{\frac{\lambda^2}{2}  k_{\lambda}}(D) \\
&= \frac{\lambda^2}{2}  k_{\lambda}(x,y)     F_{ k_{\lambda}}(D) \\
&= \frac12 \D k_{\lambda}( x,y)  F_{ k_{\lambda}}(D).\\
&= \frac12 \D f_{\lambda}((x,y),D)\\
\end{align*}
Let $ (X_t)_{t \ge 0}$ be a $\mathbb{R}^2$-valued Brownian motion that starts at $X_0= (x_1,x_2) \in \mathbb{R}^2$ and $(\hat{D}_t)_{t \ge 0}$ a $\widehat{\SL} $ diffusion that starts at $D_0$ independent of $(X_t)_{t \ge 0}$.
{Even if we stop the diffusion, we can assume that its lifetime is infinite and we add indicators as described  in Appendix \ref{appD}.}
For all $ 0 \le s \le t $ , we have 
  $$ df_{\lambda}(X_{t-s}, \hat{D}_s) \overset{m} {=} - \frac{1}{2} \D f_{\lambda}(X_{t-s},\hat{ D}_s)ds +  \widehat{\SL} f_{\lambda} (X_{t-s},\hat{ D}_s)ds \overset{m} {=} 0.$$ 
  Hence for all $\lambda \in \mathbb{R} $ we have
  \begin{equation}\label{passe-passe}
   \E[f_{\lambda}(X_{t}, D_0) ] = \E[f_{\lambda}(X_{0}, \hat{D}_t) ].
   \end{equation}
Since the left hand side of the above equation does not depend on the $\widehat{\SL} $ diffusion, we get that for any   $\widehat{\SL} $ diffusion $(\tilde{D}_t)_{t \ge 0}$  that starts at $D_0$ :
 
  $$  \E[f_{\lambda}(X_{0}, \hat{D}_t) ] = \E[f_{\lambda}(X_{0}, \tilde{D}_t) ] ,$$ 
  and so  $$\E[ F_{k_{\lambda}}(D_{t })] = \E[ F_{k_{\lambda}}( \tilde{D}_{t }))] .$$
  
In order to apply Theorem 4.2 of \cite{MR838085}, we have to show that the above equation characterizes the law of  the one-dimensional distribution,  i.e. we have to show that $ ( F_{k_{\lambda}}) $ is separating in the space of probability measures on $\cD^{2+\a} $. This is equivalent to separate domains. 
Let $ A, B \in \cD^{2+\a}$ such that $  F_{k_{\lambda}}(A)=  F_{k_{\lambda}}(B) $ for all $\lambda \in \mathbb{R} $ and $k_{\lambda} \in \langle  e^{\lambda x},e^{\lambda y} \rangle $, we have for all $ \lambda$:
$$ \int_A k_{ \lambda}(x,y) d\mu = \int_B k_{ \lambda}(x,y) d\mu.$$
After  successive derivations in $ \lambda$ and evaluation at $\lambda =0 $, we get for all $n \in \mathbb{N}$
$$ \int_A x^n d\mu = \int_B x^n d\mu  ,$$
$$ \int_A y^n d\mu = \int_B y^n d\mu  ,$$
  
The above computations could be done also for $\tilde{k}_{\lambda_1,\lambda_2} =  e^{\lambda_1 x + \lambda_2 y},$ since $\frac12 \D \tilde{k}_{\lambda_1,\lambda_2}  = \frac{\lambda_1^2 + \lambda_2^2}{2}  \tilde{k}_{\lambda_1,\lambda_2}, $ and after derivations in $\lambda_1,\lambda_2 $ and evaluating at $(0,0)$ we get that for all $n,m \in \mathbb{N}$:
$$ \int_A x^ny^m d\mu = \int_B x^ny^m d\mu  ,$$
hence, using the boundary regularity, we get $A= B$.\par 
{We could also apply Stone-Weierstrass' theorem to the function algebra generated by the mappings $(x,y)\mapsto e^{\lambda_1 x}$ and $(x,y)\mapsto e^{ \lambda_2 y} $}.\par
The proof is the same for all Euclidean spaces.\par
If $M$ is a compact manifold 
let 
$$f_{\lambda_i} (X,D) := k_{\lambda_i}(X) F_{k_{\lambda_i}}(D), $$ where 
$\lambda_i $ is an 
 eigenvalue of $\frac12 \D $ and $k_i$ is the associated eigenfunction (respectively the Neumann eigenvalue).
By the same computation as above  \eqref{passe-passe} is also valid for the boundary reflecting Brownian motion), to get the conclusion we have  to show that $ (F_{k_{\lambda_i}})_i $ separates domains.
Since $(k_{\lambda_i})_i$ is an orthonormal  basis of $ L^2(\mu)   $
we get that if $ A,B \in \cD^{2+\a}$ be such that for all $i$,
 $$  F_{k_{\lambda_i}}(A) = F_{k_{\lambda_i}}(B)$$ i.e $\langle \un_A ,k_{\lambda_i}  \rangle_{L^2} = \langle \un_B ,k_{\lambda_i}  \rangle_{L^2} $, then $\un_A \overset{L^2}{ = } \un_B $ 
 hence $A=B $.
 
 For the  complete manifold $M$, let $ \Omega_k$ be an exhaustion of $M$ with a  regular boundary  such that  $ D_0 \subset \Omega_k$, and stop the $\widehat{\SL} $ diffusion  when it hit $\Omega_k^c$ and use the above result for the manifold with boundary $\Omega_k $, we get  the result by localization. 
 
\end{proof}

\begin{prop}\label{prop-G2}
The martingale problem associated to $\SL$ is well-posed. 
\end{prop}
\begin{proof}
Let $D_t$ be a $\SL $ diffusion that starts at $D_0$, defined on $(\Omega,\mathcal{F}^D ,\mathbb{Q})$. We first recall  that there exist an enlargement of the  probability space such that it carries a one dimensional Brownian motion $ B$ such that for  all $ k \in C^{\infty}(M) $
\begin{equation}\label{eq-G2}
F_k(D_t)=F_k(D_0)+\int_0^t \SL[F_k](D_s)\,ds +\int_0^t\sqrt{\Gamma_{\SL}[F_k,F_k]}(D_s)\, dB_s
\end{equation}
where $\sqrt{\Gamma_{\SL}[F_k,F_k]}(D) := \int_{\partial D } k\, d\sigma $, this is actually Proposition 53 in \cite{zbMATH07470497}. Note that this procedure of enlargement (Theorem 1.7 chapter V in \cite{MR1725357}) could be done by gluing the same independent Brownian motion for each   $(\Omega,\mathcal{F}^D ,\mathbb{Q})$. We denote by $ (\tilde{\Omega},\tilde{\mathcal{F}}^D ,\tilde{\mathbb{Q})}$ the enlarged probability space. Since $\SL $ is an $h$-transform of $\widehat{\SL} $ namely
$$\SL[F_k] = \widehat{\SL}[F_k] +\frac{\Gamma_{\widehat{\SL}}( F_1,F_k)}{F_1},$$
equation \eqref{eq-G2} becomes in  a differential form
\begin{equation}\label{eq-G3}
dF_k(D_t)- \widehat{\SL}[F_k](D_t) dt = (\int_{\partial D } k\, d\sigma)\big(dB_t + \frac{\usm^{\partial D_t}(\partial D_t)}{\mu(D_t)}dt\big). 
\end{equation}
 Let  
$$ M_t = e^{-\int_0^t \langle \frac{\usm^{\partial D_s}(\partial D_s)}{\mu(D_s)}, \, dB_s \rangle -\frac12 \int_0^t\big(\frac{\usm^{\partial D_s}(\partial D_s)}{\mu(D_s)} \big)^2 \,ds }    , $$
$$ \mathbb{P}_{\vert \mathcal{F}_t} = M_t\tilde{\mathbb{Q}}_{\vert \mathcal{F}_t} . $$ 
Using Girsanov transform, $D_t $ is solution of the $\widehat{\SL} $ martingale problem  on the probability space $ (\tilde{\Omega},\tilde{\mathcal{F}}^D ,\mathbb{P})$. Since $ \tilde{\mathbb{Q}} = M^{-1}\mathbb{P}$ we get the uniqueness in law of the $\SL $ diffusion by Proposition \ref{prop-G1}. 
\end{proof}
\section{Convergence in law: a key lemma}\label{appH}
\setcounter{equation}0

This Appendix is devoted to the adaptation to some domain-valued sequences of processes, of Lemma~4 in~\cite{Zheng:85}, which states stability of some time integrals under convergence in law. 
\begin{lemma}
 \label{Zheng}
 Let $\tilde\SF:=\tilde \SF^{\a,\e}$. We endow the set of continuous paths $\di \SC\left([0,\infty), M\times \tilde\SF\right)$ with the two dissimilarity measures $d_\b$,  $\b\in\{0,\a\}$, defined as:
\begin{equation}
 \label{H1}
 d_\b\left((x^1,D^1), (x^2,D^2)\right)=\sup_{t\ge 0}\rho(x^1(t),x^2(t))+\sup_{t\ge 0}  d_{\b,\tilde\SF} (D^1(t), D^2(t)),
\end{equation}
where for two domains $D$ and $D'$
\begin{equation}\label{H2}
\begin{split}
&d_{\b,\tilde\SF}(D,D')=\left\{
\begin{array}{cc}
d_{\b,D}(D,D') \wedge d_{\b,D'}(D',D)
\wedge\e&\hbox{if } \  H(D,D')<\e\\
\e &  \hbox{ otherwise.}\hfill
\end{array}
\right.
\end{split}
\end{equation}
Here $H(D,D')$ is the Hausdorff distance between $D$ and $D'$ and the distance $d_{\b,D}$ is defined  in~\eqref{dD}. 

 Let $\di (X_t^n,D_t^n,\tau_\e^n)_{t\ge 0}\df(X_t^{\d_n},D_t^{\d_n},\tau_\e^{\d_n})_{t\ge 0}$ a subsequence of~\eqref{4.4.5} converging in law to the limit defined in~\eqref{4.4.6} for the product of $d_\a$ and the Euclidean distance in $\RR_+$. 
 
 Let $f_n\st(x,D)\mapsto f_n(x,D)$ and $f\st(x,D)\mapsto f(x,D)$ be maps on $M\times \tilde\SF$ with values in some Euclidean space, and $U$ an open set in $M\times \tilde\SF$ for~$d_0$. Assume that:
 \begin{itemize}
  \item[(i)] the random variables $\di \int_0^\infty |f_n(X_s^n,D_s^n)|^p\, ds $ are uniformly bounded in probability for some $p>1$,
  \item[(ii)] in the open set $U$, the functions $f_n$ converge locally uniformly to $f$ with respect to $d_0$,  and are $d_0$-continuous,
  \item[(iii)] for a.e. $t\geq 0$, $(X_t,D_t)\in U$.
 \end{itemize}
 Then $\di \left(X_t^n, D_t^n,\int_0^t f_n(X_s^n, D_s^n)\, ds\right)_{t\ge 0}$ converges  in law to $\di \left(X_t, D_t,\int_0^t f(X_s, D_s)\, ds\right)_{t\ge 0}$ for $(d_\a,|\cdot|)$.
\end{lemma}
\begin{remark}
 \label{HR1}
 In the applications we will always take 
 \begin{equation}
  \label{H3}
  U=\left\{(x,D)\in M\times \tilde\SF, \ x\in D\backslash S(D)\right\},
 \end{equation}
 which is easily seen to be $d_0$-open thanks to Assumption~\ref{A4.0} on $\tilde\SF$.
\end{remark}
\begin{proof}
We will follow the proof of Lemma~4 in~\cite{Zheng:85}, but with several differences due to infinite dimensional spaces. Set for $n\in \NN$, $t\ge 0$, 
\begin{equation}
 \label{H4}
 A_t^n:=\int_0^tf_n(X_s^n, D_s^n)\, ds, \quad A_t:=\int_0^tf(X_s, D_s)\, ds.
\end{equation}
Condition (i) implies that the processes $A^n$ are tight.
To get the conclusion il is sufficient to show that all the converging subsequences have the same limit. So  assume that
\begin{equation}
\label{H5}
\left( X_t^n, D_t^n, A_t^n\right)_{t\ge 0} \overset{\SL}{\longrightarrow} \left( X_t, D_t, a_t\right)_{t\ge 0} .
\end{equation}
and let us prove that $\di \left( a_t\right)_{t\ge 0}= \left( A_t\right)_{t\ge 0}$.
By Skorohod theorem we may realize all processes 
\begin{equation}
\label{H6}
\left( X_t^n, D_t^n, A_t^n,  X_t, D_t, a_t\right)_{t\ge 0}
\end{equation}
on the same probability space $(\Om,\SF, \PP)$ in such a way that 
\begin{equation}
\label{H7}
(Z_t^n)_{t\ge 0}:=\left( X_t^n, D_t^n, A_t^n\right)_{t\ge 0} \overset{\hbox{a.s.}}{\longrightarrow} \left( X_t, D_t, a_t\right)_{t\ge 0}=:(Z_t)_{t\ge 0} .
\end{equation}
This means that $Z_t^n\to Z_t$ a.s. uniformly in $t\ge 0$. 

Fix $\om\in \Om$. Let $t>0$ be such that $(X_t(\om), D_t(\om))\in U$. For some $\e'>0$ we have $(X_s(\om), D_s(\om))\in U$ for all $s\in [t-\e', t+\e']$. The set 
\begin{equation}
\label{H8}
S:=\left\{\left(X_s(\om), D_s(\om)\right), \quad s\in [t-\e', t+\e'] \right\}
\end{equation}
is $d_\a$-compact in $M\times \tilde \SF$, so it has a $d_\a$-neighbourhood $V$ included in $U$ of the form 
\begin{equation}
\label{H9}
V=\left\{ (x,D)\in M\times \tilde\SF, \quad d_\a\left( (x,D),S\right)\le \e''    \right\}.
\end{equation}
for some small enough $ \e''>0$.
For $n$ sufficiently large, $\di \left(X_s^n(\om), D_s^n(\om)\right)\in V$ for all $s\in [t-\e', t+\e']$. 
On the other hand $V$ is bounded for the  distance~$d_\a$. This implies by Arzela-Ascoli theorem that it is compact for the  distance~$d_0$. 
We have the two following facts, the first one being an assumption on the $f_n$ and $f$, the second one being a consequence of the $d_0$-compactness of $V$
\begin{itemize}
\item[(a)] $f_n\to f$ as $n\to\infty$ uniformly in $(V,d_0)$;
\item[(b)] $f$ is uniformly continuous in $(V,d_0)$.
\end{itemize}
Then
\begin{align*}
 &\sup_{s\in[t-\e,t+\e]}\left|f_n(X_s^n(\om),D_s^n(\om))-f(X_s(\om),D_s(\om))\right|\\
 \le & \sup_{s\in[t-\e,t+\e]}\left|f_n(X_s^n(\om),D_s^n(\om))-f(X_s^n(\om),D_s^n(\om))\right|\\&+\sup_{s\in[t-\e,t+\e]}\left|f(X_s^n(\om),D_s^n(\om))-f(X_s(\om),D_s(\om))\right|.
\end{align*}
Both terms in the right converge to~$0$, the first one by~(a) and the second one by~(b).
So we have by~\eqref{H7} and the above calculation
\begin{equation}
 \label{H10}
 \left\{
 \begin{array}{cc}
  \hfill\left(A_s^n(\om)\right)_{s\in[t-\e,t+\e]}&\to \left(a_s(\om)\right)_{s\in[t-\e,t+\e]}\hfill\\
  \left((A_s^n(\om))'=f_n(X_s^n(\om), D_s^n(\om))\right)_{s\in[t-\e,t+\e]}&\to \left(f(X_s(\om), D_s(\om))\right)_{s\in[t-\e,t+\e]}
 \end{array}
 \right.
\end{equation}
both uniformly in $s\in[t-\e,t+\e]$. This implies that $a_s(\om)$ is differentiable in $(t-\e, t+\e)$ with derivative $f(X_s(\om), D_s(\om))$ and in particular at~$t$.

We have that for all $t\ge 0$, $(X_t(\om), D_t(\om))\in U$ a.s.. So for all $t\ge 0$, 
\begin{equation}
\label{H11}
\f{d}{dt}a_t(\om)=f(X_t(\om), D_t(\om))\quad\hbox{a.s.}.
\end{equation}
This implies that $\om$ a.s.
\begin{equation}
\label{H12}
\f{d}{dt}a_t(\om)=f(X_t(\om), D_t(\om))\quad\hbox{for a.e. $t$}.
\end{equation}
On the other hand  we know by~\cite{Meyer-Zheng:84} Theorem~10 that $(a_t)_{t\ge 0}$ is absolutely continuous : 
\begin{equation}
\label{H12bis}
 a_t(\om)=\int_0^t \ell_s(\om)\, ds.
\end{equation}
By Lebesgue theorem, $\om$ a.s.,  for a.e. $t\ge 0$ 
\begin{equation}
\label{H13}
\lim_{\e \searrow 0} \f1{2\e}\int_{t-\e}^{t+\e}|\ell_s(\om)-\ell_t(\om)|\, ds=0.
\end{equation}
Equalities~\eqref{H12} and~\eqref{H12bis} imply that  $\om$ a.s.
\begin{equation}
\label{H14}
\lim_{\e \searrow 0} \f1{2\e}\int_{t-\e}^{t+\e}\ell_s(\om)\, ds=f(X_t(\om), D_t(\om))\quad\hbox{for a.e. $ t$}. 
\end{equation}
On the other hand 
\begin{align*}
&\left|\f1{2\e}  \int_{t-\e}^{t+\e}\ell_s(\om)-\ell_t(\om)\, ds \right|\le \f1{2\e}\int_{t-\e}^{t+\e}|\ell_s(\om)-\ell_t(\om)|\, ds
\end{align*}
so~\eqref{H13} implies that $\om$ a.s. for a.e. $t\ge 0$
\begin{equation}
\label{H15}
\lim_{\e \searrow 0} \f1{2\e}\int_{t-\e}^{t+\e}\ell_s(\om)\, ds=\ell_t(\om).
\end{equation}
Consequently, using~\eqref{H12} and~\eqref{H15}, we get  $\om$ a.s. for a.e. $t\ge 0$
\begin{equation}
\label{H16}
\ell_t(\om)=f(X_t(\om), D_t(\om))
\end{equation}
Integrating we get $\om$-a.s. for all $t\ge 0$
\begin{equation}
\label{H17}
a_t(\om)=A_t(\om)=\int_0^t f(X_s(\om), D_s(\om))\, ds.
\end{equation}
This together with~\eqref{H4} proves the lemma.
\end{proof}

\providecommand{\bysame}{\leavevmode\hbox to3em{\hrulefill}\thinspace}


\begin{thebibliography}{10}

\bibitem{Albano:16}
Paolo Albano
\newblock {On the stability of the cutlocus}.
\newblock {\em Nonlinear Analysis}, 136 (2016) 51-61.

\bibitem{ACM:21}
\textsc{Marc Arnaudon, Kol\'eh\`e Coulibaly and Laurent Miclo},
\emph{Intertwining Brownian motions with symmetric convex sets}, in preparation
\bibitem{ACM3:22}
\textsc{Marc Arnaudon, Kol\'eh\`e Coulibaly and Laurent Miclo},
\emph{On the separation cut-off phenomenon for Brownian motions on high dimensional spheres}, in preparation


\bibitem{Arnaudon-Li:17}
\textsc{Marc Arnaudon, Xue-Mei Li},
\emph{Reflected Brownian motion: selection, approximation and linearization}, Electronic Journal of Probability 22 (2017), no. 31, 1-55.

\bibitem{MR3155209}
\textsc{Dominique Bakry, Ivan Gentil, and Michel Ledoux}.
\emph{ Analysis and geometry of {M}arkov diffusion operators}, volume
  348 of { Grundlehren der Mathematischen Wissenschaften [Fundamental
  Principles of Mathematical Sciences]}.
 Springer, Cham, 2014.
 
 \bibitem{MR1654531}
Philippe Carmona, Fr{\'e}d{\'e}rique Petit, and Marc Yor.
\newblock Beta-gamma random variables and intertwining relations between
  certain {M}arkov processes.
\newblock {\em Rev. Mat. Iberoamericana}, 14(2):311--367, 1998.

\bibitem{Chavel:93}
\textsc{Isaac Chavel}, \emph{Riemannian geometry: a modern introduction}, Cambridge University Press, 1993

\bibitem{Cheeger-Ebin:75}
\textsc{ Jeff Cheeger and David Ebin}, \emph{Comparison Theorems in Riemannian Geometry}, North Holland, 1975

\bibitem{zbMATH07470497}
Kol\'eh\`e {Coulibaly-Pasquier} and Laurent {Miclo}.
\newblock {On the evolution by duality of domains on manifolds}.
\newblock {\em {M\'em. Soc. Math. Fr., Nouv. S\'er.}}, 171:1--110, 2021.

\bibitem{MR1071805}
\textsc{Persi Diaconis and James~Allen Fill}.
\emph{ Strong stationary times via a new form of duality},
 Ann. Probab., 18(4):1483--1522, 1990.

\bibitem{MR838085}
\textsc{Stewart~N. Ethier and Thomas~G. Kurtz}.
\emph{ Markov processes}.
 Wiley Series in Probability and Mathematical Statistics: Probability
  and Mathematical Statistics. John Wiley \& Sons Inc., New York, 1986.
 Characterization and convergence.

\bibitem{MR3571247}
James~Allen Fill and Vince Lyzinski.
\newblock Strong stationary duality for diffusion processes.
\newblock {\em J. Theoret. Probab.}, 29(4):1298--1338, 2016.

   \bibitem{Ikeda-Watanabe}
\textsc{Nobuyuki Ikeda and Shinzo Watanabe}, \emph{Stochastic differential equations and diffusion processes}, second edition,
North Holland Mathematical Library, 24, 1989.

\bibitem{MR1167270}
\textsc{Jean-Pierre Imhof}.
\emph{A simple proof of {P}itman's {$2M-X$} theorem},
 Adv. in Appl. Probab., 24(2):499--501, 1992.

\bibitem{2019arXiv190807559M}
\textsc{Motoya {Machida}}.
\emph{$\Lambda$-linked coupling for drifting Brownian motions}.
 arXiv e-prints, 1908.07559, Aug 2019.

\bibitem{Meyer-Zheng:84}
\textsc{Paul-Andr\'e Meyer and Wei An Zheng}. \emph{Tightness criteria for laws of semimartingales}, Ann. Inst. Henri Poincar\'e, Vol. 20, n. 4, 1984, p. 353-372.

\bibitem{MR3634282}
\textsc{Laurent Miclo}.
\emph{ Strong stationary times for one-dimensional diffusions},
 Ann. Inst. Henri Poincar{\'e} Probab. Stat., 53(2):957--996,
  2017.

\bibitem{Mic2020}
Laurent Miclo.
\newblock On the construction of measure-valued dual processes.
\newblock {\em Electron. J. Probab.}, 25:1--64, 2020.


\bibitem{NO:10}
\textsc{Kaj  Nystr\"om and Thomas \"Onskog},
\emph{ The Skorohod oblique reflection problem in time-dependent domains},
 Ann. Prob, Volume 38, Number 6 (2010), 2170--2223.

\bibitem{2013arXiv1306.0857P}
Soumik {Pal} and Mykhaylo {Shkolnikov}.
\newblock {Intertwining diffusions and wave equations}.
\newblock {\em ArXiv e-prints}, June 2013.

\bibitem{MR0375485}
\textsc{Jim~W. Pitman}.
\emph{ One-dimensional {B}rownian motion and the three-dimensional {B}essel
  process},
 Advances in Appl. Probability, 7(3):511--526, 1975.

\bibitem{MR1725357}
\textsc{Daniel Revuz and Marc Yor}.
\emph{ Continuous martingales and {B}rownian motion}, volume 293 of
  Grundlehren der Mathematischen Wissenschaften [Fundamental Principles of
  Mathematical Sciences], 
Springer-Verlag, Berlin, third edition, 1999.

\bibitem{MR624684}
L.~Chris~G. Rogers and Jim~W. Pitman.
\newblock Markov functions.
\newblock {\em Ann. Probab.}, 9(4):573--582, 1981.

\bibitem{MR3234570}
\textsc{Ren\'{e}~L. Schilling and Lothar Partzsch}.
\emph{ Brownian motion}.
 De Gruyter Graduate. De Gruyter, Berlin, second edition, 2014.
 An introduction to stochastic processes, With a chapter on simulation
  by Bj\"{o}rn B\"{o}ttcher.

\bibitem{MR2190038}
\textsc{Daniel~W. Stroock and S.~R.~Srinivasa Varadhan}.
\emph{ Multidimensional diffusion processes}.
 Classics in Mathematics. Springer-Verlag, Berlin, 2006.
 Reprint of the 1997 edition.

\bibitem{Yor_intertwinings}
Marc Yor.
\newblock Intertwinings of {Bessel} processes.
\newblock Technical Report No. 174, Department of Statistics, University of
  California, Berkeley, California, October 1988.

\bibitem{MR96a:46001}
\textsc{K{\=o}saku Yosida}.
\emph{ Functional analysis}.
 Classics in Mathematics. Springer-Verlag, Berlin, 1995.
 Reprint of the sixth (1980) edition.

\bibitem{Zheng:85}
\textsc{Wei An  Zheng}.
\emph{Tightness results for laws of diffusion processes; application to stochastic mechanics}.
 Annales de l'I. H. P., section B, tome 21, n. 2 (1985), p. 103--124.

\end{thebibliography}
\end{document}